\documentclass[a4paper,10pt]{amsart}
\usepackage{amsmath,amssymb,hyperref,multirow,breakurl}

\theoremstyle{plain}
\newtheorem{theorem}{\bf Theorem}[section]
\newtheorem{proposition}[theorem]{\bf Proposition}
\newtheorem{lemma}[theorem]{\bf Lemma}

\theoremstyle{definition}
\newtheorem{example}[theorem]{\bf Example}

\newtheorem{remark}[theorem]{\bf Remark}

\hypersetup{hypertexnames=false}
\numberwithin{equation}{section}

\begin{document}

\title{On monoids of ideals of orders in quadratic number fields}

\author{Johannes Brantner, Alfred Geroldinger,}
\address{Institut f\"ur Mathematik und wissenschaftliches Rechnen, Karl-Franzens-Universit\"at Graz, NAWI Graz, Heinrichstra{\ss}e 36, 8010 Graz, Austria}
\email{j.brantner@gmx.at, alfred.geroldinger@uni-graz.at}
\urladdr{https://imsc.uni-graz.at/geroldinger}

\author{Andreas Reinhart}
\address{Department of Mathematical Sciences, New Mexico State University, Las Cruces, NM 88003, United States}
\email{andreas.reinhart@uni-graz.at}



\thanks{This work was supported by the Austrian Science Fund FWF, Project Numbers J4023-N35 and P28864-N35}

\begin{abstract}
We determine the set of catenary degrees, the set of distances, and the unions of sets of lengths of the monoid of nonzero ideals and of the monoid of invertible ideals of orders in quadratic number fields.
\end{abstract}

\maketitle
\section{Introduction}\label{1}
\smallskip

Factorization theory for Mori domains and their semigroups of ideals splits into two cases. The first and best understood case is that of Krull domains (i.e., of completely integrally closed Mori domains). The arithmetic of a Krull domain depends only on the class group and on the distribution of prime divisors in the classes, and it can be studied -- at least to a large extent -- with methods from additive combinatorics. The link to additive combinatorics is most powerful when the Krull domain has a finite class group and when each class contains at least one prime divisor (this holds true, among others, for rings of integers in number fields). Then sets of lengths, sets of distances and of catenary degrees of the domain can be studied in terms of zero-sum problems over the class group. Moreover, we obtain a variety of explicit results for arithmetical invariants in terms of classical combinatorial invariants (such as the Davenport constant of the class group) or even in terms of the group invariants of the class group. We refer to \cite{Ge09a} for a description of the link to additive combinatorics and to the recent survey \cite{Sc16a} discussing explicit results for arithmetical invariants.

Let us consider Mori domains that are not completely integrally closed but have a nonzero conductor towards their complete integral closure. The best investigated classes of such domains are weakly Krull Mori domains with finite $v$-class group and C-domains. For them there is a variety of abstract arithmetical finiteness results but in general there are no precise results. For example, it is well-known that sets of distances and of catenary degrees are finite but there are no reasonable bounds for their size. The simplest not completely integrally closed Mori domains are orders in number fields. They are one-dimensional noetherian with nonzero conductor, finite Picard group, and all factor rings modulo nonzero ideals are finite. Thus they are weakly Krull domains and C-domains. Although there is recent progress for seminormal orders, for general orders in number fields there is no characterization of half-factoriality (for progress in the local case see \cite{Ka05b}) and there is no information on the structure of their sets of distances or catenary degrees (neither for orders nor for their monoids of ideals).

In the present paper we focus on monoids of ideals of orders in quadratic number fields and establish precise results for their set of distances $\Delta (\cdot)$ and their set of catenary degrees ${\rm Ca}(\cdot)$. Orders in quadratic number fields are intimately related with quadratic irrationals, continued fractions, and binary quadratic forms and all these areas provide a wealth of number theoretic tools for the investigation of orders. We refer to \cite{HK13a} for a modern presentation of these connections and to \cite{Co-Ma-Ok17a,Pe-Za16a} for recent progress on the arithmetic and ideal theoretic structure of quadratic orders.

Let $\mathcal{O}$ be an order in a quadratic number field, $\mathcal{I}^*(\mathcal{O})$ be the monoid of invertible ideals, and $\mathcal{I}(\mathcal{O})$ be the monoid of nonzero ideals (note that $\mathcal{I}(\mathcal{O})$ is not cancellative if $\mathcal{O}$ is not maximal). Since $\mathcal{I}^*(\mathcal{O})$ is a divisor-closed submonoid of $\mathcal{I}(\mathcal{O})$, the set of catenary degrees and the set of distances of $\mathcal{I}^*(\mathcal{O})$ are contained in the respective sets of $\mathcal{I}(\mathcal{O})$.
We formulate a main result of this paper and then we compare it with related results in the literature.

\smallskip
\begin{theorem}\label{theorem 1.1}
Let $\mathcal{O}$ be an order in a quadratic number field $K$ with discriminant $d_K$ and conductor $\mathfrak{f}=f\mathcal{O}_K$ for some $f\in\mathbb{N}_{\geq 2}$.
\begin{enumerate}
\item[\textnormal{1.}] The following statements are equivalent{\rm \,:}
\begin{enumerate}
\item[\textnormal{(a)}] $\mathcal{I}(\mathcal{O})$ is half-factorial.
\item[\textnormal{(b)}] $\mathsf{c}\big(\mathcal{I}(\mathcal{O})\big)=2$.
\item[\textnormal{(c)}] $\mathsf{c}\big(\mathcal{I}^*(\mathcal{O})\big)=2$.
\item[\textnormal{(d)}] $\mathcal{I}^*(\mathcal{O})$ is half-factorial.
\item[\textnormal{(e)}] $f$ is squarefree and all prime divisors of $f$ are inert.
\end{enumerate}

\item[\textnormal{2.}] Suppose that $\mathcal{I}^*(\mathcal{O})$ is not half-factorial.
\begin{enumerate}
\item[\textnormal{(a)}] If $f$ is squarefree, then ${\rm Ca}\big(\mathcal{I}(\mathcal{O})\big)=[1,3]$, ${\rm Ca}\big(\mathcal{I}^*(\mathcal{O})\big)=[2,3]$,

    $\Delta\big(\mathcal{I}(\mathcal{O})\big)=\Delta\big(\mathcal{I}^*(\mathcal{O})\big)=\{1\}$.

\item[\textnormal{(b)}] Suppose that $f$ is not squarefree.
\begin{enumerate}
\item[\textnormal{(i)}] If ${\rm v}_2\left(f\right)\not\in\{2,3\}$ or $d_K\not\equiv 1\mod 8$, then ${\rm Ca}\big(\mathcal{I}(\mathcal{O})\big)=[1,4]$,

     ${\rm Ca}\big(\mathcal{I}^*(\mathcal{O})\big)=[2,4]$, and $\Delta\big(\mathcal{I}(\mathcal{O})\big)=\Delta\big(\mathcal{I}^*(\mathcal{O})\big)=[1,2]$.

\item[\textnormal{(ii)}] If ${\rm v}_2\left(f\right)\in\{2,3\}$ and $d_K\equiv 1\mod 8$, then ${\rm Ca}\big(\mathcal{I}(\mathcal{O})\big)=[1,5]$,

     ${\rm Ca}\big(\mathcal{I}^*(\mathcal{O})\big)=[2,5]$, and $\Delta\big(\mathcal{I}(\mathcal{O})\big)=\Delta\big(\mathcal{I}^*(\mathcal{O})\big)=[1,3]$.
\end{enumerate}
\end{enumerate}
\end{enumerate}
\end{theorem}
We say that a cancellative monoid $H$ is {\it weakly Krull} if $\bigcap_{P\in\mathfrak{X}(H)} H_P=H$ and $\{P\in\mathfrak{X}(H)\mid a\in P\}$ is finite for each $a\in H$ (where $\mathfrak{X}(H)$ denotes the set of height-one prime ideals of $H$). Moreover, a cancellative monoid $H$ is called {\it weakly factorial} if every nonunit of $H$ is a finite product of primary elements of $H$. Let all notation be as in Theorem~\ref{theorem 1.1}, and recall that $\mathcal{I}^*(\mathcal{O})$ is a weakly factorial C-monoid, and that for every atomic monoid $H$ with $\Delta(H)\ne\emptyset$ we have $\min\Delta(H)=\gcd\Delta(H)$.

There is a characterization (due to Halter-Koch) when the order $\mathcal{O}$ is half-factorial (\cite[Theorem 3.7.15]{Ge-HK06a}). This characterization and Theorem~\ref{theorem 1.1} or \cite[Corollary 4.6]{Ph12b} show that the half-factoriality of $\mathcal{O}$ implies the half-factoriality of $\mathcal{I}^*(\mathcal{O})$.
Consider the case of seminormal orders whence suppose that $\mathcal{O}$ is seminormal. Then $f$ is squarefree (this follows from an explicit characterization of seminormal orders given by Dobbs and Fontana in \cite[Corollary 4.5]{Do-Fo87}). Moreover, $\mathcal{I}^*(\mathcal{O})$ is seminormal and if $\mathcal{I}^*(\mathcal{O})$ is not half-factorial, then its catenary degree equals three by \cite[Theorems 5.5 and 5.8]{Ge-Ka-Re15a}. Clearly, this coincides with 2.(a) of the above theorem.
Among others, Theorem~\ref{theorem 1.1} shows that the sets of distances and of catenary degrees are intervals and that the minimum of the set of distances equals $1$. We discuss some analogous results and some results which are in sharp contrast to this.
If $H$ is a Krull monoid with finite class group, then $H$ is a weakly Krull C-monoid and if there are prime divisors in all classes, then the sets ${\rm Ca}(H)$ and $\Delta (H)$ are intervals (\cite[Theorem 4.1]{Ge-Zh19a}). On the other hand, for every finite set $S\subset\mathbb{N}$ with $\min S=\gcd S$ (resp. every finite set $S\subset\mathbb{N}_{\ge 2}$) there is a finitely generated Krull monoid $H$ such that $\Delta (H)=S$ (resp. ${\rm Ca}(H)= S$) (\cite{Ge-Sc17a} resp. \cite[Proposition 3.2]{Fa-Ge17a}). Just as the monoids of ideals under discussion, every numerical monoid is a weakly factorial C-monoid. However, in contrast to them, the set of distances need not be an interval (\cite{Co-Ka17a}), its minimum need not be $1$ (\cite[Proposition 2.9]{B-C-K-R06}), and a recent result of O'Neill and Pelayo (\cite{ON-Pe18a}) shows that for every finite set $S\subset\mathbb{N}_{\ge 2}$ there is a numerical monoid $H$ such that ${\rm Ca}(H)=S$.

We proceed as follows. In Section~\ref{2} we summarize the required background on the arithmetic of monoids. In Section~\ref{3} we do the same for orders in quadratic number fields and we provide an explicit description of (invertible) irreducible ideals in orders of quadratic number fields (Theorem~\ref{theorem 3.6}). In Section~\ref{4} we give the proof of Theorem~\ref{theorem 1.1}. Based on this result we establish a characterization of those orders $\mathcal{O}$ with $\min\Delta(\mathcal{O})>1$ (Theorem~\ref{theorem 4.14}) which allows us to give the first explicit examples of orders $\mathcal{O}$ with $\min\Delta(\mathcal{O})>1$.
Our third main result (given in Theorem~\ref{theorem 5.2}) states that unions of sets of lengths of $\mathcal{I}(\mathcal{O})$ and of $\mathcal{I}^*(\mathcal{O})$ are intervals.

\smallskip
\section{Preliminaries on the arithmetic of monoids}\label{2}
\smallskip

Let $\mathbb{N}$ be the set of positive integers, $\mathbb P\subset\mathbb{N}$ the set of prime numbers, and for every $m\in\mathbb{N}$, we denote by
\[
\varphi (m)=\big| (\mathbb{Z}/m\mathbb{Z})^{\times}\big|\quad\text{\it Euler's $\varphi$-function}\,.
\]
For $a,b\in\mathbb{Q}\cup\{-\infty,\infty\}$, $[a,b]=\{x\in\mathbb{Z}\mid a\le x\le b\}$ denotes the discrete interval between $a$ and $b$. Let $L,L'\subset
\mathbb{Z}$. We denote by $L+L'=\{a+b\mid a\in L,\,b\in L'\}$ their
{\it sumset}. A positive
integer $d\in\mathbb{N}$ is called a\ {\it distance}\ of $L$\ if there
exists a $k\in L$ such that $L\cap [k,k+d]=\{k,k+d\}$, and we denote by $\Delta (L)$ the {\it set of distances} of $L$. If $\emptyset\not=L\subset\mathbb{N}$, we denote by $\rho (L)=\sup L/\min L\in\mathbb{Q}_{\ge 1}\cup\{\infty\}$ the {\it elasticity} of $L$. We set $\rho (\{0\})=1$ and $\max\emptyset=\min\emptyset=\sup\emptyset=0$.
All rings and semigroups are commutative and have an identity element.

\smallskip
\noindent
\subsection{\bf Monoids.}\label{Monoids}
Let $H$ be a multiplicatively written commutative semigroup. We denote by $H^{\times}$ the group of invertible elements of $H$. We say that $H$ is reduced if $H^{\times}=\{1\}$ and we denote by $H_{\text{\rm red}}=\{aH^{\times}\mid a\in H\}$ the associated reduced semigroup of $H$. An element $u \in H$ is said to be cancellative if $au=bu$ implies that $a=b$ for all $a, b \in H$. The semigroup $H$ is said to be
\begin{itemize}
\item {\it cancellative} if every element of $H$ is cancellative.
\item {\it unit-cancellative} if $a,u\in H$ and $a=au$ implies that $u\in H^{\times}$.
\end{itemize}
By definition, every cancellative semigroup is unit-cancellative. All semigroups of ideals, that are studied in this paper, are unit-cancellative but not necessarily cancellative.

\medskip
\centerline{\it Throughout this paper, a monoid means a }
\centerline{\it commutative unit-cancellative semigroup with identity element.}
\medskip

\noindent
Let $H$ be a monoid. A submonoid $S\subset H$ is said to be
{\it divisor-closed} if $a\in S$ and $b\in H$ with $b\mid a$ implies that $b\in S$.
An element $u\in H$ is said to be
\begin{itemize}
\item {\it prime} if $u\notin H^{\times}$ and, for all $a,b\in H$, $u\mid ab$ and $u\nmid a$ implies $u\mid b$.
\item {\it primary} if $u\notin H^{\times}$ and, for all $a,b\in H$, $u\mid ab$ and $u\nmid a$ implies $u\mid b^n$ for some $n\in\mathbb{N}$.
\item {\it irreducible} (or an {\it atom}) if $u\notin H^{\times}$ and, for all $a,b\in H$, $u=ab$ implies that $a\in H^{\times}$ or $b\in H^{\times}$.
\end{itemize}
The monoid $H$ is said to be {\it atomic} if every $a\in H\setminus H^{\times}$ is a product of finitely many atoms. If $H$ satisfies the ACC (ascending chain condition) on principal ideals, then $H$ is atomic (\cite[Lemma 3.1]{F-G-K-T17}).

\smallskip
\noindent
\subsection{\bf Sets of lengths.}\label{Sets of lengths}
For a set $P$, we denote by $\mathcal{F}(P)$ the free abelian monoid with basis $P$. Every $a\in\mathcal{F}(P)$ is written in the form
\[
a=\prod_{p\in P} p^{\mathsf v_p (a)}\text{ with }\mathsf v_p (a)\in\mathbb{N}_0\quad\text{and}\quad\mathsf v_p (a)= 0\text{ for almost all $p\in P$}\,.
\]
We call $|a|=\sum_{p\in P}\mathsf v_p (a)$ the length of $a$ and ${\rm supp}(a)=\{p\in P\mid\mathsf v_p(a)> 0\}\subset P$ the support of $a$. Let $H$ be an atomic monoid. The free abelian monoid $\mathsf{Z}(H)=\mathcal{F} (\mathcal{A}(H_{\text{\rm red}}))$ denotes the {\it factorization monoid} of $H$ and
\[
\pi\colon\mathsf{Z}(H)\to H_{\text{\rm red}}\quad\text{satisfying}\quad\pi(u)=u\text{ for all } u\in\mathcal{A}(H_{\text{\rm red}})
\]
denotes the {\it factorization homomorphism} of $H$. For every $a\in H$,
\[
\begin{aligned}
\mathsf{Z}_H(a)=\mathsf{Z}(a)&=\pi^{-1}(aH^{\times})\quad\text{is the {\it set of factorizations} of $a$ and }\\
\mathsf{L}_H(a)=\mathsf{L}(a)&=\{|z|\mid z\in\mathsf{Z}(a)\}\quad\text{is the {\it set of lengths} of $a$}\,.
\end{aligned}
\]
For a divisor-closed submonoid $S\subset H$ and an element $a\in S$, we have $\mathsf{Z}(S)\subset\mathsf{Z}(H)$ whence $\mathsf{Z}_S (a)=\mathsf{Z}_H (a)$, and $\mathsf{L}_S(a)=\mathsf{L}_H(a)$. We denote by
\begin{itemize}
\item $\mathcal{L}(H)=\{\mathsf{L}(a)\mid a\in H\}$ the {\it system of sets of lengths} of $H$ and by
\item $\Delta(H)=\bigcup_{L\in\mathcal{L}(H)}\Delta(L)\subset\mathbb{N}$ the {\it set of distances} of $H$.
\end{itemize}
The monoid $H$ is said to be {\it half-factorial} if $\Delta(H)=\emptyset$ and if $H$ is not half-factorial, then $\min\Delta(H)=\gcd\Delta(H)$.

\smallskip
\noindent
\subsection{\bf Distances and chains of factorizations.}\label{Distances}
Let two factorizations $z,z'\in\mathsf{Z}(H)$ be given, say
\[
z=u_1\cdot\ldots\cdot u_{\ell}v_1\cdot\ldots\cdot v_m\quad\text{and}\quad z'=u_1\cdot\ldots\cdot u_{\ell} w_1\cdot\ldots\cdot w_n\,,
\]
where $\ell,m,n\in\mathbb{N}_0$ and all $u_i,v_j,w_k\in\mathcal{A}(H_{\text{\rm red}})$ such that $v_j\ne w_k$ for all $j\in [1,m]$ and all $k\in [1,n]$. Then $\mathsf{d}(z,z')=\max\{m,n\}$ is the {\it distance} between $z$ and $z'$. If $\pi(z)=\pi(z')$ and $z\ne z'$, then
\begin{equation}\label{equation 1}
1+\bigl||z |-|z'|\bigr|\le\mathsf{d}(z,z')\text{ resp. } 2+\bigl||z |-|z'|\bigr|\le\mathsf{d}(z,z')\text{ if $H$ is cancellative}
\end{equation}
(see \cite[Proposition 3.2]{F-G-K-T17} and \cite[Lemma 1.6.2]{Ge-HK06a}). Let $a\in H$ and $N\in\mathbb{N}_0$. A finite sequence $z_0,\ldots,z_k\in\mathsf{Z}(a)$ is called an $N$-chain of factorizations (concatenating $z_0$ and $z_k$) if $\mathsf{d}(z_{i-1},z_i)\le N$ for all $i\in [1,k]$. For $z ,z'\in\mathsf{Z}(H)$ with $\pi(z)=\pi(z')$, we set $\mathsf{c}(z,z')=\min\{N\in\mathbb{N}_0\mid z$ and $z'$ can be concatenated by an $N$-chain of factorizations from $\mathsf{Z}\big(\pi(z)\big)\}$. Then, for every $a\in H$,
\[
\mathsf{c}(a)=\sup\{\mathsf{c}(z,z')\mid z,z'\in\mathsf{Z}(a)\}\in\mathbb{N}_0\cup\{\infty\}\quad\text{is the {\it catenary degree} of $a$}.
\]
Clearly, $a$ has unique factorization (i.e., $|\mathsf{Z}(a)|=1$) if and only if $\mathsf{c}(a)=0$. We denote by
\[
{\rm Ca}(H)=\{\mathsf{c}(a)\mid a\in H,\mathsf{c}(a)>0\}\subset\mathbb{N}\quad\text{the {\it set of catenary degrees} of $H$},
\]
and then
\[
\mathsf{c}(H)=\sup{\rm Ca}(H)\in\mathbb{N}_0\cup\{\infty\}\quad\text{is the {\it catenary degree} of $H$}.
\]
We use the convention that $\sup\emptyset=0$ whence $H$ is factorial if and only if $\mathsf{c}(H)=0$.
Note that $\mathsf{c}(a)=0$ for all atoms $a\in H$. The restriction to positive catenary degrees in the definition of ${\rm Ca}(H)$ simplifies the statement of some results whence it is usual to restrict to elements with positive catenary degrees. If $H$ is cancellative, then Equation~\eqref{equation 1} implies that min ${\rm Ca}(H)\ge 2$ and
\[
2+\sup\Delta (H)\le\mathsf{c}(H)\quad\text{if $H$ is not factorial}\,.
\]
If $H=\coprod_{i\in I}H_i$, then a straightforward argument shows that
\begin{equation}\label{equation 2}
{\rm Ca}(H)=\bigcup_{i\in I}{\rm Ca}(H_i)\quad\text{whence}\quad\mathsf{c}(H)=\sup\{\mathsf{c}(H_i)\mid i\in I\}\,.
\end{equation}

\smallskip
\noindent
\subsection{\bf Semigroups of ideals.} \label{Semigroups of ideals}
Let $R$ be a domain. We denote by $\mathsf{q}(R)$ its quotient field, by $\mathfrak{X}(R)$ the set of minimal nonzero prime ideals of $R$, and by $\overline R$ its integral closure. Then $R \setminus \{0\}$ is a cancellative monoid,
\begin{itemize}
\item $\mathcal{I}(R)$ is the semigroup of nonzero ideals of $R$ (with usual ideal multiplication),
\item $\mathcal{I}^*(R)$ is the subsemigroup of invertible ideals of $R$, and
\item ${\rm Pic}(R)$ is the Picard group of $R$.
\end{itemize}
For every $I\in\mathcal{I}(R)$, we denote by $\sqrt{I}$ its radical and by $\mathcal{N}(I)=(R\negthinspace :\negthinspace I)=|R/I|\in\mathbb{N}\cup\{\infty\}$ its norm.

Let $S$ be a Dedekind domain and $R\subset S$ a subring. Then $R$ is called an {\it order} in $S$ if one of the following two equivalent conditions hold:
\begin{itemize}
\item $\mathsf{q}(R)=\mathsf{q}(S)$ and $S$ is a finitely generated $R$-module.
\item $R$ is one-dimensional noetherian and $\overline R=S$ is a finitely generated $R$-module.
\end{itemize}
Let $R$ be an order in a Dedekind domain $S=\overline R$. We analyze the structure of $\mathcal{I}^*(R)$ and of $\mathcal{I}(R)$.

Since $R$ is noetherian, Krull's Intersection Theorem holds for $R$ whence $\mathcal{I}(R)$ is unit-cancellative (\cite[Lemma 4.1]{Ge-Re18d}). Thus $\mathcal{I}(R)$ is a reduced atomic monoid with identity $R$ and $\mathcal{I}^*(R)$ is a reduced cancellative atomic divisor-closed submonoid. For the sake of clarity, we will say that an ideal of $R$ is an ideal atom if it is an atom of the monoid $\mathcal{I}(R)$. If $I,J\in\mathcal{I}^*(R)$, then $I\mid J$ if and only if $J\subset I$. The prime elements of $\mathcal{I}^*(R)$ are precisely the invertible prime ideals of $R$.
Every ideal is a product of primary ideals belonging to distinct prime ideals (in particular, $\mathcal{I}^*(R)$ is a weakly factorial monoid). Thus every ideal atom (i.e., every $I\in\mathcal{A}(\mathcal{I}(R)$) is primary, and if $\sqrt{I}=\mathfrak{p}\in\mathfrak{X}(R)$, then $I$ is $\mathfrak p$-primary.
Since $\overline R$ is a finitely generated $R$-module, the conductor $\mathfrak{f}=(R\negthinspace :\negthinspace\overline R)$ is nonzero, and we
set
\[
\mathcal{P}=\{\mathfrak p\in\mathfrak{X}(R)\mid\mathfrak{p}\not\supset\mathfrak{f}\} \quad \text{and} \quad \mathcal{P}^*=\mathfrak{X}(R)\setminus\mathcal{P} \,.
\]
Let $\mathfrak{p}\in\mathfrak{X}(R)$. We denote by
\[
\mathcal{I}^*_{\mathfrak{p}}(R)=\{I\in\mathcal{I}^*(R)\mid\sqrt{I}\supset\mathfrak{p}\} \quad \textnormal{ and } \quad \mathcal{I}_{\mathfrak{p}}(R)=\{I\in\mathcal{I}(R)\mid\sqrt{I}\supset\mathfrak{p}\}
\]
the set of invertible $\mathfrak{p}$-primary ideals of $R$ and the set of $\mathfrak{p}$-primary ideals of $R$. Clearly, these are monoids and, moreover,
\[
\mathcal{I}_{\mathfrak{p}}(R) \subset \mathcal{I}(R), \quad \mathcal{I}^*_{\mathfrak{p}}(R) \subset \mathcal{I}_{\mathfrak{p}}(R), \quad \text{and} \quad \mathcal{I}^*_{\mathfrak{p}}(R) \subset \mathcal I^* (R)
\]
are divisor-closed submonoids. Thus $\mathcal{I}^*_{\mathfrak{p}}(R)$ is a reduced cancellative atomic monoid, $\mathcal{I}_{\mathfrak{p}}(R)$ is a reduced atomic monoid, and
if $\mathfrak{p}\in\mathcal{P}$, then $\mathcal{I}^*_{\mathfrak{p}}(R)=\mathcal{I}_{\mathfrak{p}}(R)$ is free abelian. Since $R$ is noetherian and one-dimensional,
\begin{equation}\label{equation 3}
\alpha:\mathcal{I}(R)\rightarrow\coprod_{\mathfrak{p}\in\mathfrak{X}(R)}\mathcal{I}_{\mathfrak{p}}(R), \quad \text{ defined by} \quad \alpha(I)=(I_{\mathfrak{p}}\cap R)_{\mathfrak{p}\in\mathfrak{X}(R)}
\end{equation}
is a monoid isomorphism which induces a monoid isomorphism
\begin{equation}\label{equation 4}
\alpha_{\mid\mathcal{I}^*(R)}:\mathcal{I}^*(R)\rightarrow\coprod_{\mathfrak{p}\in\mathfrak{X}(R)}\mathcal{I}^*_{\mathfrak{p}}(R) \,.
\end{equation}

\smallskip
\section{Orders in quadratic number fields}\label{3}
\smallskip

The goal of this section is to prove Theorem~\ref{theorem 3.6} which provides an explicit description of (invertible) ideal atoms of an order in a quadratic number field. These results are essentially due to Butts and Pall (see \cite{Bu-Pa72} where they are given in a different style), and they were summarized without proof by Geroldinger and Lettl in \cite{Ge-Le90}. Unfortunately, that presentation is misleading in one case (namely, in case $p=2$ and $d_K\equiv 5\mod 8$). Thus we restate the results and provide a full proof.

First we put together some facts on orders in quadratic number fields and fix our notation which remains valid throughout the rest of this paper. For proofs, details, and any undefined notions we refer to \cite{HK13a}.
Let $d\in\mathbb{Z}\setminus\{0,1\}$ be squarefree, $K=\mathbb{Q}(\sqrt{d})$ be a quadratic number field,
\[
\omega=\begin{cases}
\sqrt{d},&\text{if $d\equiv 2,3\mod 4$;}\\
\frac{1+\sqrt{d}}{2},&\text{if $d\equiv 1\mod 4$.}
\end{cases}
\quad\text{and}\quad
d_K=\begin{cases}
4d,&\text{if $d\equiv 2,3\mod 4$;}\\
d,&\text{if $d\equiv 1\mod 4$.}
\end{cases}
\]
Then $\mathcal{O}_K=\mathbb{Z}[\omega]$ is the ring of integers and $d_K$ is the discriminant of $K$. For every $f\in\mathbb{N}$, we define
\[
\varepsilon\in\{0,1\}\text{ with }\varepsilon\equiv f d_K\mod 2\,,\quad\eta=\frac{\varepsilon - f^2d_K}{4}\,,\quad\text{and}\quad\tau=\frac{\varepsilon+f\sqrt{d_K}}{2}\,.
\]
Then
\[
\mathcal{O}_f=\mathbb{Z}\oplus f\omega\mathbb{Z}=\mathbb{Z}\oplus\tau\mathbb{Z}
\]
is an order in $\mathcal{O}_K$ with conductor $\mathfrak{f}=f\mathcal{O}_K$, and every order in $\mathcal{O}_K$ has this form. With the notation of Subsection~\ref{Semigroups of ideals} we have
\[
\mathcal{P}^*=\{\mathfrak p\in\mathfrak X (\mathcal{O}_f)\mid\mathfrak p\supset\mathfrak f \}= \{p\mathbb{Z}+f\omega\mathbb{Z}\mid p\in\mathbb{P},p\mid f\}\,.
\]
If $\alpha=a+b\sqrt{d}\in K$, then $\overline{\alpha}=a-b\sqrt{d}$ is its conjugate, $\mathcal{N}_{K/\mathbb{Q}}(\alpha)=\alpha\overline{\alpha}=a^2-b^2d$ is its norm, and ${\rm tr}(\alpha)=\alpha+\overline{\alpha}=2a$ is its trace. For an $I \in \mathcal I ( \mathcal{O}_f)$, $\overline{I} = \{\overline{\alpha} \mid \alpha \in I\}$ denotes the conjugate ideal. A simple calculation shows that
\[
\mathcal{N}_{K/\mathbb{Q}}(r+\tau)=r^2+\varepsilon r+\eta\quad\text{for each}\ r\in\mathbb{Z}\,.
\]
If $\mathcal{O}$ is an order and $I\in\mathcal{I}^*(\mathcal{O})$, then $(\mathcal{O}_K\negthinspace :\negthinspace I\mathcal{O}_K)=(\mathcal{O}\negthinspace :\negthinspace I)$ and if $a\in\mathcal{O}\setminus\{0\}$, then
\[
(\mathcal{O}\negthinspace :\negthinspace a\mathcal{O})=(\mathcal{O}_K\negthinspace :\negthinspace a O_K)=|\mathcal{N}_{K/\mathbb{Q}}(a)|
\]
(see \cite[Pages 99 and 100]{Ge-HK-Ka95} and note that the factor rings $\mathcal{O}_K/I\mathcal{O}_K$ and $\mathcal{O}/I$ need not be isomorphic).
For $p\in\mathbb{P}$ and for $a\in\mathbb{Z}$ we denote by $\left(\frac{a}{p}\right)\in\{-1,0,1\}$ the {\it Kronecker symbol} of $a$ modulo $p$.
A prime number $p\in\mathbb{Z}$ is called
\begin{itemize}
\item {\it inert} if $p\mathcal{O}_K\in {\rm spec}(\mathcal{O}_K)$.
\item {\it split} if $p\mathcal{O}_K$ is a product of two distinct prime ideals of $\mathcal{O}_K$.
\item {\it ramified} if $p\mathcal{O}_K$ is the square of a prime ideal of $\mathcal{O}_K$.
\end{itemize}
An odd prime
\[
p\text{ is }\begin{cases}
\text{inert}\ &\text{ if }\left(\frac{d_K}{p}\right)=-1;\\
\text{split}\ &\text{ if }\left(\frac{d_K}{p}\right)=1;\\
\text{ramified}\ &\text{ if }\left(\frac{d_K}{p}\right)=0\,.
\end{cases}
\,\text{and}\ 2\text{ is }
\begin{cases}
\text{inert}\ &\text{ if}\ d_K\equiv 5\mod 8;\\
\text{split}\ &\text{ if}\ d_K\equiv 1\mod 8;\\
\text{ramified}\ &\text{ if}\ d_K\equiv 0\mod 2\,.
\end{cases}
\]

\smallskip
\begin{proposition}\label{proposition 3.1}
Let $p$ be a prime divisor of $f$, $\mathcal{O}=\mathcal{O}_f$, and $\mathfrak{p}=p\mathbb{Z}+f\omega\mathbb{Z}$.
\begin{enumerate}
\item[\textnormal{1.}] The primary ideals with radical $\mathfrak{p}$ are exactly the ideals of the form
\[
\mathfrak{q}=p^\ell(p^m\mathbb{Z}+(r+\tau)\mathbb{Z})
\]
with $\ell,m\in\mathbb{N}_0$, $\ell+m\geq1$, $0\leq r<p^m$ and $\mathcal N_{K/\mathbb{Q}}(r+\tau)\equiv 0\mod p^m$. Moreover, $\mathcal{N}(\mathfrak{q})=p^{2\ell+m}$.

\item[\textnormal{2.}] A primary ideal $\mathfrak{q}=p^\ell(p^m\mathbb{Z}+(r+\tau)\mathbb{Z})$ is invertible if and only if
\[
\mathcal N_{K/\mathbb{Q}}(r+\tau)\not\equiv 0\mod p^{m+1}.
\]
\end{enumerate}
\end{proposition}

\begin{proof}
1. Let $\mathfrak{q}$ be a $\mathfrak{p}$-primary ideal in $\mathcal{O}$. By \cite[Theorem 5.4.2]{HK13a} there exist nonnegative integers $\ell,m,r$ such that
$\mathfrak{q}=\ell(m\mathbb{Z}+(r+\tau)\mathbb{Z})$, $r<m$ and $\mathcal N_{K/\mathbb{Q}}(r+\tau)\equiv 0\mod m$. Since $\mathfrak{q}$ is nonzero and proper, we have $\ell m>1$. We prove, that $\ell m$ is a power of $p$. First observe that $\mathfrak{q}\subset\sqrt{\mathfrak{q}}=\mathfrak{p}$ implies that $p\mid\ell m$. Assume to the contrary that there exists another rational prime $p'\not=p$ dividing $\ell m$, say $\ell m=p's$.
But then $p's\in\mathfrak{q}$, $s\not\in\mathfrak{q}$ and $p'\not\in\mathfrak{p}=\sqrt{\mathfrak{q}}$. A contradiction to $\mathfrak{q}$ being primary. Conversely, assume that $\mathfrak{q}=p^\ell(p^m\mathbb{Z}+(r+\tau)\mathbb{Z})$ for integers $\ell,m\in\mathbb{N}_0,\ell+m\geq 1,0\leq r<p^m$ and $\mathcal N_{K/\mathbb{Q}}(r+\tau)\equiv 0\mod p^m.$ By \cite[Theorem 5.4.2]{HK13a}, $\mathfrak{q}$ is an ideal of $\mathcal{O}$. Since $p\in\sqrt{\mathfrak{q}}$ and $\mathfrak{p}$ is the only prime ideal in $\mathcal{O}$ containing $p$ we obtain that $\sqrt{\mathfrak{q}}=\bigcap_{\substack{\mathfrak{a}\in {\rm spec}(\mathcal{O}),\mathfrak{a}\supset\mathfrak{q}}}\mathfrak{a}=\mathfrak{p}$. The nonzero prime ideal $\mathfrak{p}$ is maximal, since $\mathcal{O}$ is one-dimensional. Therefore, $\mathfrak{q}$ is $\mathfrak{p}$-primary. It follows from \cite[Theorem 5.4.2]{HK13a} that $\mathcal{N}(\mathfrak{q})=p^{2\ell+m}$.

\smallskip
2. By \cite[Theorem 5.4.2]{HK13a}, $\mathfrak{q}=p^\ell(p^m\mathbb{Z}+(r+\tau)\mathbb{Z})$ is invertible if and only if $\gcd(p^m,2r+\varepsilon,\frac{\mathcal N_{K/\mathbb{Q}}(r+\tau)}{p^m})=1$. Since $p\mid f$ and $\mathcal{N}_{K/\mathbb{Q}}(r+\tau)=\frac{1}{4}((2r+\varepsilon)^2-f^2d_K)$, this is the case if and only if $p\nmid\frac{\mathcal N_{K/\mathbb{Q}}(r+\tau)}{p^m}$, that is $\mathcal N_{K/\mathbb{Q}}(r+\tau)\not\equiv 0\mod p^{m+1}$.
\end{proof}

If $x\in\mathbb{Z}$ and $y\in\mathbb{N}$, then let ${\rm rem}(x,y)$ be the unique $z\in [0,y-1]$ such that $y\mid x-z$. Let $p$ be a prime divisor of $f$. Note that ${\rm v}_p(0)=\infty$, and if $\emptyset\not=A\subseteq\mathbb{N}_0$, then $\min(A\cup\{\infty\})=\min A$. We set
\[
\begin{aligned}
P_{f,p}=p\mathbb{Z}+f\omega\mathbb{Z},\quad\mathcal{I}^*_p(\mathcal{O}_f) &=\mathcal{I}^*_{P_{f,p}}(\mathcal{O}_f), \mathcal{I}_p(\mathcal{O}_f)=\mathcal{I}_{P_{f,p}}(\mathcal{O}_f),\quad\text{and}\\
\mathcal{M}_{f,p} &=\{(x,y,z)\in\mathbb{N}_0^3\mid z<p^y,{\rm v}_p(z^2+\varepsilon z+\eta)\geq y\}
\end{aligned}
\]
Let $\ast:\mathcal{M}_{f,p}\times\mathcal{M}_{f,p}\rightarrow\mathcal{M}_{f,p}$ be defined by $(u,v,w)\ast (x,y,z)=(a,b,c)$, where
\begin{align*}
&a=u+x+g,\textnormal{ }b=v+y+e-2g,\\
&c={\rm rem}\left(h-t\frac{h^2+\varepsilon h+\eta}{p^g},p^b\right),\textnormal{ }g=\min\{v,y,{\rm v}_p(w+z+\varepsilon)\},\\
&e=\min\{g,{\rm v}_p(w-z),{\rm v}_p(w^2+\varepsilon w+\eta)-v,{\rm v}_p(z^2+\varepsilon z+\eta)-y\},\\
&t\in\mathbb{Z}\textnormal{ is such that }t\frac{w+z+\varepsilon}{p^g}\equiv 1\mod\textnormal{ }p^{\min\{v,y\}-g},\textnormal{ and }h=\begin{cases} z &\textnormal{ if }y\geq v\\ w &\textnormal{ if }v>y\end{cases}.
\end{align*}
Let $\xi_{f,p}:\mathcal{M}_{f,p}\rightarrow\mathcal{I}_p(\mathcal{O}_f)$ be defined by $\xi_{f,p}(x,y,z)=p^x(p^y\mathbb{Z}+(z+\tau)\mathbb{Z})$.

\begin{proposition}\label{proposition 3.2}
Let $p$ be a prime divisor of $f$ and $I,J\in\mathcal{I}_p(\mathcal{O}_f)$.
\begin{enumerate}
\item[\textnormal{1.}] $(\mathcal{M}_{f,p},\ast)$ is a reduced monoid and $\xi_{f,p}$ is a monoid isomorphism.

\item[\textnormal{2.}] If $w,z\in\mathbb{Z}$ are such that ${\rm v}_p(w^2+\varepsilon w+\eta)>0$ and ${\rm v}_p(z^2+\varepsilon z+\eta)>0$, then ${\rm v}_p(w+z+\varepsilon)>0$ and ${\rm v}_p(w-z)>0$.

\item[\textnormal{3.}] $\mathcal{N}(I)\mathcal{N}(J)\mid\mathcal{N}(IJ)$ and $\mathcal{N}(IJ)=\mathcal{N}(I)\mathcal{N}(J)$ if and only if $I$ is invertible or $J$ is invertible. If $I$ and $J$ are proper, then $IJ\subset p\mathcal{O}_f$.

\item[\textnormal{4.}] If $I\in\mathcal{A}(\mathcal{I}_p(\mathcal{O}_f))$, then there is some $I^{\prime}\in\mathcal{A}(\mathcal{I}^*_p(\mathcal{O}_f))$ such that $\mathcal{N}(IJ)\mid\mathcal{N}(I^{\prime}J)$. If $I\in\mathcal{A}(\mathcal{I}_p(\mathcal{O}_f))$ is not invertible, then $\mathcal{N}(I)\mid\mathcal{N}(I^{\prime})$ and $\mathcal{N}(I)<\mathcal{N}(I^{\prime})$ for some $I^{\prime}\in\mathcal{A}(\mathcal{I}^*_p(\mathcal{O}_f))$.

\item[\textnormal{5.}] If $I\in\mathcal{A}(\mathcal{I}^*_p(\mathcal{O}_f))$, then $\overline{I}\in\mathcal{A}(\mathcal{I}^*_p(\mathcal{O}_f))$ and $I\overline{I}=\mathcal{N}(I)\mathcal{O}_f$.
\end{enumerate}
\end{proposition}

\begin{proof}
1. Let $(u,v,w),(x,y,z)\in\mathcal{M}_{f,p}$. Set $g=\min\{v,y,{\rm v}_p(w+z+\varepsilon)\}$ and $e=\min\{g,{\rm v}_p(w-z),{\rm v}_p(w^2+\varepsilon w+\eta)-v,{\rm v}_p(z^2+\varepsilon z+\eta)-y\}$. Note that ${\rm gcd}(p^{\min\{v,y\}},w+z+\varepsilon)=p^g$, and hence there are some $s,t\in\mathbb{Z}$ such that $sp^{\min\{v,y\}}+t(w+z+\varepsilon)=p^g$. This implies that $t\frac{w+z+\varepsilon}{p^g}\equiv 1\mod\textnormal{ }p^{\min\{v,y\}-g}$. Set $a=u+x+g$, $b=v+y+e-2g$ and let $h=z$ if $y\geq v$ and $h=w$ if $v>y$. Finally, set $c={\rm rem}(h-t\frac{h^2+\varepsilon h+\eta}{p^g},p^b)$. First we show that $c$ does not depend on the choice of $t$. Let $t^{\prime}\in\mathbb{Z}$ be such that $t^{\prime}\frac{w+z+\varepsilon}{p^g}\equiv 1\mod\textnormal{ }p^{\min\{v,y\}-g}$. Then $p^{\min\{v,y\}-g}\mid t-t^{\prime}$. Note that $\min\{v,y\}+{\rm v}_p(h^2+\varepsilon h+\eta)\geq v+y+e$, and hence $p^b\mid (t-t^{\prime})\frac{h^2+\varepsilon h+\eta}{p^g}$. Consequently, $c={\rm rem}(h-t^{\prime}\frac{h^2+\varepsilon h+\eta}{p^g},p^b)$.

Next we show that $(a,b,c)\in\mathcal{M}_{f,p}$. It is clear that $(a,b,c)\in\mathbb{N}_0^3$ and $c<p^b$. It remains to show that ${\rm v}_p(c^2+\varepsilon c+\eta)\geq b$. Without restriction we can assume that $v\leq y$. Then $h=z$. Set $k=z-t\frac{z^2+\varepsilon z+\eta}{p^g}$. There is some $r\in\mathbb{Z}$ such that $c=k+rp^b$. Since $c^2+\varepsilon c+\eta=k^2+\varepsilon k+\eta+mp^b$ for some $m\in\mathbb{Z}$, it is sufficient to show that ${\rm v}_p(k^2+\varepsilon k+\eta)\geq b$.

Observe that $k^2+\varepsilon k+\eta=\frac{z^2+\varepsilon z+\eta}{p^{2g}}(p^{2g}-tp^g(2z+\varepsilon)+t^2(z^2+\varepsilon z+\eta))=\frac{z^2+\varepsilon z+\eta}{p^{2g}}(sp^{v+g}+tp^g(w-z)+t^2(z^2+\varepsilon z+\eta))$. Note that $g+{\rm v}_p(w-z)=\min\{v+{\rm v}_p(w-z),{\rm v}_p(w+z+\varepsilon)+{\rm v}_p(w-z)\}=\min\{v+{\rm v}_p(w-z),{\rm v}_p(w^2+\varepsilon w+\eta-(z^2+\varepsilon z+\eta))\}\geq\min\{v+{\rm v}_p(w-z),{\rm v}_p(z^2+\varepsilon z+\eta),{\rm v}_p(w^2+\varepsilon w+\eta)\}\geq v$. Moreover, we have ${\rm v}_p(z^2+\varepsilon z+\eta)\geq y+e$. Therefore, ${\rm v}_p(k^2+\varepsilon k+\eta)\geq {\rm v}_p(z^2+\varepsilon z+\eta)-2g+\min\{v+g,g+{\rm v}_p(w-z),{\rm v}_p(z^2+\varepsilon z+\eta)\}\geq y+e-2g+v=b$.

Now we prove that $p^u(p^v\mathbb{Z}+(w+\tau)\mathbb{Z})p^x(p^y\mathbb{Z}+(z+\tau)\mathbb{Z})=p^a(p^b\mathbb{Z}+(c+\tau)\mathbb{Z})$. (Note that this can be shown by using \cite[Theorem 5.4.6]{HK13a}.) Set $I=p^u(p^v\mathbb{Z}+(w+\tau)\mathbb{Z})p^x(p^y\mathbb{Z}+(z+\tau)\mathbb{Z})$. Without restriction let $v\leq y$. Note that $(w+\tau)(z+\tau)=wz-\eta+(w+z+\varepsilon)\tau$. Set $\alpha=p^v(z+\tau)$ and $\beta=wz-\eta+(w+z+\varepsilon)\tau$. We infer that $I=p^{u+x}(p^{v+y}\mathbb{Z}+p^y(w+\tau)\mathbb{Z}+\alpha\mathbb{Z}+\beta\mathbb{Z})$.

Moreover, $p^y(w+\tau)\mathbb{Z}+\alpha\mathbb{Z}=p^y(w-z)\mathbb{Z}+\alpha\mathbb{Z}$. Observe that $s\alpha+t\beta=p^gz-t(z^2+\varepsilon z+\eta)+p^g\tau$. Set $k=z-t\frac{z^2+\varepsilon z+\eta}{p^g}$. Then $s\alpha+t\beta=p^g(k+\tau)$. We have $\alpha-p^v(k+\tau)=tp^{v-g}(z^2+\varepsilon z+\eta)$ and $(w+z+\varepsilon)(k+\tau)-\beta=sp^{v-g}(z^2+\varepsilon z+\eta)$. Set $r=p^{v-g}(z^2+\varepsilon z+\eta)$. Consequently, $\alpha\mathbb{Z}+\beta\mathbb{Z}=sr\mathbb{Z}+tr\mathbb{Z}+p^g(k+\tau)\mathbb{Z}=r\mathbb{Z}+p^g(k+\tau)\mathbb{Z}$,
since ${\rm gcd}(s,t)=1$. Putting these facts together gives us
$I=p^{u+x}(p^{v+y}\mathbb{Z}+p^y(w-z)\mathbb{Z}+r\mathbb{Z}+p^g(k+\tau)\mathbb{Z})$.

We have ${\rm gcd}(p^{v+y},p^y(w-z),r)=p^{\ell}$ with $\ell=\min\{v+y,y+{\rm v}_p(w-z),v-g+{\rm v}_p(z^2+\varepsilon z+\eta)\}$ and $p^{v+y}\mathbb{Z}+p^y(w-z)\mathbb{Z}+r\mathbb{Z}=p^{\ell}\mathbb{Z}$. Note that $\ell=v+y-g+\min\{g,{\rm v}_p(w-z)-v+g,{\rm v}_p(z^2+\varepsilon z+\eta)-y\}$ and ${\rm v}_p(w-z)-v+g=\min\{{\rm v}_p(w-z),{\rm v}_p(w-z)+{\rm v}_p(w+z+\varepsilon)-v\}=\min\{{\rm v}_p(w-z),{\rm v}_p(w^2+\varepsilon w+\eta-(z^2+\varepsilon z+\eta))-v\}$, and hence $\ell=v+y-g+\min\{g,{\rm v}_p(w-z),{\rm v}_p(w^2+\varepsilon w+\eta-(z^2+\varepsilon z+\eta))-v,{\rm v}_p(z^2+\varepsilon z+\eta)-y\}$.

\smallskip
CASE 1: ${\rm v}_p(w^2+\varepsilon w+\eta)\geq {\rm v}_p(z^2+\varepsilon z+\eta)$. Then ${\rm v}_p(w^2+\varepsilon w+\eta)-v\geq {\rm v}_p(z^2+\varepsilon z+\eta)-y$ and ${\rm v}_p(w^2+\varepsilon w+\eta-(z^2+\varepsilon z+\eta))-v\geq {\rm v}_p(z^2+\varepsilon z+\eta)-y$.

\smallskip
CASE 2: ${\rm v}_p(z^2+\varepsilon z+\eta)>{\rm v}_p(w^2+\varepsilon w+\eta)$. Then ${\rm v}_p(w^2+\varepsilon w+\eta-(z^2+\varepsilon z+\eta))-v={\rm v}_p(w^2+\varepsilon w+\eta)-v$.

\smallskip
In any case we have $\min\{{\rm v}_p(w^2+\varepsilon w+\eta-(z^2+\varepsilon z+\eta))-v,{\rm v}_p(z^2+\varepsilon z+\eta)-y\}=\min\{{\rm v}_p(w^2+\varepsilon w+\eta)-v,{\rm v}_p(z^2+\varepsilon z+\eta)-y\}$. Obviously, $\ell=v+y+e-g$ and $I=p^{u+x+g}(p^{v+y+e-2g}\mathbb{Z}+(z-t\frac{z^2+\varepsilon z+\eta}{p^g}+\tau)\mathbb{Z})$. Consequently, $I=p^a(p^b\mathbb{Z}+(c+\tau)\mathbb{Z})$.

So far we know that $\ast$ is an inner binary operation on $\mathcal{M}_{f,p}$. It follows from Proposition~\ref{proposition 3.1}.1 that $\xi_{f,p}$ is surjective. It follows from \cite[Theorem 5.4.2]{HK13a} that $\xi_{f,p}$ is injective. It is clear that $(\mathcal{I}_p(O_f),\cdot)$ is a reduced monoid. We have shown that $\xi_{f,p}$ maps products of elements of $\mathcal{M}_{f,p}$ to products of elements of $\mathcal{I}_p(O_f)$. It is clear that $(0,0,0)$ is an identity element of $\mathcal{M}_{f,p}$ and $\xi_{f,p}(0,0,0)=\mathcal{O}_f$. Therefore, $(\mathcal{M}_{f,p},\ast)$ is a reduced monoid and $\xi_{f,p}$ is a monoid isomorphism.

\smallskip
2. Let $w,z\in\mathbb{Z}$ be such that ${\rm v}_p(w^2+\varepsilon w+\eta)>0$ and ${\rm v}_p(z^2+\varepsilon z+\eta)>0$. Then $p\mid z^2+\varepsilon z+\eta=\frac{1}{4}((2z+\varepsilon)^2-f^2d_K)$, and hence $p\mid 2z+\varepsilon$. Moreover $p\mid w^2+\varepsilon w+\eta-(z^2+\varepsilon z+\eta)=(w+z+\varepsilon)(w-z)$, and thus $p\mid w+z+\varepsilon$ or $p\mid w-z$. Since $p\mid 2z+\varepsilon$, we infer that $p\mid w+z+\varepsilon$ if and only if $p\mid w-z$. Consequently, $\min\{{\rm v}_p(w+z+\varepsilon),{\rm v}_p(w-z)\}>0$.

\smallskip
3. By 1., there are $(u,v,w),(x,y,z),(a,b,c)\in\mathcal{M}_{f,p}$ such that $I=p^u(p^v\mathbb{Z}+(w+\tau)\mathbb{Z})$, $J=p^x(p^y\mathbb{Z}+(z+\tau)\mathbb{Z})$ and $IJ=p^a(p^b\mathbb{Z}+(c+\tau)\mathbb{Z})$ with $a=u+x+g$, $b=v+y+e-2g$, $g=\min\{v,y,{\rm v}_p(w+z+\varepsilon)\}$ and $e=\min\{g,{\rm v}_p(w-z),{\rm v}_p(w^2+\varepsilon w+\eta)-v,{\rm v}_p(z^2+\varepsilon z+\eta)-y\}$. It follows by Proposition~\ref{proposition 3.1}.1 that $\mathcal{N}(I)=p^{2u+v}$, $\mathcal{N}(J)=p^{2x+y}$ and $\mathcal{N}(IJ)=p^{2a+b}=p^{2(u+x)+v+y+e}$. It is obvious that $\mathcal{N}(I)\mathcal{N}(J)\mid\mathcal{N}(IJ)$. Moreover, $\mathcal{N}(IJ)=\mathcal{N}(I)\mathcal{N}(J)$ if and only if $e=0$. We infer by 2. that $e=0$ if and only if $v=0$ or $y=0$ or ${\rm v}_p(w^2+\varepsilon w+\eta)=v$ or ${\rm v}_p(z^2+\varepsilon z+\eta)=y$, which is the case if and only if $I$ is invertible or $J$ is invertible by Proposition~\ref{proposition 3.1}.2. If $I$ and $J$ are proper, then $u+v>0$ and $x+y>0$, and hence $a>0$ by 2. This implies that $IJ\subset p(p^b\mathbb{Z}+(c+\tau)\mathbb{Z})\subset p\mathcal{O}_f$.

\smallskip
4. Let $I\in\mathcal{A}(\mathcal{I}_p(\mathcal{O}_f))$. Without restriction let $I$ be not invertible. We have $I=p^b\mathbb{Z}+(r+\tau)\mathbb{Z}$ for some $(0,b,r)\in\mathcal{M}_{f,p}$ and $b<{\rm v}_p(r^2+\varepsilon r+\eta)$. Set $c={\rm v}_p(r^2+\varepsilon r+\eta)$ and $I^{\prime}=p^c\mathbb{Z}+(r+\tau)\mathbb{Z}$. Then $I^{\prime}\in\mathcal{A}(\mathcal{I}^*_p(\mathcal{O}_f))$, $\mathcal{N}(I)\mid\mathcal{N}(I^{\prime})$, and $\mathcal{N}(I)<\mathcal{N}(I^{\prime})$ by Proposition~\ref{proposition 3.1}. There is some $(x,y,z)\in\mathcal{M}_{f,p}$ such that $J=p^x(p^y\mathbb{Z}+(z+\tau)\mathbb{Z})$. Then $\mathcal{N}(I^{\prime}J)=p^{c+2x+y}$ and $\mathcal{N}(IJ)=p^{b+2x+y+e}$ with $e=\min\{b,y,{\rm v}_p(r+z+\varepsilon),{\rm v}_p(r-z),c-b,{\rm v}_p(z^2+\varepsilon z+\eta)-y\}\leq c-b$. Therefore, $\mathcal{N}(IJ)\mid\mathcal{N}(I^{\prime}J)$.

\smallskip
5. Let $I\in\mathcal{A}(\mathcal{I}^*_p(\mathcal{O}_f))$. If $I=p\mathcal{O}_f$, then $\overline{I}=p\mathcal{O}_f$ and $\mathcal{N}(I)=p^2$ by Proposition~\ref{proposition 3.1}.1. Therefore, $I\overline{I}=\mathcal{N}(I)\mathcal{O}_f$. Now let $I\not=p\mathcal{O}_f$. There is some $(0,m,r)\in\mathcal{M}_{f,p}$ such that $I=p^m\mathbb{Z}+(r+\tau)\mathbb{Z}$. Set $s=p^m-r-\varepsilon$. It follows that $\overline{I}=p^m\mathbb{Z}+(r+\overline{\tau})\mathbb{Z}=p^m\mathbb{Z}+(r+\varepsilon-\tau)\mathbb{Z}=p^m\mathbb{Z}+(s+\tau)\mathbb{Z}$. Observe that $s^2+\varepsilon s+\eta=r^2+\varepsilon r+\eta+p^m(p^m-(2r+\varepsilon))$. Since $p\mid r^2+\varepsilon r+\eta=\frac{1}{4}((2r+\varepsilon)^2-f^2d_K)$, we have ${\rm v}_p(2r+\varepsilon)>0$, and hence ${\rm v}_p(p^m(p^m-(2r+\varepsilon)))>m$. Since ${\rm v}_p(r^2+\varepsilon r+\eta)=m$, we infer that ${\rm v}_p(s^2+\varepsilon s+\eta)=m$, and thus $(0,m,s)\in\mathcal{M}_{f,p}$. Therefore, $\overline{I}\in\mathcal{A}(\mathcal{I}^*_p(\mathcal{O}_f))$. Note that $\min\{m,{\rm v}_p(r+s+\varepsilon)\}=m$, and thus $I\overline{I}=p^m\mathcal{O}_f=\mathcal{N}(I)\mathcal{O}_f$ by 1. and Proposition~\ref{proposition 3.1}.1.
\end{proof}

\begin{proposition}\label{proposition 3.3}
Let $p$ be a prime divisor of $f$ and $f^{\prime}=p^{{\rm v}_p(f)}$. Set $\mathcal{O}=\mathcal{O}_f$, $\mathcal{O}^{\prime}=\mathcal{O}_{f^{\prime}}$, $P=P_{f,p}$ and $P^{\prime}=P_{f^{\prime},p}$. For $g\in\mathbb{N}$ let $\varphi_{g,p}:\mathcal{I}_p(\mathcal{O}_g)\rightarrow\mathcal{I}(({\mathcal{O}_g})_{P_{g,p}})$ be defined by $\varphi_{g,p}(I)=I_{P_{g,p}}$ and $\zeta_{g,p}:\mathcal{I}(({\mathcal{O}_g})_{P_{g,p}})\rightarrow\mathcal{I}_p(\mathcal{O}_g)$ be defined by $\zeta_{g,p}(J)=J\cap\mathcal{O}_g$.
\begin{enumerate}
\item[\textnormal{1.}] $\mathcal{O}_P=\mathcal{O}^{\prime}_{P^{\prime}}$.

\item[\textnormal{2.}] $\varphi_{f,p}$ and $\zeta_{f,p}$ are mutually inverse monoid isomorphisms.

\item[\textnormal{3.}] There is a monoid isomorphism $\delta:\mathcal{I}_p(\mathcal{O})\rightarrow\mathcal{I}_p(\mathcal{O}^{\prime})$ such that $\delta(p\mathcal{O})=p\mathcal{O}^{\prime}$ and $\delta_{\mid\mathcal{I}^*_p(\mathcal{O})}:\mathcal{I}^*_p(\mathcal{O})\rightarrow\mathcal{I}^*_p(\mathcal{O}^{\prime})$ is a monoid isomorphism.
\end{enumerate}
\end{proposition}

\begin{proof}
1. It is clear that $\mathcal{O}\subset\mathcal{O}^{\prime}$ and $P^{\prime}\cap\mathcal{O}=P$. Therefore, $\mathcal{O}_P\subset\mathcal{O}^{\prime}_{P^{\prime}}$. Observe that $\mathcal{O}\setminus P=(\mathbb{Z}\setminus p\mathbb{Z})+f\omega\mathbb{Z}$ and $\mathcal{O}^{\prime}\setminus P^{\prime}=(\mathbb{Z}\setminus p\mathbb{Z})+f^{\prime}\omega\mathbb{Z}$. It remains to show that $\{f^{\prime}\omega\}\cup\{x^{-1}\mid x\in (\mathbb{Z}\setminus p\mathbb{Z})+f^{\prime}\omega\mathbb{Z}\}\subset\mathcal{O}_P$. Since $\frac{f}{f^{\prime}}f^{\prime}\omega=f\omega\in\mathcal{O}$ and $\frac{f}{f^{\prime}}\in\mathbb{Z}\setminus p\mathbb{Z}\subset\mathcal{O}\setminus P$, we have $f^{\prime}\omega\in\mathcal{O}_P$. Therefore, $\mathcal{O}^{\prime}\subset\mathcal{O}_P$. Now let $a\in\mathbb{Z}\setminus p\mathbb{Z}$ and $b\in\mathbb{Z}$. Observe that $a+bf^{\prime}\overline{\omega}\in\mathcal{O}^{\prime}\subset\mathcal{O}_P$. Since $\omega+\overline{\omega},\omega\overline{\omega}\in\mathbb{Z}$, we have $(a+bf^{\prime}\omega)(a+bf^{\prime}\overline{\omega})=a^2+abf^{\prime}(\omega+\overline{\omega})+b^2(f^{\prime})^2\omega\overline{\omega}\in\mathbb{Z}\setminus p\mathbb{Z}\subset\mathcal{O}\setminus P$. Therefore, $\frac{1}{a+bf^{\prime}\omega}=\frac{a+bf^{\prime}\overline{\omega}}{(a+bf^{\prime}\omega)(a+bf^{\prime}\overline{\omega})}\in\mathcal{O}_P$.

\smallskip
2. It is clear that $\varphi_{f,p}$ is a well-defined monoid homomorphism. Note that $\zeta_{f,p}$ is a well-defined map (since every nonzero proper ideal $J$ of $\mathcal{O}_P$ is $P_P$-primary, and hence $J\cap\mathcal{O}$ is $P$-primary). Moreover, $\zeta_{f,p}(\mathcal{O}_P)=\mathcal{O}$. Now let $J_1,J_2\in\mathcal{I}(\mathcal{O}_P)$. Observe that $J_1J_2\cap\mathcal{O}$ and $(J_1\cap\mathcal{O})(J_2\cap\mathcal{O})$ coincide locally (note that both are either $P$-primary or not proper). Therefore, $J_1J_2\cap\mathcal{O}=(J_1\cap\mathcal{O})(J_2\cap\mathcal{O})$, and hence $\zeta_{f,p}$ is a monoid homomorphism. If $J\in\mathcal{I}(\mathcal{O}_P)$, then $(J\cap\mathcal{O})_P=J$. Therefore, $\varphi_{f,p}\circ\zeta_{f,p}={\rm id}_{\mathcal{I}(\mathcal{O}_P)}$. If $I$ is a $P$-primary ideal of $\mathcal{O}$, then $I_P\cap\mathcal{O}=I$. This implies that $\zeta_{f,p}\circ\varphi_{f,p}={\rm id}_{\mathcal{I}_p(\mathcal{O})}$.

\smallskip
3. Set $\delta=\zeta_{f^{\prime},p}\circ\varphi_{f,p}$. Then $\delta:\mathcal{I}_p(\mathcal{O})\rightarrow\mathcal{I}_p(\mathcal{O}^{\prime})$ is a monoid isomorphism by 1. and 2. Furthermore, we have by 1. that $\delta(p\mathcal{O})=\zeta_{f^{\prime},p}(\varphi_{f,p}(p\mathcal{O}))=\zeta_{f^{\prime},p}(p\mathcal{O}_P)=\zeta_{f^{\prime},p}(p\mathcal{O^{\prime}}_{P^{\prime}})=p\mathcal{O}^{\prime}_{P^{\prime}}\cap\mathcal{O}^{\prime}=p\mathcal{O}^{\prime}$.

Since $\mathcal{O}$ is noetherian, we have $\mathcal{I}^*_p(\mathcal{O})$ is the set of cancellative elements of $\mathcal{I}_p(\mathcal{O})$. It follows by analogy that $\mathcal{I}^*_p(\mathcal{O}^{\prime})$ is the set of cancellative elements of $\mathcal{I}_p(\mathcal{O}^{\prime})$. Therefore, $\delta(\mathcal{I}^*_p(\mathcal{O}))=\mathcal{I}^*_p(\mathcal{O}^{\prime})$, and hence $\delta_{\mid\mathcal{I}^*_p(\mathcal{O})}$ is a monoid isomorphism.
\end{proof}

\begin{lemma}\label{lemma 3.4} Let $p$ be a prime number, let $k\in\mathbb{N}_0$, let $c,n\in\mathbb{N}$ be such that ${\rm gcd}(c,p)=1$ and for each $\ell\in\mathbb{N}$ let $g_{\ell}=|\{y\in [0,p^{\ell}-1]\mid y^2\equiv c\mod p^{\ell}\}|$.
\begin{enumerate}
\item[\textnormal{1.}] If $p\not=2$, then $p^kc$ is a square modulo $p^n$ if and only if $k\geq n$ or $(k<n$, $k$ is even and $(\frac{c}{p})=1)$.

\item[\textnormal{2.}] $2^kc$ is a square modulo $2^n$ if and only if one of the following conditions holds.
\begin{enumerate}
\item[\textnormal{(a)}] $k\geq n$.
\item[\textnormal{(b)}] $k$ is even and $n=k+1$.
\item[\textnormal{(c)}] $k$ is even, $n=k+2$ and $c\equiv 1\mod 4$.
\item[\textnormal{(d)}] $k$ is even, $n\geq k+3$ and $c\equiv 1\mod 8$.
\end{enumerate}

\item[\textnormal{3.}] If $\ell\in\mathbb{N}$, then $g_{\ell}=\begin{cases} 4 &{\it if}\textnormal{ }p=2,\ell\geq 3,c\equiv 1\mod 8\\ 2 &{\it if}\textnormal{ }(p\not=2,(\frac{c}{p})=1)\textnormal{ }{\it or}\textnormal{ }(p=2,\ell=2,c\equiv 1\mod 4)\\ 1 &{\it if}\textnormal{ }p=2,\ell=1\\ 0 &{\it else}\end{cases}$.
\end{enumerate}
\end{lemma}

\begin{proof}
Note that $p^kc$ is a square modulo $p^n$ iff $k\geq n$ or $(k<n$, $k$ is even and $c$ is a square modulo $p^{n-k}$).

\smallskip
1. Let $p\not=2$. It remains to show that if $\ell\in\mathbb{N}$, then $c$ is a square modulo $p^{\ell}$ if and only if $(\frac{c}{p})=1$. If $\ell\in\mathbb{N}$ and $c$ is a square modulo $p^{\ell}$, then $c$ is a square modulo $p$, and hence $(\frac{c}{p})=1$. Now let $(\frac{c}{p})=1$. It suffices to show by induction that $c$ is a square modulo $p^{\ell}$ for all $\ell\in\mathbb{N}$. The statement is clearly true for $\ell=1$. Now let $\ell\in\mathbb{N}$ and let $x\in\mathbb{Z}$ be such that $x^2\equiv c\mod p^{\ell}$. Without restriction let ${\rm v}_p(x^2-c)=\ell$. Note that $p\nmid x$, and hence $2bx\equiv -1\mod p$ for some $b\in\mathbb{Z}$. Set $y=x+b(x^2-c)$. Then $y^2\equiv c\mod p^{\ell+1}$.

\smallskip
2. It remains to show that if $\ell\in\mathbb{N}$, then $c$ is a square modulo $2^{\ell}$ if and only if $\ell=1$ or $(\ell=2$ and $c\equiv 1\mod 4)$ or $(\ell\geq 3$ and $c\equiv 1\mod 8)$. Let $\ell\in\mathbb{N}$ and let $c$ be a square modulo $2^{\ell}$. If $\ell=2$, then $c$ is a square modulo $4$ and $c\equiv 1\mod 4$. Moreover, if $\ell\geq 3$, then $c$ is a square modulo $8$ and $c\equiv 1\mod 8$.

Clearly, if $\ell=1$ or ($\ell=2$ and $c\equiv 1\mod 4$), then $c$ is a square modulo $2^{\ell}$. Now let $c\equiv 1\mod 8$. It is sufficient to show by induction that $c$ is a square modulo $2^{\ell}$ for each $\ell\in\mathbb{N}_{\geq 3}$. The statement is obviously true for $\ell=3$. Now let $\ell\in\mathbb{N}_{\geq 3}$ and let $x\in\mathbb{Z}$ be such that $x^2\equiv c\mod 2^{\ell}$. Without restriction let ${\rm v}_2(x^2-c)=\ell$. Set $y=x+2^{\ell-1}$. Then $y^2\equiv c\mod 2^{\ell+1}$.

\smallskip
3. Let $\ell\in\mathbb{N}$. By 1. and 2., it is sufficient to consider the case $g_{\ell}>0$. Let $g_{\ell}>0$. Observe that $g_{\ell}=|\{y\in [0,p^{\ell}-1]\mid y^2\equiv 1\mod p^{\ell}\}|=|\{y\in (\mathbb{Z}/p^{\ell}\mathbb{Z})^{\times}\mid {\rm ord}(y)\leq 2\}|$. If $p=2$ and $\ell=1$, then $(\mathbb{Z}/p^{\ell}\mathbb{Z})^{\times}$ is trivial, and hence $g_{\ell}=1$. If ($p=2$, $\ell=2$ and $c\equiv 1\mod 4$) or ($p\not=2$ and $(\frac{c}{p})=1$), then $(\mathbb{Z}/p^{\ell}\mathbb{Z})^{\times}$ is a cyclic group of even order, and thus $g_{\ell}=2$. Finally, if $p=2$, $\ell\geq 3$ and $c\equiv 1\mod 8$, then $(\mathbb{Z}/2^{\ell}\mathbb{Z})^{\times}\cong\mathbb{Z}/2\mathbb{Z}\times\mathcal{C}_{2^{\ell-2}}$ is the product of two cyclic groups of even order. Consequently, $g_{\ell}=4$.
\end{proof}

\begin{lemma}\label{lemma 3.5} Let $p$ be a prime number, $a,m\in\mathbb{N}$, $c=\frac{a}{p^{{\rm v}_p(a)}}$, $M=\{x\in [0,p^m-1]\mid {\rm v}_p(x^2-a)=m\}$, $N=|M|$ and for each $\ell\in\mathbb{N}$ let $g_{\ell}=|\{y\in [0,p^{\ell}-1]\mid y^2\equiv c\mod p^{\ell}\}|$.
\begin{enumerate}
\item[\textnormal{1.}] If $m<{\rm v}_p(a)$, then $N=\begin{cases}\varphi(p^{m/2}) &{\it if}\textnormal{ }m\textnormal{ }{\it is}\textnormal{ }{\it even}\\ 0 &{\it if}\textnormal{ }m\textnormal{ }{\it is}\textnormal{ }{\it odd}\end{cases}$.

\item[\textnormal{2.}] Let $m={\rm v}_p(a)$.

\begin{enumerate}
\item[\textnormal{(a)}] If $a$ is a square modulo $p^{m+1}$, then $N=\begin{cases}p^{m/2-1}(p-2) &{\it if}\textnormal{ }p\not=2\\ 2^{{m/2}-1} &{\it if}\textnormal{ }p=2\end{cases}$.

\item[\textnormal{(b)}] If $a$ is not a square modulo $p^{m+1}$, then $N=p^{\lfloor m/2\rfloor}$.
\end{enumerate}

\item[\textnormal{3.}] If $m>{\rm v}_p(a)$ and $a$ is not a square modulo $p^m$, then $N=0$.

\item[\textnormal{4.}] If $k\in\mathbb{N}$ is such that $m=k+{\rm v}_p(a)$ and $a$ is a square modulo $p^m$, then $N=p^{{\rm v}_p(a)/2-1}(pg_k-g_{k+1})$.
\end{enumerate}
\end{lemma}

\begin{proof}
1. Let $m<{\rm v}_p(a)$. Observe that $M=\{x\in [0,p^m-1]\mid 2{\rm v}_p(x)=m\}$. Clearly, if $m$ is odd, then $N=0$. Now let $m$ be even. We have $M=\{p^{m/2}y\mid y\in [0,p^{m/2}-1],{\rm gcd}(y,p)=1\}$, and thus $N=|\{y\in [0,p^{m/2}-1]\mid {\rm gcd}(y,p)=1\}|=\varphi(p^{m/2})$.

\smallskip
2. Note that $M=\{x\in [0,p^m-1]\mid 2{\rm v}_p(x)\geq m,x^2\not\equiv a\mod p^{m+1}\}$ and $|\{x\in [0,p^m-1]\mid 2{\rm v}_p(x)\geq m\}|=p^{\lfloor m/2\rfloor}$. Set $M^{\prime}=\{x\in [0,p^m-1]\mid x^2\equiv a\mod p^{m+1}\}$. Then $M^{\prime}=\{x\in [0,p^m-1]\mid 2{\rm v}_p(x)\geq m,x^2\equiv a\mod p^{m+1}\}$ and $N=p^{\lfloor m/2\rfloor}-|M^{\prime}|$. If $a$ is not a square modulo $p^{m+1}$, then $M^{\prime}=\emptyset$, and hence $N=p^{\lfloor m/2\rfloor}$. Now let $a$ be a square modulo $p^{m+1}$. Then $M^{\prime}\not=\emptyset$, and thus $m$ is even. Observe that $M^{\prime}=\{x\in [0,p^m-1]\mid 2{\rm v}_p(x)=m,x^2\equiv a\mod p^{m+1}\}=\{p^{m/2}y\mid y\in [0,p^{m/2}-1],y^2\equiv c\mod p\}$. Therefore, $|M^{\prime}|=|\{y\in [0,p^{m/2}-1]\mid y^2\equiv c\mod p\}|=p^{m/2-1}|\{y\in [0,p-1]\mid y^2\equiv c\mod p\}|$.

If $p\not=2$, then $N=p^{\lfloor m/2\rfloor}-|M^{\prime}|=p^{m/2}-2p^{m/2-1}=p^{m/2-1}(p-2)$ by Lemma~\ref{lemma 3.4}.3. Moreover, if $p=2$, then $N=2^{\lfloor m/2\rfloor}-|M^{\prime}|=2^{m/2}-2^{m/2-1}=2^{m/2-1}$ by Lemma~\ref{lemma 3.4}.3.

\smallskip
3. This is obvious.

4. Let $k\in\mathbb{N}$ be such that $m=k+{\rm v}_p(a)$ and let $a$ be a square modulo $p^m$. It follows by Lemma~\ref{lemma 3.4} that ${\rm v}_p(a)$ is even. Set $r={\rm v}_p(a)/2$ and for $\theta\in\{0,1\}$ set $M_{\theta}=\{x\in [0,p^m-1]\mid 2{\rm v}_p(x)={\rm v}_p(a),x^2\equiv a\mod p^{m+\theta}\}$. Then $M=\{x\in [0,p^m-1]\mid {\rm v}_p(x)=r,{\rm v}_p(x^2-a)=m\}=M_0\setminus M_1$. Since $\{x\in [0,p^m-1]\mid {\rm v}_p(x)=r\}=\{p^ry\mid y\in [0,p^{k+r}-1],{\rm gcd}(y,p)=1\}$, we infer that $M_{\theta}=\{p^ry\mid y\in [0,p^{k+r}-1],y^2\equiv c\mod p^{k+\theta}\}$. Therefore, $|M_{\theta}|=|\{y\in [0,p^{k+r}-1]\mid y^2\equiv c\mod p^{k+\theta}\}|=p^{r-\theta}|\{y\in [0,p^{k+\theta}-1]\mid y^2\equiv c\mod p^{k+\theta}\}|=p^{r-\theta}g_{k+\theta}$. This implies that $N=|M_0|-|M_1|=p^rg_k-p^{r-1}g_{k+1}=p^{r-1}(pg_k-g_{k+1})$.
\end{proof}

\begin{theorem}\label{theorem 3.6}
Let $\mathcal{O}$ be an order in a quadratic number field $K$ with conductor $\mathfrak{f}=f\mathcal{O}_K$ for some $f\in\mathbb{N}_{\geq 2}$, $p$ be a prime divisor of $f$, and $\mathfrak{p}=P_{f,p}$.
\begin{enumerate}
\item[\textnormal{1.}] The primary ideals with radical $\mathfrak{p}$ are exactly the ideals of the form
\[
\mathfrak{q}=p^\ell(p^m\mathbb{Z}+(r+\tau)\mathbb{Z})
\]
with $\ell,m\in\mathbb{N}_0$, $\ell+m\geq1$, $0\leq r<p^m$, and $\mathcal N_{K/\mathbb{Q}}(r+\tau)\equiv 0\mod p^m$. Moreover, $\mathcal{N}(\mathfrak{q})=p^{2\ell+m}$.

\item[\textnormal{2.}] A primary ideal $\mathfrak{q}=p^\ell(p^m\mathbb{Z}+(r+\tau)\mathbb{Z})$ is invertible if and only if
\[
\mathcal N_{K/\mathbb{Q}}(r+\tau)\not\equiv 0\mod p^{m+1}.
\]

\item[\textnormal{3.}] A primary ideal $\mathfrak{q}$ with radical $\mathfrak{p}$ is an ideal atom if and only if $\mathfrak{q}=p\mathcal{O}$ or $\mathfrak{q}=p^m\mathbb{Z}+(r+\tau)\mathbb{Z}$ with
$m\in\mathbb{N}$ and $p^m\mid \mathcal N_{K/\mathbb{Q}}(r+\tau)$.

\item[\textnormal{4.}] Table~\ref{table1} gives the number of invertible ideal atoms of the form $p^m\mathbb{Z}+(r+\tau)\mathbb{Z}$ with norm $p^m$; this number is $0$ if $m$ is not listed in the table.
\begin{table}[htbp]
\centering
\begin{tabular}{c|c|c|c|c|}
$m$ & $2h$ & $2{\rm v}_p\left(f\right)$ & $2{\rm v}_p\left(f\right)+1$ & $>2{\rm v}_p\left(f\right)+1$\\
& $1\leq h<{\rm v}_p\left(f\right)$ & & &\\
\hline
$p$\textnormal{ is inert} &\multirow{3}{*}{$\varphi\left(p^{m/2}\right)$}& $p^{{\rm v}_p\left(f\right)}$ &\multicolumn{2}{|c|}{$0$}\\
\cline{1-1}\cline{3-4}
$p$\textnormal{ is ramified} &\multicolumn{2}{c|}{} & $p^{{\rm v}_p\left(f\right)}$ &\\
\cline{1-1}\cline{3-5}
$p$\textnormal{ splits} & & $p^{{\rm v}_p\left(f\right)-1}\left(p-2\right)$ &\multicolumn{2}{|c|}{$2\varphi\left(p^{{\rm v}_p\left(f\right)}\right)$} \\
\hline
\end{tabular}
\caption{Number of nontrivial invertible $\mathfrak{p}$-primary ideal atoms}\label{table1}
\end{table}

\item[\textnormal{5.}] The number of ideal atoms with radical $\mathfrak{p}$ is finite if and only if the number of invertible ideal atoms with radical $\mathfrak{p}$ is finite if and only if $p$ does not split.
\end{enumerate}
\end{theorem}

\begin{proof}
1. and 2. are an immediate consequence of Proposition~\ref{proposition 3.1}.

\smallskip
3. In 1. we have seen, that all $\mathfrak{p}$-primary ideals of $\mathcal{O}$ are of the form
$\mathfrak{q}=p^\ell(p^m\mathbb{Z}+(r+\tau)\mathbb{Z})$. If both $\ell$ and $m$ are greater than $0$, then $\mathfrak{q}$ is not an ideal atom. Indeed, $\mathfrak{q}=(p\mathcal{O})^\ell(p^m\mathbb{Z}+(r+\tau)\mathbb{Z})$ is a nontrivial factorization. It remains to be proven, that $p\mathcal{O}$ and $p^m\mathbb{Z}+(r+\tau)\mathbb{Z}$ are ideal atoms.

Assume that there exist proper ideals $\mathfrak{a}_1,\mathfrak{a}_2$ of $\mathcal{O}$ such that $p\mathcal{O}=\mathfrak{a}_1\mathfrak{a}_2$. Since $p\mathcal{O}$ is $\mathfrak{p}$-primary, we have $\mathfrak{a}_1$ and $\mathfrak{a}_2$ are $\mathfrak{p}$-primary. Using this information, we deduce, that $p\mathcal{O}\subset\mathfrak{p}^2$, implying
\[
p\in p\mathcal{O}\subset\mathfrak{p}^2=(p^2,pf\omega,f^2\omega^2)=p(p,f\omega,\frac{f}{p}\omega f\omega)=p(p,f\omega)=p\mathfrak{p}.
\]
Therefore, $1\in\mathfrak{p}$, a contradiction.

Assume that there exist proper ideals $\mathfrak{a}_1,\mathfrak{a}_2$ of $\mathcal{O}$ such that $p^m\mathbb{Z}+(r+\tau)\mathbb{Z}=\mathfrak{a}_1\mathfrak{a}_2$. Note that $\mathfrak{a}_1$ and $\mathfrak{a}_2$ are $\mathfrak{p}$-primary. By Proposition~\ref{proposition 3.2}.3, it follows that $p^m\mathbb{Z}+(r+\tau)\mathbb{Z}\subset p\mathcal{O}$, a contradiction to $r+\tau\not\in p\mathcal{O}$.

\smallskip
4. By 1. and 3., the nontrivial $\mathfrak{p}$-primary ideal atoms of norm $p^m$ are all $\mathfrak{q}=p^m\mathbb{Z}+(r+\tau)\mathbb{Z}$ with $m\in\mathbb{N}$, $0\leq r<p^m$ and $\mathcal N_{K/\mathbb{Q}}(r+\tau)\equiv 0\mod p^m$. By 2., an ideal of this form is invertible if and only if $\mathcal N_{K/\mathbb{Q}}(r+\tau)\not\equiv 0\mod p^{m+1}$.

Thus if we want to count the number of invertible $\mathfrak{p}$-primary ideal atoms of the form $\mathfrak{q}=p^m\mathbb{Z}+(r+\tau)\mathbb{Z}$ we have to count the number of solutions $r\in[0,p^m-1]$ of the equation
\begin{equation}\label{equation 5}
{\rm v}_p(\mathcal N_{K/\mathbb{Q}}(r+\tau))=m.
\end{equation}
Set $N=|\{r\in[0,p^m-1]\mid {\rm v}_p(\mathcal N_{K/\mathbb{Q}}(r+\tau))=m\}|$ and $a=\begin{cases} (\frac{f}{2})^2d_K &\textnormal{if }p=2\\ f^2d_K &\textnormal{if }p\not=2\end{cases}$. Next we show that $N=|\{r\in[0,p^m-1]\mid {\rm v}_p(r^2-a)=m\}|$. Note that $\mathcal N_{K/\mathbb{Q}}(r+\tau)=\frac{(2r+\varepsilon)^2-f^2d_K}{4}$ for each $r\in [0,p^m-1]$. If $p=2$, then $\varepsilon=0$, and hence $\mathcal N_{K/\mathbb{Q}}(r+\tau)=r^2-a$. Now let $p\not=2$. Then ${\rm v}_p(\mathcal N_{K/\mathbb{Q}}(r+\tau))={\rm v}_p((2r+\varepsilon)^2-a)$ for each $r\in [0,p^m-1]$. Let $f:\{r\in[0,p^m-1]\mid {\rm v}_p(r^2-a)=m\}\rightarrow\{r\in[0,p^m-1]\mid {\rm v}_p((2r+\varepsilon)^2-a)=m\}$ and $g:\{r\in[0,p^m-1]\mid {\rm v}_p((2r+\varepsilon)^2-a)=m\}\rightarrow\{r\in[0,p^m-1]\mid {\rm v}_p(r^2-a)=m\}$ be defined by $f(r)=\begin{cases}\frac{r-\varepsilon}{2} &\textnormal{if }r-\varepsilon\textnormal{ is even}\\\frac{r+p^m-\varepsilon}{2} &\textnormal{if }r-\varepsilon\textnormal{ is odd}\end{cases}$ and $g(r)={\rm rem}(2r+\varepsilon,p^m)$ for each $r\in [0,p^m-1]$. Observe that $f$ and $g$ are well-defined injective maps. Therefore, $N=|\{r\in[0,p^m-1]\mid {\rm v}_p(r^2-a)=m\}|$ in any case. Set $c=\frac{a}{p^{{\rm v}_p(a)}}$ and for $\ell\in\mathbb{N}$ set $g_{\ell}=|\{y\in [0,p^{\ell}-1]\mid y^2\equiv c\mod p^{\ell}\}|$. If $m<{\rm v}_p(a)$, then the statement follows immediately by Lemma~\ref{lemma 3.5}.1. Therefore, let $m\geq {\rm v}_p(a)$. In what follows we use Lemmas~\ref{lemma 3.4} and~\ref{lemma 3.5} without further citation.

\smallskip
CASE 1: $p=2$ and $2$ is inert. We have ${\rm v}_2(a)=2{\rm v}_2(f)-2$, $c\equiv d_K\equiv 5\mod 8$, $g_1=1$, $g_2=2$ and $g_3=0$. If $m={\rm v}_2(a)$, then $a$ is a square modulo $2^{m+1}$, and hence $N=2^{m/2-1}=\varphi(2^{m/2})$. If $m={\rm v}_2(a)+1$, then $a$ is a square modulo $2^m$, and thus $N=2^{{\rm v}_2(a)/2-1}(2g_1-g_2)=0$. If $m={\rm v}_2(a)+2$, then $a$ is a square modulo $2^m$, whence $N=2^{{\rm v}_2(a)/2-1}(2g_2-g_3)=2^{{\rm v}_2(a)/2+1}=2^{{\rm v}_2(f)}$. Finally, let $m\geq {\rm v}_2(a)+3$. Then $a$ is not a square modulo $2^m$, and hence $N=0$.

\smallskip
CASE 2: $p=2$ and $2$ is ramified. Note that ${\rm v}_2(a)\in\{2{\rm v}_2(f),2{\rm v}_2(f)+1\}$. First let ${\rm v}_2(a)=2{\rm v}_2(f)$. Then $a=f^2d$ with $c\equiv d\equiv 3\mod 4$, $g_1=1$ and $g_{\ell}=0$ for each $\ell\in\mathbb{N}_{\geq 2}$. If $m={\rm v}_2(a)$, then $a$ is a square modulo $2^{m+1}$, and thus $N=2^{m/2-1}=2^{{\rm v}_2(f)-1}=\varphi(2^{{\rm v}_2(f)})$. If $m={\rm v}_2(a)+1$, then $a$ is a square modulo $2^m$, and hence $N=2^{{\rm v}_2(a)/2-1}(2g_1-g_2)=2^{{\rm v}_2(f)}$. Finally, let $m\geq {\rm v}_2(a)+2$. Then $a$ is not a square modulo $2^m$, and thus $N=0$.

Now let ${\rm v}_2(a)=2{\rm v}_2(f)+1$. If $m={\rm v}_2(a)$, then $a$ is not a square modulo $2^{m+1}$, and hence $N=2^{\lfloor m/2\rfloor}=2^{{\rm v}_2(f)}$. If $m>{\rm v}_2(a)$, then $a$ is not a square modulo $2^m$, and thus $N=0$.

\smallskip
CASE 3: $p=2$ and $2$ splits. Observe that ${\rm v}_2(a)=2{\rm v}_2(f)-2$, $c\equiv d_K\equiv 1\mod 8$, $g_1=1$, $g_2=2$ and $g_{\ell}=4$ for each $\ell\in\mathbb{N}_{\geq 3}$. If $m={\rm v}_2(a)$, then $a$ is a square modulo $2^{m+1}$, and hence $N=2^{m/2-1}=\varphi(2^{m/2})$. Now let $m>{\rm v}_2(a)$ and set $k=m-{\rm v}_2(a)$. Note that $a$ is a square modulo $2^m$, and hence $N=2^{{\rm v}_2(a)/2-1}(2g_k-g_{k+1})$. If $m<{\rm v}_2(a)+3$, then $N=0$. Finally, let $m\geq {\rm v}_2(a)+3$. Then $N=2^{{\rm v}_2(a)/2+1}=2^{{\rm v}_2(f)}=2\varphi(2^{{\rm v}_2(f)})$.

\smallskip
CASE 4: $p\not=2$ and $p$ is inert. We have ${\rm v}_p(a)=2{\rm v}_p(f)$, $(\frac{c}{p})=(\frac{d_K}{p})=-1$ and $g_{\ell}=0$ for each $\ell\in\mathbb{N}$. If $m={\rm v}_p(a)$, then $a$ is not a square modulo $p^{m+1}$, and hence $N=p^{\lfloor m/2\rfloor}=p^{{\rm v}_p(f)}$. If $m>{\rm v}_p(a)$, then $a$ is not a square modulo $p^m$, and thus $N=0$.

\smallskip
CASE 5: $p\not=2$ and $p$ is ramified. It follows that ${\rm v}_p(a)=2{\rm v}_p(f)+1$. If $m={\rm v}_p(a)$, then $a$ is not a square modulo $p^{m+1}$, and thus $N=p^{\lfloor m/2\rfloor}=p^{{\rm v}_p(f)}$. If $m>{\rm v}_p(a)$, then $a$ is not a square modulo $p^m$, and thus $N=0$.

\smallskip
CASE 6: $p\not=2$ and $p$ splits. Note that ${\rm v}_p(a)=2{\rm v}_p(f)$, $(\frac{c}{p})=(\frac{d_K}{p})=1$ and $g_{\ell}=2$ for each $\ell\in\mathbb{N}$. If $m={\rm v}_p(a)$, then $a$ is a square modulo $p^{m+1}$, and hence $N=p^{m/2-1}(p-2)=p^{{\rm v}_p(f)-1}(p-2)$. If $m>{\rm v}_p(a)$, then $a$ is a square modulo $p^m$, and thus $N=p^{{\rm v}_p(a)/2-1}(pg_k-g_{k+1})=2p^{{\rm v}_p(f)-1}(p-1)=2\varphi(p^{{\rm v}_p(f)})$.

\smallskip
5. It is an immediate consequence of 4. that the number of invertible ideal atoms with radical $\mathfrak{p}$ is finite if and only if $p$ does not split. It remains to show that $\mathcal{A}(\mathcal{I}_p(\mathcal{O}))$ is finite if and only if $\mathcal{A}(\mathcal{I}^*_p(\mathcal{O}))$ is finite. It follows from \cite[Theorem 4.3]{An-Mo92} that $\mathcal{I}(\mathcal{O}_{\mathfrak{p}})$ is a finitely generated monoid if and only if $\mathcal{I}^*(\mathcal{O}_{\mathfrak{p}})$ is a finitely generated monoid. Therefore, Proposition~\ref{proposition 3.3}.2 implies that $\mathcal{I}_p(\mathcal{O})$ is a finitely generated monoid if and only if $\mathcal{I}^*_p(\mathcal{O})$ is a finitely generated monoid. Observe that $\mathcal{I}_p(\mathcal{O})$ and $\mathcal{I}^*_p(\mathcal{O})$ are atomic monoids. Therefore, $\mathcal{A}(\mathcal{I}_p(\mathcal{O}))$ is finite if and only if $\mathcal{I}_p(\mathcal{O})$ is a finitely generated monoid if and only if $\mathcal{I}^*_p(\mathcal{O})$ is a finitely generated monoid if and only if $\mathcal{A}(\mathcal{I}^*_p(\mathcal{O}))$ is finite.
\end{proof}

\smallskip
\section{Sets of distances and sets of catenary degrees}\label{4}
\smallskip

The goal in this section is to prove Theorem~\ref{theorem 1.1}. The proof is based on the precise description of ideals given in Theorem~\ref{theorem 3.6}. We proceed in a series of lemmas and propositions and use all notation on orders as introduced at the beginning of Section~\ref{3}. In particular, $\mathcal{O} =\mathcal O_f$ is an order in a quadratic number with conductor $f\mathcal O_K$ for some $f\in\mathbb{N}_{\ge 2}$.

\begin{proposition}\label{proposition 4.1}
Let $H$ be a reduced atomic monoid and suppose there is a cancellative atom $u\in\mathcal{A}(H)$ such that for each $a\in H\setminus H^{\times}$ there are $n\in\mathbb{N}_0$ and $v\in\mathcal{A}(H)$ such that $a=u^nv$.
\begin{enumerate}
\item[\textnormal{1.}] For all $n,m\in\mathbb{N}_0$ and $v,w\in\mathcal{A}(H)$ such that $u^nv=u^mw$, it follows that $n=m$ and $v=w$.

\item[\textnormal{2.}] For all $n\in\mathbb{N}_0$ and $v\in\mathcal{A}(H)$, it follows that $\max\mathsf{L}(u^nv)=n+1$.

\item[\textnormal{3.}] $\mathsf{c}(H)=\sup\{\mathsf{c}(w\cdot y,u^n\cdot v)\mid n\in\mathbb{N}$ and $v,w,y\in\mathcal{A}(H)$ such that $wy=u^nv\}$.

\item[\textnormal{4.}] If $H$ is half-factorial, then $\mathsf{c}(H)\leq 2$.

\item[\textnormal{5.}] $\sup\Delta(H)=\sup\{\ell-2\mid\ell\in\mathbb{N}_{\geq 3}$ such that $\mathsf{L}(vw)\cap [2,\ell]=\{2,\ell\}$ for some $v,w\in\mathcal{A}(H)\}$.
\end{enumerate}
\end{proposition}

\begin{proof}
1. Let $n,m\in\mathbb{N}_0$ and $v,w\in\mathcal{A}(H)$ be such that $u^nv=u^mw$. Without restriction let $n\leq m$. Since $u$ is cancellative, we infer that $v=u^{m-n}w$. Since $v\in\mathcal{A}(H)$, we have $n=m$, and thus $v=w$.

\smallskip
2. It is clear that $n+1\in\mathsf{L}(u^nv)$ for all $n\in\mathbb{N}_0$ and $v\in\mathcal{A}(H)$. Therefore, it is sufficient to show by induction that for all $n\in\mathbb{N}_0$ and $v\in\mathcal{A}(H)$, $\max\mathsf{L}(u^nv)\leq n+1$. Let $n\in\mathbb{N}_0$ and $v\in\mathcal{A}(H)$. If $n=0$, then the assertion is obviously true. Now let $n>0$ and $z\in\mathsf{Z}(u^nv)$. Then there are some $z^{\prime},z^{\prime\prime}\in\mathsf{Z}(H)\setminus\{1\}$ such that $z=z^{\prime}\cdot z^{\prime\prime}$. There are some $m^{\prime},m^{\prime\prime}\in\mathbb{N}_0$ and $w^{\prime},w^{\prime\prime}\in\mathcal{A}(H)$ such that $\pi(z^{\prime})=u^{m^{\prime}}w^{\prime}$ and $\pi(z^{\prime\prime})=u^{m^{\prime\prime}}w^{\prime\prime}$. There are some $\ell\in\mathbb{N}$ and $y\in\mathcal{A}(H)$ such that $w^{\prime}w^{\prime\prime}=u^{\ell}y$. We infer that $u^nv=u^{m^{\prime}+m^{\prime\prime}+\ell}y$, and thus $n=m^{\prime}+m^{\prime\prime}+\ell$ by 1. Since $m^{\prime},m^{\prime\prime}<n$, it follows by the induction hypothesis that $|z^{\prime}|\leq m^{\prime}+1$ and $|z^{\prime\prime}|\leq m^{\prime\prime}+1$. Consequently, $|z|\leq m^{\prime}+m^{\prime\prime}+2\leq m^{\prime}+m^{\prime\prime}+\ell+1=n+1$.

\smallskip
3. Set $k=\sup\{\mathsf{c}(w\cdot y,u^n\cdot v)\mid n\in\mathbb{N}_0$ and $v,w,y\in\mathcal{A}(H)$ such that $wy=u^nv\}$. Since $\mathsf{c}(H)=\sup\{\mathsf{c}(z,z^{\prime})\mid a\in H, z,z^{\prime}\in\mathsf{Z}(a)\}$, it is obvious that $k\leq\mathsf{c}(H)$. It remains to show by induction that for all $n\in\mathbb{N}_0$ and $v\in\mathcal{A}(H)$, it follows that $\mathsf{c}(u^nv)\leq k$. Let $n\in\mathbb{N}_0$ and $v\in\mathcal{A}(H)$. Since $\mathsf{c}(v)=0$, we can assume without restriction that $n>0$. Since $\mathsf{c}(u^nv)=\sup\{\mathsf{c}(z,u^n\cdot v)\mid z\in\mathsf{Z}(u^nv)\}$, it remains to show that $\mathsf{c}(z,u^n\cdot v)\leq k$ for all $z\in\mathsf{Z}(u^nv)$. Let $z\in\mathsf{Z}(u^nv)$.

\smallskip
CASE 1: For all $w,y\in\mathcal{A}(H)\setminus\{u\}$, we have $w\cdot y\nmid z$. There are some $m\in\mathbb{N}$ and $w\in\mathcal{A}(H)$ such that $z=u^m\cdot w$. We infer by 1. that $z=u^n\cdot v$, and thus $\mathsf{c}(z,u^n\cdot v)=0\leq k$.

\smallskip
CASE 2: There are some $w,y\in\mathcal{A}(H)\setminus\{u\}$ such that $w\cdot y\mid z$. Set $z^{\prime}=\frac{z}{w\cdot y}$. There exist $m\in\mathbb{N}$ and $a\in\mathcal{A}(H)$ such that $wy=u^ma$. We infer that $m\leq n$ and $u^nv=\pi(z)=\pi(w\cdot y)\pi(z^{\prime})=u^ma\pi(z^{\prime})$, and thus $a\pi(z^{\prime})=u^{n-m}v$. Observe that $\mathsf{c}(z,u^m\cdot a\cdot z^{\prime})\leq\mathsf{c}(w\cdot y,u^m\cdot a)\leq k$. Since $n-m<n$, it follows by the induction hypothesis that $\mathsf{c}(u^m\cdot a\cdot z^{\prime},u^n\cdot v)\leq\mathsf{c}(a\cdot z^{\prime},u^{n-m}\cdot v)\leq k$, and hence $\mathsf{c}(z,u^n\cdot v)\leq k$.

\smallskip
4. Let $H$ be half-factorial, $n\in\mathbb{N}$ and $v,w,y\in\mathcal{A}(H)$ be such that $wy=u^nv$. We infer that $n=1$, and thus $\mathsf{c}(w\cdot y,u^n\cdot v)\leq\mathsf{d}(w\cdot y,u\cdot v)\leq 2$. Therefore, $\mathsf{c}(H)\leq 2$ by 3.

\smallskip
5. Set $N=\sup\{\ell-2\mid\ell\in\mathbb{N}_{\geq 3}$ such that $\mathsf{L}(vw)\cap [2,\ell]=\{2,\ell\}$ for some $v,w\in\mathcal{A}(H)\}$. It is obvious that $N\leq\sup\Delta(H)$. It remains to show that $k\leq N$ for each $k\in\Delta(H)$. Let $k\in\Delta(H)$. Then there are some $a\in H$ and $r,s\in\mathsf{L}(a)$ such that $r<s$, $\mathsf{L}(a)\cap [r,s]=\{r,s\}$, and $k=s-r$. Let $z\in\mathsf{Z}(a)$ with $|z|=r$ be such that ${\rm v}_u(z)=\max\{{\rm v}_u(z^{\prime})\mid z^{\prime}\in\mathsf{Z}(a) \ \text{with} \ |z^{\prime}|=r\}$. Since $r<\max\mathsf{L}(a)$, it follows by 2., that there are some $v,w\in\mathcal{A}(H)\setminus\{u\}$ such that $v\cdot w\mid z$. There are some $n\in\mathbb{N}$ and $y\in\mathcal{A}(H)$ such that $vw=u^ny$. Since ${\rm v}_u(z)$ is maximal amongst all factorizations of $a$ of length $r$, we have $n\geq 2$. Consequently, there is some $\ell\in\mathsf{L}(vw)$ such that $2<\ell\leq n+1$ and $\mathsf{L}(vw)\cap [2,\ell]=\{2,\ell\}$. Note that $r+\ell-2\in\mathsf{L}(a)$, and thus $s\leq r+\ell-2$. This implies that $k\leq\ell-2\leq N$.
\end{proof}

Theorem~\ref{theorem 3.6} implies that, for all prime divisors $p$ of $f$, $\mathcal{I}^*_p(\mathcal{O}_f)$ and $\mathcal{I}_p(\mathcal{O}_f)$ are reduced atomic monoids satisfying the assumption in Proposition~\ref{proposition 4.1}.

\begin{lemma}\label{lemma 4.2}
Let $p$ be a prime divisor of $f$.
\begin{enumerate}
\item[\textnormal{1.}] $\mathsf{Z}(pP_{f,p})=\{A\cdot P_{f,p}\mid A=P_{f,p}$ or $A\in\mathcal{A}(\mathcal{I}^*_p(\mathcal{O}_f))$ such that $\mathcal{N}(A)=p^2\}$ and $1\in {\rm Ca}(\mathcal{I}_p(\mathcal{O}_f))$.

\item[\textnormal{2.}] If $I,J\in\mathcal{A}(\mathcal{I}^*_p(\mathcal{O}_f))$ are such that $\mathcal{N}(I)=p^2$ and $\mathcal{N}(J)>p^2$, then $IJ=pL$ for some $L\in\mathcal{A}(\mathcal{I}^*_p(\mathcal{O}_f))$.

\item[\textnormal{3.}] $2\in {\rm Ca}(\mathcal{I}^*_p(\mathcal{O}_f))$.
\end{enumerate}
\end{lemma}

\begin{proof}
1. Note that $\{I\in\mathcal{I}_p(\mathcal{O}_f)\mid\mathcal{N}(I)=p\}=\{P_{f,p}\}$. First we show that $\mathsf{Z}(pP_{f,p})=\{A\cdot P_{f,p}\mid A=P_{f,p}$ or $A\in\mathcal{A}(\mathcal{I}^*_p(\mathcal{O}_f))$ such that $\mathcal{N}(A)=p^2\}$.

Let $z\in\mathsf{Z}(pP_{f,p})$. It follows from Proposition~\ref{proposition 4.1}.2 that $|z|\leq 2$, and hence $|z|=2$. Consequently, $z=A\cdot B$ for some $A,B\in\mathcal{A}(\mathcal{I}_p(\mathcal{O}_f))$. By Proposition~\ref{proposition 3.2}.1 there are some $(u,v,w),(x,y,t)\in\mathcal{M}_{f,p}$ such that $A=p^u(p^v\mathbb{Z}+(w+\tau)\mathbb{Z})$ and $B=p^x(p^y\mathbb{Z}+(t+\tau)\mathbb{Z})$. Set $g=\min\{v,y,{\rm v}_p(w+t+\varepsilon)\}$ and $e=\min\{g,{\rm v}_p(w-t),{\rm v}_p(w^2+\varepsilon w+\eta)-v,{\rm v}_p(t^2+\varepsilon t+\eta)-y\}$. We infer by Proposition~\ref{proposition 3.2}.1 that $u+x+g=1$ and $v+y+e-2g=1$. Note that $g\in\{0,1\}$. If $g=0$, then $u+x=v+y=1$, and thus ($A=p\mathcal{O}_f$ and $B=P_{f,p}$) or ($A=P_{f,p}$ and $B=p\mathcal{O}_f$). Now let $g=1$. Then $u=x=0$, $v,y\geq 1$, $v+y+e=3$, and $e\in\{0,1\}$. If $e=1$, then $v=y=1$, and thus $A=B=P_{f,p}$. Now let $e=0$. Then ($v=1$ and $y=2$) or ($v=2$ and $y=1$). Without restriction let $v=2$ and $y=1$. Then $B=P_{f,p}$, $\mathcal{N}(A)=p^v=p^2$, and $\mathcal{N}(A)\mathcal{N}(B)=p^3=\mathcal{N}(pP_{f,p})=\mathcal{N}(AB)$. Since $B$ is not invertible, it follows by Proposition~\ref{proposition 3.2}.3 that $A$ is invertible.

To prove the converse inclusion note that $P_{f,p}=p\mathbb{Z}+(r+\tau)\mathbb{Z}$ for some $(0,1,r)\in\mathcal{M}_{f,p}$. By Proposition~\ref{proposition 3.2}.1 we have $P_{f,p}^2=p^a(p^b\mathbb{Z}+(c+\tau)\mathbb{Z}$ with $(a,b,c)\in\mathcal{M}_{f,p}$, $a=\min\{1,{\rm v}_p(2r+\varepsilon)\}$ and $b=2+e-2a$ with $e=\min\{a,{\rm v}_p(r^2+\varepsilon r+\eta)-1\}$. By Proposition~\ref{proposition 3.2}.3 we have $a>0$, and thus $a=b=e=1$. Consequently, $P_{f,p}^2=pP_{f,p}$. Now let $A\in\mathcal{A}(\mathcal{I}^*_p(\mathcal{O}_f))$ be such that $\mathcal{N}(A)=p^2$. It follows by Proposition~\ref{proposition 3.2}.3 that $\mathcal{N}(AP_{f,p})=\mathcal{N}(A)\mathcal{N}(P_{f,p})=p^3$ and $AP_{f,p}=pI$ for some $I\in\mathcal{I}_p(\mathcal{O}_f)$. We infer that $\mathcal{N}(I)=p$, and hence $I=P_{f,p}$.

Observe that $\mathsf{d}(z^{\prime},z^{\prime\prime})\leq 1$ for all $z^{\prime},z^{\prime\prime}\in\mathsf{Z}(pP_{f,p})$ and $(p\mathcal{O}_f)\cdot P_{f,p}$ and $P_{f,p}^2$ are distinct factorizations of $pP_{f,p}$. Therefore, $1=\mathsf{c}(pP_{f,p})\in {\rm Ca}(\mathcal{I}_p(\mathcal{O}_f))$.

\smallskip
2. Let $I,J\in\mathcal{A}(\mathcal{I}^*_p(\mathcal{O}_f))$ be such that $\mathcal{N}(I)=p^2$ and $\mathcal{N}(J)>p^2$. Without restriction we can assume that $I\not=p\mathcal{O}_f$. There are some $(0,2,r),(0,k,s)\in\mathcal{M}_{f,p}$ such that $I=p^2\mathbb{Z}+(r+\tau)\mathbb{Z}$ and $J=p^k\mathbb{Z}+(s+\tau)\mathbb{Z}$. Since $I$ and $J$ are invertible, we have ${\rm v}_p(r^2+\varepsilon r+\eta)=2$ and ${\rm v}_p(s^2+\varepsilon s+\eta)=k>2$. Therefore, ${\rm v}_p(r+s+\varepsilon)+{\rm v}_p(r-s)={\rm v}_p(r^2+\varepsilon r+\eta-(s^2+\varepsilon s+\eta))=2$, and thus ${\rm v}_p(r+s+\varepsilon)=1$, by Proposition~\ref{proposition 3.2}.2. Therefore, $\min\{2,k,{\rm v}_p(r+s+\varepsilon)\}=1$, and hence $IJ=pL$ for some $L\in\mathcal{A}(\mathcal{I}^*_p(\mathcal{O}_f))$ by Proposition~\ref{proposition 3.2}.1.

\smallskip
3. We distinguish two cases.

CASE 1: $p\not=2$ or ${\rm v}_p(f)\geq 2$ or $d\not\equiv 1\mod 8$. It follows from Theorem~\ref{theorem 3.6} that there is some $I\in\mathcal{A}(\mathcal{I}^*_p(\mathcal{O}_f))$ such that $\mathcal{N}(I)=p^2$ and $I\not=p\mathcal{O}_f$. We have $I\overline{I}=(p\mathcal{O}_f)^2$, and hence $\mathsf{L}(I\overline{I})=\{2\}$. Since $I\cdot\overline{I}$ and $(p\mathcal{O}_f)\cdot (p\mathcal{O}_f)$ are distinct factorizations of $I\overline{I}$, we have $2=\mathsf{c}(I\overline{I})\in {\rm Ca}(\mathcal{I}^*_p(\mathcal{O}_f))$.

\smallskip
CASE 2: $p=2$, ${\rm v}_p(f)=1$ and $d\equiv 1\mod 8$. By Proposition~\ref{proposition 3.3}.3 we can assume without restriction that $f=2$. By Theorem~\ref{theorem 3.6} there is some $I\in\mathcal{A}(\mathcal{I}^*_2(\mathcal{O}_f))$ such that $\mathcal{N}(I)=8$. There is some $(0,3,r)\in\mathcal{M}_{f,2}$ such that $I=8\mathbb{Z}+(r+\tau)\mathbb{Z}$. We have ${\rm v}_2(r^2-d)=3$, and hence ${\rm v}_2(r)=0$. Therefore, $\min\{3,{\rm v}_2(2r)\}=1$, and thus $I^2=2J$ for some $J\in\mathcal{A}(\mathcal{I}^*_2(\mathcal{O}_f))$. Consequently, $\mathsf{L}(I^2)=\{2\}$. Since $I\cdot I$ and $(2\mathcal{O}_f)\cdot J$ are distinct factorizations of $I^2$, it follows that $2=\mathsf{c}(I^2)\in {\rm Ca}(\mathcal{I}^*_p(\mathcal{O}_f))$.
\end{proof}

\begin{proposition}\label{proposition 4.3}
Let $p$ be an odd prime divisor of $f$ such that ${\rm v}_p(f)\geq 2$.
\begin{enumerate}
\item[\textnormal{1.}] There is a $C\in\mathcal{A}(\mathcal{I}^*_p(\mathcal{O}_f))$ such that $\mathsf{L}(C^2)=\{2,3\}$ whence $1\in\Delta(\mathcal{I}^*_p(\mathcal{O}_f))$ and $3\in {\rm Ca}(\mathcal{I}^*_p(\mathcal{O}_f))$. Moreover, if $(p\not=3$ or $d\not\equiv 2\mod 3$ or ${\rm v}_p(f)>2)$, then there are $I,J,L\in\mathcal{A}(\mathcal{I}^*_p(\mathcal{O}_f))$ such that $I^2=p^2J$ and $J^2=p^2L$.

\item[\textnormal{2.}] If $|{\rm Pic}(\mathcal{O}_f)|\leq 2$ and $(p\not=3$ or $d\not\equiv 2\mod 3$ or ${\rm v}_p(f)>2)$, then there is a nonzero primary $a\in\mathcal{O}_f$ such that $2,3\in\mathsf{L}(a)$ whence $1\in\Delta(\mathcal{O}_f)$.
\end{enumerate}
\end{proposition}

\begin{proof}
1. By Proposition~\ref{proposition 3.3}.3 there is a monoid isomorphism $\delta:\mathcal{I}^*_p(\mathcal{O}_f)\rightarrow\mathcal{I}^*_p(\mathcal{O}_{\frac{f}{2^{{\rm v}_2(f)}}})$ such that $\delta(p\mathcal{O}_f)=p\mathcal{O}_{\frac{f}{2^{{\rm v}_2(f)}}}$. Therefore, we can assume without restriction that $f$ is odd.

\smallskip
CLAIM: $\mathsf{L}(I^2)=\{2,3\}$ for some $I\in\mathcal{A}(\mathcal{I}^*_p(\mathcal{O}_f))$, $1\in\Delta(\mathcal{I}^*_p(\mathcal{O}_f))$, $3\in {\rm Ca}(\mathcal{I}^*_p(\mathcal{O}_f))$ and if ${\rm v}_p(p^4+f^2d)=4$, then $I^2=p^2J$ and $J^2=p^2L$ for some $I,J,L\in\mathcal{A}(\mathcal{I}^*_p(\mathcal{O}_f))$.

\smallskip
For $r\in\mathbb{N}_0$ set $k={\rm v}_p(\mathcal{N}_{K/\mathbb{Q}}(r+\tau))$ and $I=p^k\mathbb{Z}+(r+\tau)\mathbb{Z}$. Let $k>0$ and $r<p^k$. Then $I\in\mathcal{A}(\mathcal{I}^*_p(\mathcal{O}_f))$. Moreover, $I^2=p^a(p^b\mathbb{Z}+(c+\tau)\mathbb{Z})$ with $a=\min\{k,{\rm v}_p(2r+\varepsilon)\}$, $b=2(k-a)$ and $c={\rm rem}(r-t\frac{\mathcal{N}_{K/\mathbb{Q}}(r+\tau)}{p^a},p^b)$ for each $t\in\mathbb{Z}$ with $t\frac{2r+\varepsilon}{p^a}\equiv 1\mod p^{k-a}$. Set $J=p^b\mathbb{Z}+(c+\tau)\mathbb{Z}$. Then $I^2=p^aJ$ and if $b>0$, then $J\in\mathcal{A}(\mathcal{I}^*_p(\mathcal{O}_f))$. In particular, if $a=2$ and $b>0$, then $I,J\in\mathcal{A}(\mathcal{I}^*_p(\mathcal{O}_f))$ and $\mathsf{L}(I^2)=\{2,3\}$, and hence $1\in\Delta(I^2)\subseteq\Delta(\mathcal{I}^*_p(\mathcal{O}_f))$ and $3=\mathsf{c}(I^2)\in {\rm Ca}(\mathcal{I}^*_p(\mathcal{O}_f))$. Observe that $J^2=p^{a^{\prime}}(p^{b^{\prime}}\mathbb{Z}+(c^{\prime}+\tau)\mathbb{Z})$ with $a^{\prime}=\min\{b,{\rm v}_p(2c+\varepsilon)\}$, $b^{\prime}=2(b-a^{\prime})$ and $c^{\prime}\in\mathbb{N}_0$ such that $c^{\prime}<p^{b^{\prime}}$. Set $L=p^{b^{\prime}}\mathbb{Z}+(c^{\prime}+\tau)\mathbb{Z}$. Then $J^2=p^{a^{\prime}}L$ and if $b^{\prime}>0$, then $L\in\mathcal{A}(\mathcal{I}^*_p(\mathcal{O}_f))$.

\smallskip
CASE 1: $d\not\equiv 1\mod 4$. Set $r=p^2$. We have $\mathcal{N}_{K/\mathbb{Q}}(r+\tau)=p^4-f^2d$, $k\geq 4$, $a=2$, $b=2(k-2)>0$, $r<p^k$, and $t=\frac{p^{k-2}+1}{2}$ satisfies the congruence. Therefore, $c={\rm rem}(p^2-\frac{(p^{k-2}+1)(p^4-f^2d)}{2p^2},p^{2(k-2)})=\frac{p^4+f^2d+p^{k-2}f^2d-p^{k+2}+2\ell p^{2(k-1)}}{2p^2}$ for some $\ell\in\mathbb{Z}$. For the rest of this case let ${\rm v}_p(p^4+f^2d)=4$. It follows that ${\rm v}_p(c)=2$, and hence $a^{\prime}=\min\{2(k-2),{\rm v}_p(2c)\}=2$ and $b^{\prime}=4(k-3)>0$.

\smallskip
CASE 2: $d\equiv 1\mod 4$. Set $r=\frac{p^2-1}{2}$. Observe that $\mathcal{N}_{K/\mathbb{Q}}(r+\tau)=\frac{p^4-f^2d}{4}$, $k\geq 4$, $a=2$, $b=2(k-2)>0$, $r<p^k$, and $t=1$ satisfies the congruence. Consequently, $2c+\varepsilon=2{\rm rem}(\frac{p^2-1}{2}-\frac{p^4-f^2d}{4p^2},p^{2(k-2)})+1=\frac{p^4+f^2d+4\ell p^{2(k-1)}}{2p^2}$ for some $\ell\in\mathbb{Z}$. For the rest of this case let ${\rm v}_p(p^4+f^2d)=4$. We infer that $a^{\prime}=\min\{2(k-2),{\rm v}_p(2c+\varepsilon)\}=2$. Moreover, $b^{\prime}=4(k-3)>0$. This proves the claim.

\smallskip
Note that if $g\in\mathbb{N}$ with ${\rm v}_p(g)={\rm v}_p(f)$, then there is a monoid isomorphism $\alpha:\mathcal{I}^*_p(\mathcal{O}_f)\rightarrow\mathcal{I}^*_p(\mathcal{O}_g)$ such that $\alpha(p\mathcal{O}_f)=p\mathcal{O}_g$ by Proposition~\ref{proposition 3.3}.3. By the claim it remains to show that if $(p\not=3$ or $d\not\equiv 2\mod 3$ or ${\rm v}_p(f)>2)$, then there is some odd $g\in\mathbb{N}$ such that ${\rm v}_p(g)={\rm v}_p(f)$ and ${\rm v}_p(p^4+g^2d)=4$.

Let $(p\not=3$ or $d\not\equiv 2\mod 3$ or ${\rm v}_p(f)>2)$. Furthermore, let ${\rm v}_p(p^4+f^2d)>4$. This implies that ${\rm v}_p(f)=2$ and $p\nmid d$. Without restriction we can assume that ${\rm v}_p(p^4+(p^2)^2d)>4$. We have ${\rm v}_p(1+d)>0$, and hence $p\not=3$. Set $g=(p-2)p^2$. Then ${\rm v}_p(g)={\rm v}_p(f)$. Assume that ${\rm v}_p(p^4+g^2d)>4$. Then $p^5\mid p^4+(p-2)^2p^4d-p^4(1+d)$, and thus $p\mid (p-2)^2-1=p^2-4p+3$. It follows that $p=3$, a contradiction.

\smallskip
2. Let $|{\rm Pic}(\mathcal{O}_f)|\leq 2$ and let $p\not=3$ or $d\not\equiv 2\mod 3$ or ${\rm v}_p(f)>2$. By 1. there are some $I,J,L\in\mathcal{A}(\mathcal{I}^*_p(\mathcal{O}_f))$ such that $I^2=p^2J$ and $J^2=p^2L$. We infer that $I^2$ is principal, and hence $J$ and $L$ are principal. Consequently, there are some $u,v\in\mathcal{A}(\mathcal{O}_f)$ such that $J=u\mathcal{O}_f$, $L=v\mathcal{O}_f$ and $u^2=p^2v$. Note that $u^2$ is primary. Since $p\in\mathcal{A}(\mathcal{O}_f)$, we have $2,3\in\mathsf{L}(u^2)$. Therefore, $1\in\Delta(\mathcal{O}_f)$.
\end{proof}

\begin{proposition}\label{proposition 4.4}
Let $p$ be a prime divisor of $f$ such that ${\rm v}_p(f)\geq 2$. Then there are $I,J\in\mathcal{A}(\mathcal{I}^*_p(\mathcal{O}_f))$ such that $\mathsf{L}(IJ)=\{2,4\}$ whence $2\in\Delta(\mathcal{I}^*_p(\mathcal{O}_f))$ and $4\in {\rm Ca}(\mathcal{I}^*_p(\mathcal{O}_f))$.
\end{proposition}

\begin{proof}
CASE 1: $p\not=2$ or ${\rm v}_p(f)>2$ or $d\not\equiv 1\mod 8$. By Theorem~\ref{theorem 3.6} there is some $I\in\mathcal{A}(\mathcal{I}^*_p(\mathcal{O}_f))$ such that $\mathcal{N}(I)=p^4$. Set $J=\overline{I}$. We infer that $IJ=(p\mathcal{O}_f)^4$, and hence $\{2,4\}\subset\mathsf{L}(IJ)\subset\{2,3,4\}$. Assume that $3\in\mathsf{L}(IJ)$. Then there are some $A,B,C\in\mathcal{A}(\mathcal{I}^*_p(\mathcal{O}_f))$ such that $IJ=ABC$ and $\mathcal{N}(A)\leq\mathcal{N}(B)\leq\mathcal{N}(C)$. Again by Theorem~\ref{theorem 3.6} we have $\mathcal{N}(L)\in\{p^2\}\cup\{p^n\mid n\in\mathbb{N}_{\geq 4}\}$ for all $L\in\mathcal{A}(\mathcal{I}^*_p(\mathcal{O}_f))$. This implies that $\mathcal{N}(A)=\mathcal{N}(B)=p^2$ and $\mathcal{N}(C)=p^4$. It follows by Lemma~\ref{lemma 4.2}.2 that $ABC=p^2L$ for some $L\in\mathcal{A}(\mathcal{I}^*_p(\mathcal{O}_f))$. Consequently, $L=p^2\mathcal{O}_f$, a contradiction. We infer that $\mathsf{L}(IJ)=\{2,4\}$ whence $2\in\Delta(\mathcal{I}^*_2(\mathcal{O}_f))$ and $4\in {\rm Ca}(\mathcal{I}^*_2(\mathcal{O}_f))$.

\smallskip
CASE 2: $p=2$, ${\rm v}_p(f)=2$ and $d\equiv 1\mod 8$. Since $\mathcal{I}^*_2(\mathcal{O}_4)\cong\mathcal{I}^*_2(\mathcal{O}_f)$ by Proposition~\ref{proposition 3.3}.3, we can assume without restriction that $f=4$. We set
\[
w=\begin{cases}6 &\textnormal{if }d\equiv 1\mod\textnormal{ } 16\\2 &\textnormal{if }d\equiv 9\mod\textnormal{ } 16\end{cases}\quad\textnormal{and}\quad
z=\begin{cases}18 &\textnormal{if }d\equiv 1\mod\textnormal{ } 32\\22 &\textnormal{if }d\equiv 9\mod\textnormal{ } 32\\2 &\textnormal{if }d\equiv 17\mod\textnormal{ } 32\\6 &\textnormal{if }d\equiv 25\mod\textnormal{ } 32\end{cases}.
\]
In any case, we have ${\rm v}_2(\mathcal{N}_{K/\mathbb{Q}}(w+\tau))=5$ and ${\rm v}_2(\mathcal{N}_{K/\mathbb{Q}}(z+\tau))=6$. Set $I=32\mathbb{Z}+(w+\tau)\mathbb{Z}$ and $J=64\mathbb{Z}+(z+\tau)\mathbb{Z}$. Then $I,J\in\mathcal{A}(\mathcal{I}^*_2(\mathcal{O}_4))$ and Proposition~\ref{proposition 3.2}.1 implies that $IJ=2^a(2^b\mathbb{Z}+(c+\tau)\mathbb{Z})$ with $a=\min\{5,6,{\rm v}_2(w+z)\}$, $b=5+6-2a$ and $c\in\mathbb{N}_0$ such that $c<2^b$. Observe that ${\rm v}_2(w+z)=3$, and thus $a=3$ and $b=5$. Set $L=32\mathbb{Z}+(c+\tau)\mathbb{Z}$. Then $L\in\mathcal{A}(\mathcal{I}^*_2(\mathcal{O}_4))$ and $IJ=(2\mathcal{O}_4)^3L$. We infer that $\{2,4\}\subset\mathsf{L}(IJ)\subset\{2,3,4\}$, by Proposition~\ref{proposition 4.1}.2.

Assume that $3\in\mathsf{L}(IJ)$. Then there are some $A,B,C\in\mathcal{A}(\mathcal{I}^*_2(\mathcal{O}_4))$ such that $IJ=ABC$ and $\mathcal{N}(A)\leq\mathcal{N}(B)\leq\mathcal{N}(C)$. It follows by Theorem~\ref{theorem 3.6} that $\mathcal{N}(U)\in\{4\}\cup\{2^n\mid n\geq 5\}$ for all $U\in\mathcal{A}(\mathcal{I}^*_2(\mathcal{O}_4))$. Since $\mathcal{N}(A)\mathcal{N}(B)\mathcal{N}(C)=\mathcal{N}(I)\mathcal{N}(J)=2048$, we infer that $\mathcal{N}(A)=\mathcal{N}(B)=4$ and $\mathcal{N}(C)=128$. It follows by Lemma~\ref{lemma 4.2}.2 that $ABC=4D$ for some $D\in\mathcal{A}(\mathcal{I}^*_2(\mathcal{O}_4))$. This implies that $D=2L$, a contradiction. Consequently, $\mathsf{L}(IJ)=\{2,4\}$, and thus $2\in\Delta(\mathcal{I}^*_2(\mathcal{O}_4))$ and $4=\mathsf{c}(IJ)\in {\rm Ca}(\mathcal{I}^*_2(\mathcal{O}_4))$.
\end{proof}

\begin{proposition}\label{proposition 4.5}
Suppose that one of the following conditions hold{\rm \,:}
\begin{enumerate}
\item[\textnormal{(a)}] ${\rm v}_2(f)\geq 5$ or $({\rm v}_2(f)=4$ and $d\not\equiv 1\mod 4)$.
\item[\textnormal{(b)}] ${\rm v}_2(f)=3$ and $d\equiv 2\mod 4$.
\item[\textnormal{(c)}] ${\rm v}_2(f)=2$ and $d\equiv 1\mod 4$.
\end{enumerate}
Then there are $I,J\in\mathcal{A}(\mathcal{I}^*_2(\mathcal{O}_f))$ with $\mathsf{L}(IJ)=\{2,3\}$ whence $1\in\Delta(\mathcal{I}^*_2(\mathcal{O}_f))$ and $3\in {\rm Ca}(\mathcal{I}^*_2(\mathcal{O}_f))$. If $|{\rm Pic}(\mathcal{O}_f)|\leq 2$, then there is a nonzero primary $a\in\mathcal{O}_f$ with $2,3\in\mathsf{L}(a)$ whence $1\in\Delta(\mathcal{O}_f)$.
\end{proposition}

\begin{proof}
CASE 1: ${\rm v}_2(f)\geq 5$ or (${\rm v}_2(f)=4$ and $d\not\equiv 1\mod 4$). We show that there are some $A,B,I,J,L\in\mathcal{A}(\mathcal{I}^*_2(\mathcal{O}_f))$ such that $A^2=32I$, $B^2=16J$ and $IJ=4L$. Set $k={\rm v}_2(\mathcal{N}_{K/\mathbb{Q}}(16+\tau))$ and $A=2^k\mathbb{Z}+(16+\tau)\mathbb{Z}$. Then $k\geq 8$, $A\in\mathcal{A}(\mathcal{I}^*_2(\mathcal{O}_f))$ and $A^2=32(2^{2k-10}\mathbb{Z}+(c+\tau)\mathbb{Z})$ with $(5,2k-10,c)\in\mathcal{M}_{f,2}$ and ${\rm v}_2(c)\geq 3$. Set $I=2^{2k-10}\mathbb{Z}+(c+\tau)\mathbb{Z}$. Then $I\in\mathcal{A}(\mathcal{I}^*_2(\mathcal{O}_f))$. Set $B=64\mathbb{Z}+(8+\tau)\mathbb{Z}$. Then $B\in\mathcal{A}(\mathcal{I}^*_2(\mathcal{O}_f))$ and $B^2=16(16\mathbb{Z}+(4+\tau)\mathbb{Z})$. Set $J=16\mathbb{Z}+(4+\tau)\mathbb{Z}$. Then $B^2=16J$, $J\in\mathcal{A}(\mathcal{I}^*_2(\mathcal{O}_f))$ and $IJ=4L$ with $L\in\mathcal{A}(\mathcal{I}^*_2(\mathcal{O}_f))$.

\smallskip
CASE 2: ${\rm v}_2(f)=3$ and $d\equiv 2\mod 4$. We show that $AB=2I$, $AC=2I^{\prime}$, $BC=8I^{\prime\prime}$, $B^2=16J$, $IJ=4L$, $I^{\prime}J=4L^{\prime}$, $I^{\prime\prime}J=4L^{\prime\prime}$ for some $A,B,C,I,I^{\prime},$ $I^{\prime\prime},J,L,L^{\prime},L^{\prime\prime}\in\mathcal{A}(\mathcal{I}^*_2(\mathcal{O}_f))$. By Proposition~\ref{proposition 3.3}.3, we can assume without restriction that $f=8$. Set $A=4\mathbb{Z}+(2+\tau)\mathbb{Z}$, $B=64\mathbb{Z}+(8+\tau)\mathbb{Z}$ and $C=128\mathbb{Z}+\tau\mathbb{Z}$. Then $A,B,C\in\mathcal{A}(\mathcal{I}^*_2(\mathcal{O}_f))$, $AB=2(64\mathbb{Z}+(40+\tau)\mathbb{Z})$, $AC=2(128\mathbb{Z}+(64+\tau)\mathbb{Z})$, $B^2=16(16\mathbb{Z}+(12+\tau)\mathbb{Z})$ and $BC=8(128\mathbb{Z}+(c+\tau)\mathbb{Z})$ with $(3,7,c)\in\mathcal{M}_{f,2}$ and ${\rm v}_2(c)=4$. Furthermore, $(64\mathbb{Z}+(40+\tau)\mathbb{Z})(16\mathbb{Z}+(12+\tau)\mathbb{Z})=4(64\mathbb{Z}+(56+\tau)\mathbb{Z})$, $(128\mathbb{Z}+(64+\tau)\mathbb{Z})(16\mathbb{Z}+(12+\tau)\mathbb{Z})=4(128\mathbb{Z}+(r+\tau)\mathbb{Z})$ with $(2,7,r)\in\mathcal{M}_{f,2}$ and $(128\mathbb{Z}+(c+\tau)\mathbb{Z})(16\mathbb{Z}+(12+\tau)\mathbb{Z})=4(128\mathbb{Z}+(s+\tau)\mathbb{Z})$ with $(2,7,s)\in\mathcal{M}_{f,2}$. Set $J=16\mathbb{Z}+(12+\tau)\mathbb{Z}$. In particular, if $I\in\{64\mathbb{Z}+(40+\tau)\mathbb{Z},128\mathbb{Z}+(64+\tau)\mathbb{Z},128\mathbb{Z}+(c+\tau)\mathbb{Z}\}$, then $I,J\in\mathcal{A}(\mathcal{I}^*_2(\mathcal{O}_f))$ and $IJ=4L$ for some $L\in\mathcal{A}(\mathcal{I}^*_2(\mathcal{O}_f))$.

\smallskip
CASE 3: ${\rm v}_2(f)=2$ and $d\equiv 1\mod 4$. We show that $A^2=4I$ and $I^2=4L$ for some $A,I,L\in\mathcal{A}(\mathcal{I}^*_2(\mathcal{O}_f))$. By Proposition~\ref{proposition 3.3}.3, we can assume without restriction that $f=4$. First let $d\equiv 1\mod 8$. If $d\equiv 1\mod 16$, then set $A=32\mathbb{Z}+(6+\tau)\mathbb{Z}$ and if $d\equiv 9\mod 16$, then set $A=32\mathbb{Z}+(2+\tau)\mathbb{Z}$. In any case, we have $A\in\mathcal{A}(\mathcal{I}^*_2(\mathcal{O}_f))$ and $A^2=4(64\mathbb{Z}+(c+\tau)\mathbb{Z})$ with $(2,6,c)\in\mathcal{M}_{f,2}$ and ${\rm v}_2(c)=1$. Set $I=64\mathbb{Z}+(c+\tau)\mathbb{Z}$. Then $I\in\mathcal{A}(\mathcal{I}^*_2(\mathcal{O}_f))$, $A^2=4I$ and $I^2=4(256\mathbb{Z}+(r+\tau)\mathbb{Z})$ with $(2,8,r)\in\mathcal{M}_{f,2}$.

Now let $d\equiv 5\mod 8$. Set $A=16\mathbb{Z}+(2+\tau)\mathbb{Z}$. Then $A\in\mathcal{A}(\mathcal{I}^*_2(\mathcal{O}_f))$ and $A^2=4(16\mathbb{Z}+(c+\tau)\mathbb{Z})$ with $(2,4,c)\in\mathcal{M}_{f,2}$ and ${\rm v}_2(c)=1$. Set $I=16\mathbb{Z}+(c+\tau)\mathbb{Z}$. Then $A^2=4I$ and $I^2=4(16\mathbb{Z}+(z+\tau)\mathbb{Z})$ with $(2,4,z)\in\mathcal{M}_{f,2}$.

\smallskip
Using the case analysis above we can find $I,J,L\in\mathcal{A}(\mathcal{I}^*_2(\mathcal{O}_f))$ such that $IJ=4L$. In particular, $\mathsf{L}(IJ)=\{2,3\}$, $1\in\Delta(\mathcal{I}^*_p(\mathcal{O}_f))$ and $3=\mathsf{c}(IJ)\in {\rm Ca}(\mathcal{I}^*_p(\mathcal{O}_f))$. Now let $|{\rm Pic}(\mathcal{O}_f)|\leq 2$. Observe that if $A,B,C\in\mathcal{A}(\mathcal{I}^*_2(\mathcal{O}_f))$, then $A^2$ is principal and $\{AB,AC,BC\}$ contains a principal ideal of $\mathcal{O}_f$. In any case we can choose $I,J,L$ to be principal. There are some $u,v,w\in\mathcal{A}(\mathcal{O}_f)$ such that $I=u\mathcal{O}_f$, $J=v\mathcal{O}_f$, $L=w\mathcal{O}_f$ and $uv=4w$. Note that $uv$ is primary. Since $2\in\mathcal{A}(\mathcal{O}_f)$, we have $2,3\in\mathsf{L}(uv)$, and thus $1\in\Delta(\mathcal{O}_f)$.
\end{proof}

\begin{proposition}\label{proposition 4.6}
Let $p$ be a prime divisor of $f$. Then the following statements are equivalent{\rm \,:}
\begin{enumerate}
\item[\textnormal{(a)}] $\mathcal{I}^*_p(\mathcal{O}_f)$ is half-factorial.
\item[\textnormal{(b)}] $\mathcal{I}_p(\mathcal{O}_f)$ is half-factorial.
\item[\textnormal{(c)}] $\mathsf{c}(\mathcal{I}^*_p(\mathcal{O}_f))=2$.
\item[\textnormal{(d)}] $\mathsf{c}(\mathcal{I}_p(\mathcal{O}_f))=2$.
\item[\textnormal{(e)}] ${\rm v}_p(f)=1$ and $p$ is inert.
\end{enumerate}
\end{proposition}

\begin{proof}
(a) $\Rightarrow$ (e) If ${\rm v}_p(f)>1$ or $p$ is not inert, then there is some $I\in\mathcal{A}(\mathcal{I}^*_p(\mathcal{O}_f))$ such that $\mathcal{N}(I)>p^2$ by Theorem~\ref{theorem 3.6}.4. Set $k={\rm v}_p(\mathcal{N}(I))$. Then $k\geq 3$ and $I\overline{I}=(p\mathcal{O}_f)^k$ by Proposition~\ref{proposition 3.2}.5. Since $\overline{I}\in\mathcal{A}(\mathcal{I}^*_p(\mathcal{O}_f))$, we have $2,k\in\mathsf{L}(I\overline{I})$.

\smallskip
(e) $\Rightarrow$ (b) Observe that $\mathcal{N}(A)\in\{p,p^2\}$ for each $A\in\mathcal{A}(\mathcal{I}_p(\mathcal{O}_f))$, and thus $\mathcal{A}(\mathcal{I}_p(\mathcal{O}_f))=\{P_{f,p}\}\cup\{A\in\mathcal{A}(\mathcal{I}^*_p(\mathcal{O}_f))\mid\mathcal{N}(A)=p^2\}$. Let $I\in\mathcal{I}_p(\mathcal{O}_f)\setminus\{\mathcal{O}_f\}$. There are some $k\in\mathbb{N}_0$ and $J\in\mathcal{A}(\mathcal{I}_p(\mathcal{O}_f))$ such that $I=p^kJ$. Let $z\in\mathsf{Z}(I)$. Then $z=(\prod_{i=1}^n I_i)\cdot P_{f,p}^{\ell}$ with $\ell,n\in\mathbb{N}_0$ and $I_i\in\mathcal{A}(\mathcal{I}^*_p(\mathcal{O}_f))$ for each $i\in [1,n]$. Note that $|z|=n+\ell$. It is sufficient to show that $n+\ell=k+1$.

\smallskip
CASE 1: $I$ is invertible. Then $J$ is invertible and $\ell=0$. It follows that $p^{2n}=\mathcal{N}(\prod_{i=1}^n I_i)=\mathcal{N}(I)=\mathcal{N}(p^kJ)=p^{2k+2}$ by Proposition~\ref{proposition 3.2}.3, and thus $n+\ell=n=k+1$.

\smallskip
CASE 2: $I$ is not invertible. Then $J=P_{f,p}$ and $\ell>0$. It follows from Lemma~\ref{lemma 4.2} that $P_{f,p}^{\ell}=p^{\ell-1}P_{f,p}$. Consequently,
\[
p^{2(n+\ell)-1}=\mathcal{N}(\prod_{i=1}^n I_i)\mathcal{N}(p^{\ell-1}P_{f,p})=\mathcal{N}(I)=\mathcal{N}(p^kP_{f,p})=p^{2k+1}
\]
by Proposition~\ref{proposition 3.2}.3, and hence $n+\ell=k+1$.

\smallskip
(b) $\Rightarrow$ (d) Since $\mathcal{I}^*_p(\mathcal{O}_f)$ is a cancellative divisor-closed submonoid of $\mathcal{I}_p(\mathcal{O}_f)$ and not factorial, we infer by Proposition~\ref{proposition 4.1}.4 that
\[
2\leq\mathsf{c}(\mathcal{I}^*_p(\mathcal{O}_f))\leq\mathsf{c}(\mathcal{I}_p(\mathcal{O}_f))\leq 2.
\]

\smallskip
(d) $\Rightarrow$ (c) Note that $\mathcal{I}^*_p(\mathcal{O}_f)$ is a divisor-closed submonoid of $\mathcal{I}_p(\mathcal{O}_f)$, and thus $\mathsf{c}(\mathcal{I}^*_p(\mathcal{O}_f))\leq\mathsf{c}(\mathcal{I}_p(\mathcal{O}_f))=2$. Since $\mathcal{I}^*_p(\mathcal{O}_f)$ is not factorial, we infer that $\mathsf{c}(\mathcal{I}^*_p(\mathcal{O}_f))=2$.

\smallskip
(c) $\Rightarrow$ (a) Since $\mathcal{I}^*_p(\mathcal{O}_f)$ is cancellative and not factorial, it follows that $2+\sup\Delta(\mathcal{I}^*_p(\mathcal{O}_f))\leq\mathsf{c}(\mathcal{I}^*_p(\mathcal{O}_f))=2$, and thus $\sup\Delta(\mathcal{I}^*_p(\mathcal{O}_f))=0$. Consequently, $\Delta(\mathcal{I}^*_p(\mathcal{O}_f))=\emptyset$, and hence $\mathcal{I}^*_p(\mathcal{O}_f)$ is half-factorial.
\end{proof}

\begin{lemma}\label{lemma 4.7}
Let $p$ be a prime divisor of $f$, $|{\rm Pic}(\mathcal{O}_f)|\leq 2$, $I,J,L\in\mathcal{A}(\mathcal{I}^*_p(\mathcal{O}_f))$.
\begin{enumerate}
\item[\textnormal{1.}] If $J$ is principal and $IJ=p^2L$, then $1\in\Delta(\mathcal{O}_f)$.
\item[\textnormal{2.}] If $I$ and $J$ are not principal and $IJ=pL$, then $1\in\Delta(\mathcal{O}_f)$.
\end{enumerate}
\end{lemma}

\begin{proof}
Note that if $|{\rm Pic}(\mathcal{O}_f)|>1$, then it follows from \cite[Corollary 2.11.16]{Ge-HK06a} that there is some invertible prime ideal $P$ of $\mathcal{O}_f$ that is not principal. Observe that $p\in\mathcal{A}(\mathcal{O}_f)$. Also note that if $I$ is not principal, then $PI$ is principal, and hence $PI$ is generated by an atom of $\mathcal{O}_f$, since $PI$ has no nontrivial factorizations in $\mathcal{I}^*(\mathcal{O}_f)$.

\smallskip
1. Let $J$ be principal and $IJ=p^2L$. There is some $v\in\mathcal{A}(\mathcal{O}_f)$ such that $J=v\mathcal{O}_f$.

\smallskip
CASE 1: $I$ is principal. Then $L$ is principal, and hence there are some $u,w\in\mathcal{A}(\mathcal{O}_f)$ such that $I=u\mathcal{O}_f$, $L=w\mathcal{O}_f$ and $uv=p^2w$. We infer that $2,3\in\mathsf{L}(uv)$, and thus $1\in\Delta(\mathcal{O}_f)$.

\smallskip
CASE 2: $I$ is not principal. Then $L$ is not principal and $|{\rm Pic}(\mathcal{O}_f)|>1$, and thus there are some $u,w\in\mathcal{A}(\mathcal{O}_f)$ such that $PI=u\mathcal{O}_f$, $PL=w\mathcal{O}_f$ and $uv=p^2w$. It follows that $2,3\in\mathsf{L}(uv)$, and thus $1\in\Delta(\mathcal{O}_f)$.

\smallskip
2. Let $I$ and $J$ not be principal and $IJ=pL$. Then $L$ is principal and $|{\rm Pic}(\mathcal{O}_f)|>1$, and hence there are some $u,v,w,y\in\mathcal{A}(\mathcal{O}_f)$ such that $PI=u\mathcal{O}_f$, $PJ=v\mathcal{O}_f$, $P^2=w\mathcal{O}_f$, $L=y\mathcal{O}_f$ and $uv=pwy$. Therefore, $2,3\in\mathsf{L}(uv)$, and hence $1\in\Delta(\mathcal{O}_f)$.
\end{proof}

\begin{proposition}\label{proposition 4.8}
Let $p$ be a prime divisor of $f$.
\begin{enumerate}
\item[\textnormal{1.}] If ${\rm v}_p(f)\geq 2$ or $p$ is not inert, then there are $I,J\in\mathcal{A}(\mathcal{I}^*_p(\mathcal{O}_f))$ such that $\mathsf{L}(IJ)=\{2,3\}$ whence $1\in\Delta(\mathcal{I}^*_p(\mathcal{O}_f))$ and $3\in {\rm Ca}(\mathcal{I}^*_p(\mathcal{O}_f))$.

\item[\textnormal{2.}] Suppose that $\mathcal{O}_f$ is not half-factorial and that one of the following conditions holds{\rm\,:}
\begin{enumerate}
\item[\textnormal{(i)}] $|{\rm Pic}(\mathcal{O}_f)|\geq 3$ or ${\rm v}_p(f)\geq 2$ or $p$ does split.
\item[\textnormal{(ii)}] $p$ is inert and there is some $C\in\mathcal{A}(\mathcal{I}^*_p(\mathcal{O}_f))$ that is not principal.
\item[\textnormal{(iii)}] $p$ is ramified and there is some principal $C\in\mathcal{A}(\mathcal{I}^*_p(\mathcal{O}_f))$ such that $\mathcal{N}(C)=p^3$.
\item[\textnormal{(iv)}] $f$ is a squarefree product of inert primes.
\end{enumerate}
Then $1\in\Delta(\mathcal{O}_f)$.
\end{enumerate}
\end{proposition}

\begin{proof}
We prove 1. and 2. simultaneously. Set $G={\rm Pic}(\mathcal{O}_f)$. Let $\mathcal{B}(G)$ be the monoid of zero-sum sequences of $G$. It follows by \cite[Theorem 6.7.1.2]{Ge-HK06a} that if $|G|\geq 3$, then $1\in\Delta(\mathcal{B}(G))$. We infer by \cite[Proposition 3.4.7 and Theorems 3.4.10.3 and 3.7.1.1]{Ge-HK06a} that there exists an atomic monoid $\mathcal{B}(\mathcal{O}_f)$ such that $\Delta(\mathcal{B}(\mathcal{O}_f))=\Delta(\mathcal{O}_f)$ and $\mathcal{B}(G)$ is a divisor-closed submonoid of $\mathcal{B}(\mathcal{O}_f)$. In particular, if $|G|\geq 3$, then $1\in\Delta(\mathcal{O}_f)$. Thus, for the second assertion we only need to consider the case $|G|\leq 2$. By Propositions~\ref{proposition 4.3} and~\ref{proposition 4.5} we can restrict to the following cases.

\smallskip
CASE 1: $p=2$ and $(({\rm v}_2(f)\in\{3,4\}$ and $d\equiv 1\mod 4)$ or $({\rm v}_2(f)\in\{2,3\}$ and $d\equiv 3\mod 4))$. If $({\rm v}_2(f)=4$ and $d\equiv 1\mod 4)$ or $({\rm v}_2(f)=3$ and $d\equiv 3\mod 4)$, then set $I=16\mathbb{Z}+(4+\tau)\mathbb{Z}$. If ${\rm v}_2(f)=3$ and $d\equiv 1\mod 4$, then set $I=16\mathbb{Z}+\tau\mathbb{Z}$. Finally, if ${\rm v}_2(f)=2$ and $d\equiv 3\mod 4$, then there is some $I\in\mathcal{A}(\mathcal{I}^*_2(\mathcal{O}_f))$ such that $\mathcal{N}(I)=32$ by Theorem~\ref{theorem 3.6}. In any case, it follows that $I\in\mathcal{A}(\mathcal{I}^*_2(\mathcal{O}_f))$.

It is a consequence of Proposition~\ref{proposition 3.2}.1 and Theorem~\ref{theorem 3.6} that there are some $A,J\in\mathcal{A}(\mathcal{I}^*_2(\mathcal{O}_f))$ and $\ell\in\mathbb{N}$ such that $A^2=\ell J$ with values according to the following table. Let $k\in\{1,3,5,7\}$ be such that $d\equiv k\mod 8$. Note that $I=2^a\mathbb{Z}+(r+\tau)\mathbb{Z}$ and $J=2^b\mathbb{Z}+(s+\tau)\mathbb{Z}$ with $(0,a,r),(0,b,s)\in\mathcal{M}_{f,2}$.

\begin{table}[htbp]
\centering
\begin{tabular}{|c|c|c|c|c|c|c|}
\hline
${\rm v}_2(f)$ & $k$ & $\mathcal{N}(A)$ & $\ell$ & $\mathcal{N}(J)$ & ${\rm v}_2(r)$ & ${\rm v}_2(s)$\\
\hline
$4$ & $1$ & $512$ & $16$ & $1024$ & $2$ & $3$\\
\hline
$4$ & $5$ & $256$ & $16$ & $256$ & $2$ & $3$\\
\hline
$3$ & $1$ & $128$ & $8$ & $256$ & $\infty$ & $2$\\
\hline
$3$ & $5$ & $64$ & $8$ & $64$ & $\infty$ & $2$\\
\hline
$3$ & $3$ or $7$ & $128$ & $16$ & $64$ & $2$ & $\geq 4$\\
\hline
$2$ & $3$ or $7$ & $32$ & $8$ & $16$ & $2$ & $\geq 3$\\
\hline
\end{tabular}
\end{table}

Since ${\rm v}_2(r+s)=2$ in any case, we infer that $IJ=4L$ for some $L\in\mathcal{A}(\mathcal{I}^*_2(\mathcal{O}_f))$. Now let $|G|\leq 2$. We have $J$ is principal, and hence $1\in\Delta(\mathcal{O}_f)$ by Lemma~\ref{lemma 4.7}.1.

\smallskip
CASE 2: $p=2$, ${\rm v}_2(f)=2$ and $d\equiv 2\mod 4$. Set $A=32\mathbb{Z}+\tau\mathbb{Z}$ and $B=32\mathbb{Z}+(8+\tau)\mathbb{Z}$. Then $A,B\in\mathcal{A}(\mathcal{I}^*_2(\mathcal{O}_f))$ and $AB=8I$ for some $I\in\mathcal{A}(\mathcal{I}^*_2(\mathcal{O}_f))$ with $I=16\mathbb{Z}+(r+\tau)\mathbb{Z}$, $(0,4,r)\in\mathcal{M}_{f,2}$, and ${\rm v}_2(r)=2$. Therefore, we have $AI=4J$ and $BI=4L$ for some $J,L\in\mathcal{A}(\mathcal{I}^*_2(\mathcal{O}_f))$. Now let $|G|\leq 2$. Since $\{A,B,I\}$ contains a principal ideal of $\mathcal{O}_f$, we infer by Lemma~\ref{lemma 4.7}.1 that $1\in\Delta(\mathcal{O}_f)$.

\smallskip
CASE 3: $p=3$, ${\rm v}_3(f)=2$ and $d\equiv 2\mod 3$. First let $d\not\equiv 1\mod 4$. Set $I=81\mathbb{Z}+\tau\mathbb{Z}$ and $J=81\mathbb{Z}+(9+\tau)\mathbb{Z}$. Then $I,J\in\mathcal{A}(\mathcal{I}^*_3(\mathcal{O}_f))$ and $IJ=9L$ for some $L\in\mathcal{A}(\mathcal{I}^*_3(\mathcal{O}_f))$ with $L=81\mathbb{Z}+(r+\tau)\mathbb{Z}$, $(0,4,r)\in\mathcal{M}_{f,3}$, and ${\rm v}_3(r)=2$. It follows that $IL=9A$ for some $A\in\mathcal{A}(\mathcal{I}^*_3(\mathcal{O}_f))$.

Now let $d\equiv 1\mod 4$. By Proposition~\ref{proposition 3.3}.3 we can assume without restriction that $f$ is odd. Set $I=81\mathbb{Z}+(4+\tau)\mathbb{Z}$ and $J=81\mathbb{Z}+(13+\tau)\mathbb{Z}$. Then $I,J\in\mathcal{A}(\mathcal{I}^*_3(\mathcal{O}_f))$ and $IJ=9L$ for some $L\in\mathcal{A}(\mathcal{I}^*_3(\mathcal{O}_f))$. There is some $(0,4,r)\in\mathcal{M}_{f,3}$ such that $L=81\mathbb{Z}+(r+\tau)\mathbb{Z}$. Since ${\rm v}_3(2r+1)\geq 2$, we have $IL=9A$ for some $A\in\mathcal{A}(\mathcal{I}^*_3(\mathcal{O}_f))$ or $JL=9A$ for some $A\in\mathcal{A}(\mathcal{I}^*_3(\mathcal{O}_f))$.

In any case if $|G|\leq 2$, then $\{I,J,L\}$ contains a principal ideal of $\mathcal{O}_f$, and hence $1\in\Delta(\mathcal{O}_f)$ by Lemma~\ref{lemma 4.7}.1.

\smallskip
CASE 4: ${\rm v}_p(f)=1$ and $p$ splits. By Theorem~\ref{theorem 3.6} there is some $I\in\mathcal{A}(\mathcal{I}^*_p(\mathcal{O}_f))$ such that $\mathcal{N}(I)=p^3$. There is some $(0,3,r)\in\mathcal{M}_{f,p}$ such that $I=p^3\mathbb{Z}+(r+\tau)\mathbb{Z}$. Observe that ${\rm v}_p(2r+\varepsilon)=1$. We infer that $I^2=pJ$ for some $J\in\mathcal{A}(\mathcal{I}^*_p(\mathcal{O}_f))$ and $I\overline{I}=p^2L$ with $\overline{I}\in\mathcal{A}(\mathcal{I}^*_p(\mathcal{O}_f))$ and $L=p\mathcal{O}_f\in\mathcal{A}(\mathcal{I}^*_p(\mathcal{O}_f))$. Now let $|G|\leq 2$. We infer by Lemma~\ref{lemma 4.7} that $1\in\Delta(\mathcal{O}_f)$.

\smallskip
CASE 5: ${\rm v}_p(f)=1$ and $p$ is ramified. By Theorem~\ref{theorem 3.6} there is some $C\in\mathcal{A}(\mathcal{I}^*_p(\mathcal{O}_f))$ such that $\mathcal{N}(C)=p^3$. Note that $C\overline{C}=p^3\mathcal{O}_f$ and $\overline{C}\in\mathcal{A}(\mathcal{I}^*_p(\mathcal{O}_f))$. Now let $C$ be principal. It follows by Lemma~\ref{lemma 4.7}.1 that $1\in\Delta(\mathcal{O}_f)$.

\smallskip
Cases 1-5 show that there are some $I,J,L\in\mathcal{A}(\mathcal{I}^*_p(\mathcal{O}_f))$ such that $IJ=p^2L$. In particular, $\mathsf{L}(IJ)=\{2,3\}$, $1\in\Delta(\mathcal{I}^*_p(\mathcal{O}_f))$ and $3=\mathsf{c}(IJ)\in {\rm Ca}(\mathcal{I}^*_p(\mathcal{O}_f))$. This proves 1. For the rest of this proof let $\mathcal{O}_f$ be not half-factorial and $|G|\leq 2$.

\smallskip
CASE 6: ${\rm v}_p(f)=1$, $p$ is inert and there is some $C\in\mathcal{A}(\mathcal{I}^*_p(\mathcal{O}_f))$ that is not principal. We have $C^2=pL$ for some $L\in\mathcal{A}(\mathcal{I}^*_p(\mathcal{O}_f))$, and thus $1\in\Delta(\mathcal{O}_f)$ by Lemma~\ref{lemma 4.7}.2.

\smallskip
CASE 7: $f$ is a squarefree product of inert primes. Then $\mathcal{I}^*_p(\mathcal{O}_f)$ is half-factorial by Proposition~\ref{proposition 4.6}. If $G$ is trivial, then $\mathcal{O}_f$ is half-factorial, a contradiction. Note that $\mathcal{O}_f$ is seminormal by \cite[Corollary 4.5]{Do-Fo87}. It follows from \cite[Theorem 6.2.2.(a)]{Ge-Ka-Re15a} that $1\in\Delta(\mathcal{O}_f)$.
\end{proof}

\begin{lemma}\label{lemma 4.9}
Let $p$ be a prime divisor of $f$, $k\in\mathbb{N}_{\geq 2}$, and $N=\sup\{{\rm v}_p(\mathcal{N}(A))\mid A\in\mathcal{A}(\mathcal{I}^*_p(\mathcal{O}_f))\}$. If $\ell\in\mathbb{N}$ and $A\in\mathcal{I}_p(\mathcal{O}_f))$ is both a product of $k$ atoms and a product of $\ell$ atoms, then $\ell\leq\frac{kN}{2}$.
\end{lemma}

\begin{proof}
Let $\ell\in\mathbb{N}$ and suppose that a product of $k$ atoms can be written as a product of $\ell$ atoms and set $P=P_{f,p}$. There are some $a,b\in\mathbb{N}_0$, $I_i\in\mathcal{A}(\mathcal{I}_p(\mathcal{O}_f))\setminus\{P\}$ for each $[1,b]$ and $J_j\in\mathcal{A}(\mathcal{I}_p(\mathcal{O}_f))$ for each $j\in [1,k]$ such that $\ell=a+b$ and $\prod_{j=1}^k J_j=P^a\prod_{i=1}^b I_i$. Note that $p^2\mid\mathcal{N}(I_i)$ for each $i\in [1,b]$.

\smallskip
CASE 1: $a=0$. Then $b=\ell$. It follows by induction from Proposition~\ref{proposition 3.2}.4 that there are $J^{\prime}_j\in\mathcal{A}(\mathcal{I}^*_p(\mathcal{O}_f))$ for each $j\in [1,k]$ such that $\mathcal{N}(\prod_{j=1}^k J_j)\mid\mathcal{N}(\prod_{j=1}^k J^{\prime}_j)$. Set $M={\rm lcm}\{\mathcal{N}(J^{\prime}_j)\mid j\in [1,k]\}$. Then $p^{2\ell}\mid\prod_{i=1}^{\ell}\mathcal{N}(I_i)\mid\mathcal{N}(\prod_{i=1}^{\ell} I_i)=\mathcal{N}(\prod_{j=1}^k J_j)\mid\mathcal{N}(\prod_{j=1}^k J^{\prime}_j)=\prod_{j=1}^k\mathcal{N}(J^{\prime}_j)\mid M^k$. This implies that $2\ell\leq k{\rm v}_p(M)\leq kN$, and thus $\ell\leq\frac{kN}{2}$.

\smallskip
CASE 2: $a>0$. By Lemma~\ref{lemma 4.2} we have $P^a=p^{a-1}P$, and thus $\mathcal{N}(P^a)=p^{2a-1}$. Note that $\prod_{j=1}^k J_j$ is not invertible, and hence one member of the product, say $J_1$, is not invertible. Observe that ${\rm v}_p(\mathcal{N}(J_1))\leq N-1$ by Proposition~\ref{proposition 3.2}.4. We infer by induction from Proposition~\ref{proposition 3.2}.4 that there are $J^{\prime}_j\in\mathcal{A}(\mathcal{I}^*_p(\mathcal{O}_f))$ for each $j\in [2,k]$ such that $\mathcal{N}(\prod_{j=1}^k J_j)\mid\mathcal{N}(J_1\prod_{j=2}^k J^{\prime}_j)$. Set $M={\rm lcm}\{\mathcal{N}(J^{\prime}_j)\mid j\in [2,k]\}$. Then $p^{2\ell-1}\mid\mathcal{N}(P^a)\prod_{i=1}^b\mathcal{N}(I_i)\mid\mathcal{N}(P^a\prod_{i=1}^b I_i)=\mathcal{N}(\prod_{j=1}^k J_j)\mid\mathcal{N}(J_1\prod_{j=2}^k J^{\prime}_j)=\mathcal{N}(J_1)\prod_{j=2}^k\mathcal{N}(J^{\prime}_j)\mid\mathcal{N}(J_1)M^{k-1}$. This implies that $2\ell-1\leq {\rm v}_p(\mathcal{N}(J_1))+(k-1){\rm v}_p(M)\leq kN-1$, and hence $\ell\leq\frac{kN}{2}$.
\end{proof}

\begin{lemma}\label{lemma 4.10}
Let $p$ be a prime divisor of $f$. For every $I\in\mathcal{A}(\mathcal{I}^*_p(\mathcal{O}_f))$, we set ${\rm v}_I={\rm v}_p(\mathcal{N}(I))$, and let $\mathcal{B}=\{{\rm v}_A\mid A\in\mathcal{A}(\mathcal{I}^*_p(\mathcal{O}_f))\}$.
\begin{enumerate}
\item[\textnormal{1.}] For all $I\in\mathcal{A}(\mathcal{I}^*_p(\mathcal{O}_f))$, we have $\mathsf{c}(I\cdot\overline{I},(p\mathcal{O}_f)^{{\rm v}_I})\leq 2+\sup\Delta(\mathcal{B})$.

\item[\textnormal{2.}] Let $p=2$, $d\equiv 1\mod 8$, and ${\rm v}_p(f)\geq 4$. Then $\mathsf{c}(I\cdot\overline{I},(p\mathcal{O}_f)^{{\rm v}_I})\leq 4$ for all $I\in\mathcal{A}(\mathcal{I}^*_p(\mathcal{O}_f))$.
\end{enumerate}
\end{lemma}

\begin{proof}
1. It is sufficient to show by induction that for all $n\in\mathbb{N}_{\geq 2}$ and $I\in\mathcal{A}(\mathcal{I}^*_p(\mathcal{O}_f))$ with ${\rm v}_I=n$, it follows that $\mathsf{c}(I\cdot\overline{I},(p\mathcal{O}_f)^n)\leq 2+\sup\Delta(\mathcal{B})$. Let $n\in\mathbb{N}_{\geq 2}$ and $I\in\mathcal{A}(\mathcal{I}^*_p(\mathcal{O}_f))$ be such that ${\rm v}_I=n$. If $n=2$, then $\mathsf{c}(I\cdot\overline{I},(p\mathcal{O}_f)^2)\leq\mathsf{d}(I\cdot\overline{I},(p\mathcal{O}_f)^2)\leq 2\leq 2+\sup\Delta(\mathcal{B})$. Now let $n>2$. Note that $2={\rm v}_{p\mathcal{O}_f}\in\mathcal{B}$, and hence there is some $k\in\mathcal{B}$ such that $2\leq k<n$ and $\mathcal{B}\cap [k,n]=\{k,n\}$. Observe that $n-k\in\Delta(\mathcal{B})$. Furthermore, there is some $J\in\mathcal{A}(\mathcal{I}^*_p(\mathcal{O}_f))$ such that $k={\rm v}_J$. Note that $J\overline{J}=(p\mathcal{O}_f)^k$, and thus $I\overline{I}=(p\mathcal{O}_f)^{n-k}J\overline{J}$. By the induction hypothesis, we infer that $c((p\mathcal{O}_f)^{n-k}\cdot J\cdot\overline{J},(p\mathcal{O}_f)^n)\leq c(J\cdot\overline{J},(p\mathcal{O}_f)^k)\leq 2+\sup\Delta(\mathcal{B})$. Since $\mathsf{d}(I\cdot\overline{I},(p\mathcal{O}_f)^{n-k}\cdot J\cdot\overline{J})\leq 2+(n-k)\leq 2+\sup\Delta(\mathcal{B})$, it follows that $\mathsf{c}(I\cdot\overline{I},(p\mathcal{O}_f)^n)\leq 2+\sup\Delta(\mathcal{B})$.

\smallskip
2. By Proposition~\ref{proposition 3.3}.3 we can assume without restriction that $f=2^{{\rm v}_2(f)}$. We show by induction that for all $n\in\mathbb{N}_{\geq 2}$ and $I\in\mathcal{A}(\mathcal{I}^*_2(\mathcal{O}_f))$ with ${\rm v}_I=n$, we have $\mathsf{c}(I\cdot\overline{I},(2\mathcal{O}_f)^n)\leq 4$. Let $n\in\mathbb{N}_{\geq 2}$ and $I\in\mathcal{A}(\mathcal{I}^*_2(\mathcal{O}_f))$ be such that ${\rm v}_I=n$. If $n=2$, then $\mathsf{c}(I\cdot\overline{I},(2\mathcal{O}_f)^2)\leq\mathsf{d}(I\cdot\overline{I},(2\mathcal{O}_f)^2)\leq 2\leq 2+\sup\Delta(\mathcal{B})$. Next let $n>2$. Observe that $2={\rm v}_{2\mathcal{O}_f}\in\mathcal{B}$, and hence there is some $k\in\mathcal{B}$ such that $2\leq k<n$ and $\mathcal{B}\cap [k,n]=\{k,n\}$. There is some $J\in\mathcal{A}(\mathcal{I}^*_2(\mathcal{O}_f))$ such that $k={\rm v}_J$. Note that $J\overline{J}=(2\mathcal{O}_f)^k$, and hence $I\overline{I}=(2\mathcal{O}_f)^{n-k}J\overline{J}$. By the induction hypothesis, we have $c((2\mathcal{O}_f)^{n-k}\cdot J\cdot\overline{J},(2\mathcal{O}_f)^n)\leq c(J\cdot\overline{J},(2\mathcal{O}_f)^k)\leq 4$.

\smallskip
CASE 1: $n\not=2{\rm v}_2(f)+1$. It follows from Theorem~\ref{theorem 3.6} that $n-k\leq 2$. Since $\mathsf{d}(I\cdot\overline{I},(2\mathcal{O}_f)^{n-k}\cdot J\cdot\overline{J})\leq 4$, we infer that $\mathsf{c}(I\cdot\overline{I},(2\mathcal{O}_f)^n)\leq 4$.

\smallskip
CASE 2: $n=2{\rm v}_2(f)+1$. By Theorem~\ref{theorem 3.6} we have $n-k=3$. Set $A=16\mathbb{Z}+(4+\tau)\mathbb{Z}$, $B=2^{n-3}\mathbb{Z}+(2^{n-5}+\tau)\mathbb{Z}$, and $C=2^{n-3}\mathbb{Z}+(2^{n-4}+\tau)\mathbb{Z}$. Then $A,B,C\in\mathcal{A}(\mathcal{I}^*_2(\mathcal{O}_f))$ and $ABC=2^{n-5}A(16\mathbb{Z}+(12+\tau)\mathbb{Z})=(2\mathcal{O}_f)^{n-1}$. Observe that $\mathsf{d}(I\cdot\overline{I},(2\mathcal{O}_f)\cdot A\cdot B\cdot C)\leq 4$ and $\mathsf{d}((2\mathcal{O}_f)\cdot A\cdot B\cdot C,(2\mathcal{O}_f)^{n-k}\cdot J\cdot\overline{J}))\leq 4$. Therefore, $\mathsf{c}(I\cdot\overline{I},(2\mathcal{O}_f)^n)\leq 4$.
\end{proof}

\begin{proposition}\label{proposition 4.11}
Let $p$ be a prime divisor of $f$ and set $\mathcal{B}=\{{\rm v}_p(\mathcal{N}(\mathcal{A}))\mid A\in\mathcal{A}(\mathcal{I}^*_p(\mathcal{O}_f))\}$.
\begin{enumerate}
\item[\textnormal{1.}] $\sup\Delta(\mathcal{I}_p(\mathcal{O}_f))\leq\sup\Delta(\mathcal{B})$ and $\mathsf{c}(\mathcal{I}_p(\mathcal{O}_f))\leq 2+\sup\Delta(\mathcal{B})$.

\item[\textnormal{2.}] Let $p=2$, $d\equiv 1\mod 8$, and ${\rm v}_p(f)\geq 4$. Then $\sup\Delta(\mathcal{I}_2(\mathcal{O}_f))\leq 2$ and $\mathsf{c}(\mathcal{I}_2(\mathcal{O}_f))\leq 4$.
\end{enumerate}
\end{proposition}

\begin{proof}
1. First we consider the case that ${\rm v}_p(f)=1$ and $p$ is inert. It follows from Theorem~\ref{theorem 3.6} that $\sup\Delta(\mathcal{B})=0$. Proposition~\ref{proposition 4.6} implies that $\sup\Delta(\mathcal{I}_p(\mathcal{O}_f))=0$ and $\mathsf{c}(\mathcal{I}_p(\mathcal{O}_f))=2$. Now let ${\rm v}_p(f)\geq 2$ or $p$ not inert. Observe that $\sup\Delta(\mathcal{B})\geq 1$ by Theorem~\ref{theorem 3.6}. Let $I,J\in\mathcal{A}(\mathcal{I}_p(\mathcal{O}_f))$. There are some $n\in\mathbb{N}$ and $L\in\mathcal{A}(\mathcal{I}_p(\mathcal{O}_f))$ such that $IJ=p^nL$.

By Proposition~\ref{proposition 4.1}, it remains to show that $\mathsf{c}(I\cdot J,(p\mathcal{O}_f)^n\cdot L)\leq 2+\sup\Delta(\mathcal{B})$ and if $\ell\in\mathbb{N}_{\geq 3}$ is such that $\mathsf{L}(IJ)\cap [2,\ell]=\{2,\ell\}$, then $\ell-2\leq\sup\Delta(\mathcal{B})$. Set $N=\sup\mathcal{B}$. Since a product of two atoms of $\mathcal{I}_p(\mathcal{O}_f)$ can be written as a product of $n+1$ atoms, Lemma~\ref{lemma 4.9} implies that $n+1 \le N$. If $n=1$, then $\mathsf{d}(I\cdot J,(p\mathcal{O}_f)\cdot L)\leq 2\leq 2+\sup\Delta(\mathcal{B})$ and there is no $\ell\in\mathbb{N}_{\geq 3}$ with $\mathsf{L}(IJ)\cap [2,\ell]=\{2,\ell\}$. Now let $n\geq 2$ and $\ell\in\mathbb{N}_{\geq 3}$ be such that $\mathsf{L}(IJ)\cap [2,\ell]=\{2,\ell\}$.

\smallskip
CASE 1: $n\in\mathcal{B}$. Then $A\overline{A}=(p\mathcal{O}_f)^n$ for some $A\in\mathcal{A}(\mathcal{I}^*_p(\mathcal{O}_f))$. Therefore, $\mathsf{c}(A\cdot\overline{A}\cdot L,(p\mathcal{O}_f)^n\cdot L)\leq\mathsf{c}(A\cdot\overline{A},(p\mathcal{O}_f)^n)\leq 2+\sup\Delta(\mathcal{B})$ by Lemma~\ref{lemma 4.10}.1. Moreover, $\mathsf{d}(I\cdot J,A\cdot\overline{A}\cdot L)\leq 3\leq 2+\sup\Delta(\mathcal{B})$, and thus $\mathsf{c}(I\cdot J,(p\mathcal{O}_f)^n\cdot L)\leq 2+\sup\Delta(\mathcal{B})$ and $\ell-2=1\leq\sup\Delta(\mathcal{B})$.

\smallskip
CASE 2: $n\not\in\mathcal{B}$. Note that $n\geq 3$. It follows by Theorem~\ref{theorem 3.6} that ${\rm v}_p(f)\geq 2$ and $\sup\Delta(\mathcal{B})\geq 2$.

\smallskip
CASE 2.1: $p\not=2$ or $d\not\equiv 1\mod 8$ or $n\not=2{\rm v}_p(f)$. Since $n\leq N$, it follows from Theorem~\ref{theorem 3.6} that $n-1=\mathcal{N}(A)$ for some $A\in\mathcal{A}(\mathcal{I}^*_p(\mathcal{O}_f))$, and hence $A\overline{A}=(p\mathcal{O}_f)^{n-1}$. We infer that $\mathsf{c}((p\mathcal{O}_f)\cdot A\cdot\overline{A}\cdot L,(p\mathcal{O}_f)^n\cdot L)\leq\mathsf{c}(A\cdot\overline{A},(p\mathcal{O}_f)^{n-1})\leq 2+\sup\Delta(\mathcal{B})$ by Lemma~\ref{lemma 4.10}.1. Moreover, we have $\mathsf{d}(I\cdot J,A\cdot\overline{A}\cdot (p\mathcal{O}_f)\cdot L)\leq 4\leq 2+\sup\Delta(\mathcal{B})$, and thus $\mathsf{c}(I\cdot J,(p\mathcal{O}_f)^n\cdot L)\leq 2+\sup\Delta(\mathcal{B})$ and $\ell-2\leq 2\leq\sup\Delta(\mathcal{B})$.

\smallskip
CASE 2.2: $p=2$, $d\equiv 1\mod 8$ and $n=2{\rm v}_p(f)$. We infer by Theorem~\ref{theorem 3.6} that $\sup\Delta(\mathcal{B})=3$. By Theorem~\ref{theorem 3.6} there is some $A\in\mathcal{A}(\mathcal{I}^*_2(\mathcal{O}_f))$ such that $n-2=\mathcal{N}(A)$, and thus $A\overline{A}=(2\mathcal{O}_f)^{n-2}$. This implies that $\mathsf{c}((2\mathcal{O}_f)^2\cdot A\cdot\overline{A}\cdot L,(2\mathcal{O}_f)^n\cdot L)\leq\mathsf{c}(A\cdot\overline{A},(2\mathcal{O}_f)^{n-2})\leq 2+\sup\Delta(\mathcal{B})$ by Lemma~\ref{lemma 4.10}.1. Observe that $\mathsf{d}(I\cdot J,A\cdot\overline{A}\cdot (2\mathcal{O}_f)^2\cdot L)\leq 5=2+\sup\Delta(\mathcal{B})$, and hence $\mathsf{c}(I\cdot J,(2\mathcal{O}_f)^n\cdot L)\leq 2+\sup\Delta(\mathcal{B})$ and $\ell-2\leq 3=\sup\Delta(\mathcal{B})$.

\smallskip
2. By Proposition~\ref{proposition 3.3}.3 we can assume without restriction that $f=2^{{\rm v}_2(f)}$. Let $I,J\in\mathcal{A}(\mathcal{I}_2(\mathcal{O}_f))$. There are some $n\in\mathbb{N}$ and $L\in\mathcal{A}(\mathcal{I}_2(\mathcal{O}_f))$ such that $IJ=2^nL$. It follows from Lemma~\ref{lemma 4.9} that $n+1\leq\sup\mathcal{B}$. By Proposition~\ref{proposition 4.1}, it is sufficient to show that $\mathsf{c}(I\cdot J,(2\mathcal{O}_f)^n\cdot L)\leq 4$ and if $\ell\in\mathbb{N}_{\geq 3}$ is such that $\mathsf{L}(IJ)\cap [2,\ell]=\{2,\ell\}$, then $\ell-2\leq 2$. The assertion is trivially true for $n=1$. Let $n\geq 2$ and let $\ell\in\mathbb{N}_{\geq 3}$ be such that $\mathsf{L}(IJ)\cap [2,\ell]=\{2,\ell\}$.

\smallskip
CASE 1: $n\in\mathcal{B}$. There is some $A\in\mathcal{A}(\mathcal{I}^*_2(\mathcal{O}_f))$ such that $A\overline{A}=(2\mathcal{O}_f)^n$. It follows by Lemma~\ref{lemma 4.10}.2 that $\mathsf{c}(A\cdot\overline{A}\cdot L,(2\mathcal{O}_f)^n\cdot L)\leq\mathsf{c}(A\cdot\overline{A},(2\mathcal{O}_f)^n)\leq 4$. Furthermore, $\mathsf{d}(I\cdot J,A\cdot\overline{A}\cdot L)\leq 3$, and thus $\mathsf{c}(I\cdot J,(2\mathcal{O}_f)^n\cdot L)\leq 4$ and $\ell-2\leq 1$.

\smallskip
CASE 2: $n\not\in\mathcal{B}$ and $n\not=2{\rm v}_2(f)$. It follows by Theorem~\ref{theorem 3.6} that there is some $A\in\mathcal{A}(\mathcal{I}^*_2(\mathcal{O}_f))$ such that $A\overline{A}=(2\mathcal{O}_f)^{n-1}$. We infer by Lemma~\ref{lemma 4.10}.2 that $\mathsf{c}((2\mathcal{O}_f)\cdot A\cdot\overline{A}\cdot L,(2\mathcal{O}_f)^n\cdot L)\leq \mathsf{c}(\cdot A\cdot\overline{A},(2\mathcal{O}_f)^{n-1})\leq 4$. Furthermore, $\mathsf{d}(I\cdot J,(2\mathcal{O}_f)\cdot A\cdot\overline{A}\cdot L)\leq 4$, and thus $\mathsf{c}(I\cdot J,(2\mathcal{O}_f)^n\cdot L)\leq 4$ and $\ell-2\leq 2$.

\smallskip
CASE 3: $n=2{\rm v}_2(f)$. By Theorem~\ref{theorem 3.6} there is some $D\in\mathcal{A}(\mathcal{I}^*_2(\mathcal{O}_f))$ such that $D\overline{D}=(2\mathcal{O}_f)^{n-2}$. Set $A=16\mathbb{Z}+(4+\tau)\mathbb{Z}$, $B=2^{n-2}\mathbb{Z}+(2^{n-4}+\tau)\mathbb{Z}$ and $C=2^{n-2}\mathbb{Z}+(2^{n-3}+\tau)\mathbb{Z}$. Then $A,B,C\in\mathcal{A}(\mathcal{I}^*_2(\mathcal{O}_f))$ and $ABC=2^{n-4}A(16\mathbb{Z}+(12+\tau)\mathbb{Z})=(2\mathcal{O}_f)^n$. This implies that $\mathsf{c}((2\mathcal{O}_f)^2\cdot D\cdot\overline{D}\cdot L,(2\mathcal{O}_f)^n\cdot L)\leq\mathsf{c}(D\cdot\overline{D},(2\mathcal{O}_f)^{n-2})\leq 4$ by Lemma~\ref{lemma 4.10}.2. Moreover, $\mathsf{d}(A\cdot B\cdot C\cdot L,(2\mathcal{O}f)^2\cdot D\cdot\overline{D}\cdot L)\leq 4$ and $\mathsf{d}(I\cdot J,A\cdot B\cdot C\cdot L)\leq 4$. Consequently, $\mathsf{c}(I\cdot J,(2\mathcal{O}_f)^n\cdot L)\leq 4$ and $\ell-2\leq 2$.
\end{proof}

\begin{proposition}\label{proposition 4.12}
Let ${\rm v}_2(f)\in\{2,3\}$ and $d\equiv 1\mod 8$. Then $3\in\Delta(\mathcal{I}^*_2(\mathcal{O}_f))$ and $5\in {\rm Ca}(\mathcal{I}^*_2(\mathcal{O}_f))$.
\end{proposition}

\begin{proof}
We distinguish two cases.

CASE 1: ${\rm v}_2(f)=2$. By Theorem~\ref{theorem 3.6} there is some $I\in\mathcal{A}(\mathcal{I}^*_2(\mathcal{O}_f))$ such that $\mathcal{N}(I)=32$. Set $J=\overline{I}$. Then $IJ=32\mathcal{O}_f$, and hence $\{2,5\}\subset\mathsf{L}(IJ)\subset [2,5]$. Again by Theorem~\ref{theorem 3.6} we have $\mathcal{N}(L)\in\{4\}\cup\{2^n\mid n\in\mathbb{N}_{\geq 5}\}$ for all $L\in\mathcal{A}(\mathcal{I}^*_2(\mathcal{O}_f))$. Note that if $A,B,C,D\in\mathcal{A}(\mathcal{I}^*_2(\mathcal{O}_f))$, then $\mathcal{N}(ABCD)\in\{256\}\cup\mathbb{N}_{\geq 2048}$. Since $\mathcal{N}(IJ)=1024$, we have $4\not\in\mathsf{L}(IJ)$. Assume that $3\in\mathsf{L}(IJ)$. Then there are some $A,B,C\in\mathcal{A}(\mathcal{I}^*_2(\mathcal{O}_f))$ such that $IJ=ABC$ and $\mathcal{N}(A)\leq\mathcal{N}(B)\leq\mathcal{N}(C)$. Therefore, $\mathcal{N}(A)=\mathcal{N}(B)=4$ and $\mathcal{N}(C)=64$. We infer by Lemma~\ref{lemma 4.2}.2 that $ABC=4L$ for some $L\in\mathcal{A}(\mathcal{I}^*_2(\mathcal{O}_f))$, and hence $L=8\mathcal{O}_f$, a contradiction. We have $\mathsf{L}(IJ)=\{2,5\}$, and thus $3\in\Delta(\mathcal{I}^*_2(\mathcal{O}_f))$ and $5=\mathsf{c}(IJ)\in {\rm Ca}(\mathcal{I}^*_2(\mathcal{O}_f))$.

\smallskip
CASE 2: ${\rm v}_2(f)=3$. By Proposition~\ref{proposition 3.3}.3 we can assume without restriction that $f=8$. By Theorem~\ref{theorem 3.6} there are some $I,J\in\mathcal{A}(\mathcal{I}^*_2(\mathcal{O}_f))$ such that $\mathcal{N}(I)=128$ and $\mathcal{N}(J)=16$. We have $I\overline{I}=128\mathcal{O}_f$ and $J\overline{J}=16\mathcal{O}_f$, and hence $I\overline{I}=8J\overline{J}$. This implies that $\{2,5\}\subset\mathsf{L}(I\overline{I})$. It follows from Theorem~\ref{theorem 3.6} that $\mathcal{N}(L)\in\{4,16\}\cup\{2^n\mid n\in\mathbb{N}_{\geq 7}\}$ for all $L\in\mathcal{A}(\mathcal{I}^*_2(\mathcal{O}_f))$.

First assume that $3\in\mathsf{L}(I\overline{I})$. Then there exist $A,B,C\in\mathcal{A}(\mathcal{I}^*_2(\mathcal{O}_f))$ such that $I\overline{I}=ABC$, and $\mathcal{N}(A)\leq\mathcal{N}(B)\leq\mathcal{N}(C)$. Therefore, $(\mathcal{N}(A),\mathcal{N}(B),\mathcal{N}(C))\in\{(4,16,256),(4,4,1024)\}$. If $(\mathcal{N}(A),\mathcal{N}(B),\mathcal{N}(C))=(4,16,256)$, then it follows by Lemma~\ref{lemma 4.2}.2 that $AB=2D$ for some $D\in\mathcal{A}(\mathcal{I}^*_2(\mathcal{O}_f))$ with $\mathcal{N}(D)=16$. We infer that $DC=64\mathcal{O}_f$, and hence $C=4\overline{D}$, a contradiction. Now let $(\mathcal{N}(A),\mathcal{N}(B),\mathcal{N}(C))=(4,4,1024)$. Then $ABC=4D$ for some $D\in\mathcal{A}(\mathcal{I}^*_2(\mathcal{O}_f))$ by Lemma~\ref{lemma 4.2}.2, and thus $D=32\mathcal{O}_f$, a contradiction. Consequently, $3\not\in\mathsf{L}(I\overline{I})$.

Next assume that $4\in\mathsf{L}(I\overline{I})$. Then there exist $A,B,C,D\in\mathcal{A}(\mathcal{I}^*_2(\mathcal{O}_f))$ such that $I\overline{I}=ABCD$, and $\mathcal{N}(A)\leq\mathcal{N}(B)\leq\mathcal{N}(C)\leq\mathcal{N}(D)$.

Then $(\mathcal{N}(A),\mathcal{N}(B),\mathcal{N}(C),\mathcal{N}(D))\in\{(4,4,4,256),(4,16,16,16)\}$.

If $(\mathcal{N}(A),\mathcal{N}(B),\mathcal{N}(C),\mathcal{N}(D))=(4,4,4,256)$, then $ABCD=8E$ for $E\in\mathcal{A}(\mathcal{I}^*_2(\mathcal{O}_f))$ by Lemma~\ref{lemma 4.2}.2, and hence $E=16\mathcal{O}_f$, a contradiction. Now let $(\mathcal{N}(A),\mathcal{N}(B),\mathcal{N}(C),\mathcal{N}(D))=(4,16,16,16)$. By Lemma~\ref{lemma 4.2}.2 there is some $E\in\mathcal{A}(\mathcal{I}^*_2(\mathcal{O}_f))$ with $\mathcal{N}(E)=16$ such that $AB=2E$. Therefore, $ECD=64\mathcal{O}_f$, and hence $CD=4\overline{E}$. There are some $(0,4,r),(0,4,s)\in\mathcal{M}_{f,2}$ such that $C=16\mathbb{Z}+(r+\tau)\mathbb{Z}$ and $D=16\mathbb{Z}+(s+\tau)\mathbb{Z}$. We have ${\rm v}_2(r^2-16d)={\rm v}_2(s^2-16d)=4$. Since $d\equiv 1\mod 8$, this implies that ${\rm v}_2(r),{\rm v}_2(s)\geq 3$. Therefore, $\min\{4,{\rm v}_2(r+s+\varepsilon)\}\in\{3,4\}$, and hence $CD=8F$ for some $F\in\mathcal{A}(\mathcal{I}^*_2(\mathcal{O}_f))$. We infer that $\overline{E}=2F$, a contradiction. Consequently, $4\not\in\mathsf{L}(I\overline{I})$.

Therefore, $2$ and $5$ are adjacent lengths of $I\overline{I}$, and hence $3\in\Delta(\mathcal{I}^*_2(\mathcal{O}_f))$. Note that $\mathsf{c}(\mathcal{I}^*_2(\mathcal{O}_f))\leq 5$ by Proposition~\ref{proposition 4.11}.1 and Theorem~\ref{theorem 3.6}. Moreover, since $\mathcal{I}^*_2(\mathcal{O}_f)$ is a cancellative monoid, we have $5\leq 2+\sup\Delta(\mathsf{L}(I\overline{I}))\leq\mathsf{c}(I\overline{I})\leq 5$, and thus $5=\mathsf{c}(I\overline{I})\in {\rm Ca}(\mathcal{I}^*_2(\mathcal{O}_f))$.
\end{proof}

\begin{lemma}\label{lemma 4.13}
Let $H\in\{\mathcal{I}(\mathcal{O}_f),\mathcal{I}^*(\mathcal{O}_f)\}$. For every prime divisor $p$ of $f$, we set $H_p=\mathcal{I}_p(\mathcal{O}_f)$ if $H=\mathcal{I}(\mathcal{O}_f)$ and $H_p=\mathcal{I}^*_p(\mathcal{O}_f)$ if $H=\mathcal{I}^*(\mathcal{O}_f)$.
\begin{enumerate}
\item[\textnormal{1.}] $H$ is half-factorial if and only if $H_p$ is half-factorial for every $p\in\mathbb{P}$ with $p\mid f$.

\item[\textnormal{2.}] If $H$ is not half-factorial, then $\sup\Delta(H)=\sup\{\sup\Delta(H_p)\mid p\in\mathbb{P}$ with $p\mid f\}$.

\item[\textnormal{3.}] $\mathsf{c}(H)=\sup\{\mathsf{c}(H_p)\mid p\in\mathbb{P}$ with $p\mid f\}$.
\end{enumerate}
\end{lemma}

\begin{proof}
By Equations~\ref{equation 3} and ~\ref{equation 4}, we have
\[
\mathcal{I}^*(\mathcal{O}_f)\cong\coprod_{P\in\mathfrak{X}(\mathcal{O}_f)}\mathcal{I}^*_P(\mathcal{O}_f) \quad \text{ and } \quad \mathcal{I}(\mathcal{O}_f)\cong\coprod_{P\in\mathfrak{X}(\mathcal{O}_f)}\mathcal{I}_P(\mathcal{O}_f) \,.
\]
Thus the assertions are easy consequences (see \cite[Propositions 1.4.5.3 and 1.6.8.1]{Ge-HK06a}).
\end{proof}

\begin{proof}[Proof of Theorem~\ref{theorem 1.1}]
1. This is an immediate consequence of Proposition~\ref{proposition 4.6} and Lemma~\ref{lemma 4.13}.

\smallskip
2. First, suppose that $f$ is squarefree. By 1., we have $f$ is not a product of inert primes. It follows from Lemma~\ref{lemma 4.13}, Proposition~\ref{proposition 4.11}.1 and Theorem~\ref{theorem 3.6} that $\mathsf{c}(\mathcal{I}^*(\mathcal{O}))\leq\mathsf{c}(\mathcal{I}(\mathcal{O}))\leq 3$ and $\sup\Delta(\mathcal{I}^*(\mathcal{O}))\leq\sup\Delta(\mathcal{I}(\mathcal{O}))\leq 1$. By Lemma~\ref{lemma 4.2} and Proposition~\ref{proposition 4.8}.1, it follows that $1\in\Delta(\mathcal{I}^*(\mathcal{O}))$, $1\in {\rm Ca}(\mathcal{I}(\mathcal{O}))$ and $[2,3]\subset {\rm Ca}(\mathcal{I}^*(\mathcal{O}))$, and thus ${\rm Ca}(\mathcal{I}(\mathcal{O}))=[1,3]$, ${\rm Ca}(\mathcal{I}^*(\mathcal{O}))=[2,3]$, and $\Delta(\mathcal{I}(\mathcal{O}))=\Delta(\mathcal{I}^*(\mathcal{O}))=\{1\}$.

\smallskip
Now we suppose that $f$ is not squarefree and we distinguish two cases.

\smallskip
CASE 1: ${\rm v}_2\left(f\right)\not\in\{2,3\}$ or $d_K\not\equiv 1\mod 8$. By Lemma~\ref{lemma 4.13}, Proposition~\ref{proposition 4.11} and Theorem~\ref{theorem 3.6} it follows that $\mathsf{c}(\mathcal{I}^*(\mathcal{O}))\leq\mathsf{c}(\mathcal{I}(\mathcal{O}))\leq 4$ and $\sup\Delta(\mathcal{I}^*(\mathcal{O}))\leq\sup\Delta(\mathcal{I}(\mathcal{O}))\leq 2$. We infer by Lemma~\ref{lemma 4.2} and Propositions~\ref{proposition 4.4} and ~\ref{proposition 4.8} that $[1,2]\subset\Delta(\mathcal{I}^*(\mathcal{O}))$, $1\in {\rm Ca}(\mathcal{I}(\mathcal{O}))$, and $[2,4]\subset {\rm Ca}(\mathcal{I}^*(\mathcal{O}))$, and hence ${\rm Ca}(\mathcal{I}(\mathcal{O}))=[1,4]$, ${\rm Ca}(\mathcal{I}^*(\mathcal{O}))=[2,4]$, and $\Delta(\mathcal{I}(\mathcal{O}))=\Delta(\mathcal{I}^*(\mathcal{O}))=[1,2]$.

\smallskip
CASE 2: ${\rm v}_2\left(f\right)\in\{2,3\}$ and $d_K\equiv 1\mod 8$. We infer by Lemma~\ref{lemma 4.13}, Proposition~\ref{proposition 4.11}.1 and Theorem~\ref{theorem 3.6} that $\mathsf{c}(\mathcal{I}^*(\mathcal{O}))\leq\mathsf{c}(\mathcal{I}(\mathcal{O}))\leq 5$ and $\sup\Delta(\mathcal{I}^*(\mathcal{O}))\leq\sup\Delta(\mathcal{I}(\mathcal{O}))\leq 3$. Lemma~\ref{lemma 4.2} and Propositions~\ref{proposition 4.4},~\ref{proposition 4.8} and~\ref{proposition 4.12} imply that $[1,3]\subset\Delta(\mathcal{I}^*(\mathcal{O}))$, $1\in {\rm Ca}(\mathcal{I}(\mathcal{O}))$ and $[2,5]\subset {\rm Ca}(\mathcal{I}^*(\mathcal{O}))$. Consequently, ${\rm Ca}(\mathcal{I}(\mathcal{O}))=[1,5]$, ${\rm Ca}(\mathcal{I}^*(\mathcal{O}))=[2,5]$, and $\Delta(\mathcal{I}(\mathcal{O}))=\Delta(\mathcal{I}^*(\mathcal{O}))=[1,3]$.
\end{proof}

Based on the results of this section we derive a result on the set of distances of orders. Let $\mathcal{O}$ be a non-half-factorial order in a number field. Then the set of distances $\Delta(\mathcal{O})$ is finite. If $\mathcal{O}$ is a principal order, then it is easy to show that $\min\Delta(\mathcal{O})=1$ (indeed much stronger results are known, namely that sets of lengths of almost all elements -- in a sense of density -- are intervals, see \cite[Theorem 9.4.11]{Ge-HK06a}). The same is true if $|{\rm Pic}(\mathcal{O})|\ge 3$ or if $\mathcal{O}$ is seminormal (\cite[Theorem 1.1]{Ge-Zh16c}). However, it was unknown so far whether there exists an order $\mathcal{O}$ with $\min\Delta(\mathcal{O})>1$. In the next result of this section we characterize all non-half-factorial orders in quadratic number fields with $\min\Delta(\mathcal{O})>1$ which allows us to give the first explicit examples of orders $\mathcal{O}$ with $\min\Delta(\mathcal{O})>1$. A characterization of half-factorial orders in quadratic number fields is given in \cite[Theorem 3.7.15]{Ge-HK06a}.

\smallskip
Let $\mathcal{O}$ be an order in a quadratic number field $K$ with conductor $f\in\mathbb{N}_{\ge 2}$. Then the class numbers $|{\rm Pic}(\mathcal{O}_K)|$ and $|{\rm Pic}(\mathcal{O})|$ are linked by the formula (\cite[Corollary 5.9.8]{HK13a})
\begin{equation}\label{equation 6}
|{\rm Pic}(\mathcal{O})|=|{\rm Pic}(\mathcal{O}_K)|\frac{f}{(\mathcal{O}_K^{\times}:\mathcal{O}^{\times})}\prod_{p\in\mathbb{P},p\mid f}\left(1-\Big(\frac{d_K}{p}\Big) p^{-1}\right),
\end{equation}

and $|{\rm Pic}(\mathcal{O})|$ is a multiple of $|{\rm Pic}(\mathcal{O}_K)|$.

\smallskip
Since the number of imaginary quadratic number fields with class number at most two is finite (an explicit list of these fields can be found, for example, in \cite{Ri01a}), \eqref{equation 6} shows that the number of orders in imaginary quadratic number fields with $|{\rm Pic}(\mathcal{O})|=2$ is finite. The complete list of non-maximal orders in imaginary quadratic number fields with $|{\rm Pic}(\mathcal{O})|=2$ is given in \cite[page 16]{Kl12a}. We refer to \cite{HK13a} for more information on class groups and class numbers and end with explicit examples of non-half-factorial orders $\mathcal{O}$ satisfying $\min\Delta(\mathcal{O})>1$.

\smallskip
\begin{theorem}\label{theorem 4.14}
Let $\mathcal{O}$ be a non-half-factorial order in a quadratic number field $K$ with conductor $f\mathcal{O}_K$ for some $f\in\mathbb{N}_{\geq 2}$. Then the following statements are equivalent{\rm \,:}
\begin{itemize}
\item[\textnormal{(a)}] $\min\Delta(\mathcal{O})>1$.

\item[\textnormal{(b)}] $|{\rm Pic}(\mathcal{O})|=2$, $f$ is a nonempty squarefree product of ramified primes times a $($possibly empty$)$ squarefree product of inert primes, and for every prime divisor $p$ of $f$ and every $I\in\mathcal{A}(\mathcal{I}^*_p(\mathcal{O}))$, $I$ is principal if and only if $\mathcal{N}(I)=p^2$.
\end{itemize}
If these equivalent conditions are satisfied, then $K$ is a real quadratic number field and $\min\Delta(\mathcal{O})=2$.
\end{theorem}

\begin{proof}
CLAIM: If $|{\rm Pic}(\mathcal{O})|=2$, $p$ is a ramified prime with ${\rm v}_p(f)=1$, and every $I\in\mathcal{A}(\mathcal{I}^*_p(\mathcal{O}))$ with $\mathcal{N}(I)=p^3$ is not principal, then every $L\in\mathcal{A}(\mathcal{I}^*_p(\mathcal{O}))$ with $\mathcal{N}(L)=p^2$ is principal.

\smallskip
Let $|{\rm Pic}(\mathcal{O})|=2$, let $p$ be a ramified prime with ${\rm v}_p(f)=1$, and suppose that every $I\in\mathcal{A}(\mathcal{I}^*_p(\mathcal{O}))$ with $\mathcal{N}(I)=p^3$ is not principal. By Theorem~\ref{theorem 3.6} we have $\{\mathcal{N}(J)\mid J\in\mathcal{A}(\mathcal{I}^*_p(\mathcal{O}))\}=\{p^2,p^3\}$. There is some $I\in\mathcal{A}(\mathcal{I}^*_p(\mathcal{O}))$ such that $\mathcal{N}(I)=p^3$. If $J\in\mathcal{A}(\mathcal{I}^*_p(\mathcal{O}))$ with $\mathcal{N}(J)=p^3$, then $IJ=p^2L$ for some $L\in\mathcal{A}(\mathcal{I}^*_p(\mathcal{O}))$ with $\mathcal{N}(L)=p^2$ (since there are no atoms with norm bigger than $p^3$). It follows by Theorem~\ref{theorem 3.6} that $|\{J\in\mathcal{A}(\mathcal{I}^*_p(\mathcal{O}))\mid\mathcal{N}(J)=p^3\}|=|\{L\in\mathcal{A}(\mathcal{I}^*_p(\mathcal{O}))\mid\mathcal{N}(L)=p^2\}|=p$ (note that $\mathcal{N}(p\mathcal{O})=p^2$). Let $g:\{J\in\mathcal{A}(\mathcal{I}^*_p(\mathcal{O}))\mid\mathcal{N}(J)=p^3\}\rightarrow\{L\in\mathcal{A}(\mathcal{I}^*_p(\mathcal{O}))\mid\mathcal{N}(L)=p^2\}$ be defined by $g(J)=L$ where $L\in\mathcal{A}(\mathcal{I}^*_p(\mathcal{O}))$ is such that $\mathcal{N}(L)=p^2$ and $IJ=p^2L$. Then $g$ is a well-defined bijection. Now let $L\in\mathcal{A}(\mathcal{I}^*_p(\mathcal{O}))$ with $\mathcal{N}(L)=p^2$. There is some $J\in\mathcal{A}(\mathcal{I}^*_p(\mathcal{O}))$ such that $\mathcal{N}(J)=p^3$ and $IJ=p^2L$. Since $|{\rm Pic}(\mathcal{O})|=2$ and $I$ and $J$ are not principal, we have $IJ$ is principal, and hence $L$ is principal. This proves the claim.

\smallskip
(a) $\Rightarrow$ (b) Observe that if $p$ is an inert prime such that ${\rm v}_p(f)=1$, then $\{\mathcal{N}(J)\mid J\in\mathcal{A}(\mathcal{I}^*_p(\mathcal{O}))\}=\{p^2\}$ by Theorem~\ref{theorem 3.6}. Also note that if $p$ is a ramified prime such that ${\rm v}_p(f)=1$, then $\{\mathcal{N}(J)\mid J\in\mathcal{A}(\mathcal{I}^*_p(\mathcal{O}))\}=\{p^2,p^3\}$ by Theorem~\ref{theorem 3.6}. The assertion now follows by the claim and Proposition~\ref{proposition 4.8}.2.

\smallskip
(b) $\Rightarrow$ (a) Assume to the contrary that $\min\Delta(\mathcal{O})=1$. Let $\mathcal{H}$ be the monoid of nonzero principal ideals of $\mathcal{O}$. There is some minimal $k\in\mathbb{N}$ such that $\prod_{i=1}^k U_i=\prod_{j=1}^{k+1} U^{\prime}_j$ with $U_i\in\mathcal{A}(\mathcal{H})$ for each $i\in [1,k]$ and $U^{\prime}_j\in\mathcal{A}(\mathcal{H})$ for each $j\in [1,k+1]$.

Set $\mathcal{Q}_1=\{P\in\mathfrak{X}(\mathcal{O})\mid P$ is principal$\}$, $\mathcal{Q}_2=\{P\in\mathfrak{X}(\mathcal{O})\mid P$ is invertible and not principal$\}$, $\mathcal{L}=\{p\in\mathbb{P}\mid p\mid f,p$ is ramified$\}$ and $\mathcal{K}=\{\{p,q\}\mid p,q\in\mathcal{L},p\not=q\}$. For every prime divisor $p$ of $f$ set $\mathcal{A}_p=\{V\in\mathcal{A}(\mathcal{I}^*_p(\mathcal{O}))\mid\mathcal{N}(V)=p^2\}$, $a_p=|\{i\in [1,k]\mid U_i\in\mathcal{A}_p\}|$ and $a^{\prime}_p=|\{j\in [1,k+1]\mid U^{\prime}_j\in\mathcal{A}_p\}|$. For $p\in\mathcal{L}$ set $\mathcal{D}_p=\{V\in\mathcal{A}(\mathcal{I}^*_p(\mathcal{O}))\mid\mathcal{N}(V)=p^3\}$, $\mathcal{B}_p=\{PV\mid P\in\mathcal{Q}_2$ and $V\in\mathcal{D}_p\}$, $b_p=|\{i\in [1,k]\mid U_i\in\mathcal{B}_p\}|$ and $b^{\prime}_p=|\{j\in [1,k+1]\mid U^{\prime}_j\in\mathcal{B}_p\}|$. Set $\mathcal{C}=\{PQ\mid P,Q\in\mathcal{Q}_2\}$, $c=|\{i\in [1,k]\mid U_i\in\mathcal{C}\}|$ and $c^{\prime}=|\{j\in [1,k+1]\mid U^{\prime}_j\in\mathcal{C}\}|$. If $z\in\mathcal{K}$ is such that $z=\{p,q\}$ with $p,q\in\mathcal{L}$ and $p\not=q$, then set $\mathcal{E}_z=\{VW\mid V\in\mathcal{D}_p,W\in\mathcal{D}_q\}$, $e_z=|\{i\in [1,k]\mid U_i\in\mathcal{E}_z\}|$ and $e^{\prime}_z=|\{j\in [1,k+1]\mid U^{\prime}_j\in\mathcal{E}_z\}|$.

Since $|{\rm Pic}(\mathcal{O})|=2$, we have $\mathcal{A}(\mathcal{H})\subset (\mathcal{A}(\mathcal{I}^*(\mathcal{O}))\cap\mathcal{H})\cup\{VW\mid V,W\in\mathcal{A}(\mathcal{I}^*(\mathcal{O})),V$ and $W$ are not principal$\}$. As shown in the proof of the claim, $VW\not\in\mathcal{A}(\mathcal{H})$ for all $p\in\mathcal{L}$ and $V,W\in\mathcal{D}_p$. We infer that $\mathcal{A}(\mathcal{H})=\mathcal{Q}_1\cup\bigcup_{p\in\mathbb{P},p\mid f}\mathcal{A}_p\cup\bigcup_{p\in\mathcal{L}}\mathcal{B}_p\cup\mathcal{C}\cup\bigcup_{z\in\mathcal{K}}\mathcal{E}_z$.

Since $k$ is minimal, we have $U_i,U^{\prime}_j\not\in\mathcal{Q}_1$ for all $i\in [1,k]$ and $j\in [1,k+1]$. Again since $k$ is minimal and $\mathcal{I}^*_p(\mathcal{O})$ is half-factorial for all inert prime divisors $p$ of $f$ by Proposition~\ref{proposition 4.6}, we have $a_p=a^{\prime}_p=0$ for all inert prime divisors $p$ of $f$. Therefore,

\[
k=\sum_{p\in\mathcal{L}} (a_p+b_p)+c+\sum_{z\in\mathcal{K}} e_z\textnormal{ and }k+1=\sum_{p\in\mathcal{L}} (a^{\prime}_p+b^{\prime}_p)+c^{\prime}+\sum_{z\in\mathcal{K}} e^{\prime}_z.
\]

If $i\in [1,k]$, then $\sum_{P\in\mathcal{Q}_2} {\rm v}_P(U_i)=\begin{cases} 2 & \textnormal{ if } U_i\in\mathcal{C}\\ 1 & \textnormal{ if } U_i\in\bigcup_{p\in\mathcal{L}}\mathcal{B}_p\\ 0 &\textnormal{ else}\end{cases}$. This implies that $\sum_{P\in\mathcal{Q}_2} {\rm v}_P(\prod_{i=1}^k U_i)=\sum_{i=1}^k\sum_{P\in\mathcal{Q}_2}{\rm v}_P(U_i)=\sum_{p\in\mathcal{L}} b_p+2c$. It follows by analogy that $\sum_{P\in\mathcal{Q}_2} {\rm v}_P(\prod_{j=1}^{k+1} U^{\prime}_j)=\sum_{p\in\mathcal{L}} b^{\prime}_p+2c^{\prime}$. Therefore, $\sum_{p\in\mathcal{L}} b_p+2c=\sum_{p\in\mathcal{L}} b^{\prime}_p+2c^{\prime}$. Let $r\in\mathcal{L}$.

If $i\in [1,k]$, then ${\rm v}_r(\mathcal{N}((U_i)_{P_{f,r}}\cap\mathcal{O}))=\begin{cases} 3 & \textnormal{ if } U_i\in\mathcal{B}_r\cup\bigcup_{q\in\mathcal{L}\setminus\{r\}}\mathcal{E}_{\{r,q\}}\\ 2 & \textnormal{ if } U_i\in\mathcal{A}_r\\ 0 &\textnormal{ else}\end{cases}$. Consequently,

\[
{\rm v}_r(\mathcal{N}((\prod_{i=1}^k U_i)_{P_{f,r}}\cap\mathcal{O}))=\sum_{i=1}^k {\rm v}_r(\mathcal{N}((U_i)_{P_{f,r}}\cap\mathcal{O}))=2a_r+3b_r+3\sum_{q\in\mathcal{L}\setminus\{r\}} e_{\{r,q\}}.
\]

By analogy we have ${\rm v}_r(\mathcal{N}((\prod_{j=1}^{k+1} U^{\prime}_j)_{P_{f,r}}\cap\mathcal{O}))=2a^{\prime}_r+3b^{\prime}_r+3\sum_{q\in\mathcal{L}\setminus\{r\}} e^{\prime}_{\{r,q\}}$. This implies that $2a_r+3b_r+3\sum_{q\in\mathcal{L}\setminus\{r\}} e_{\{r,q\}}=2a^{\prime}_r+3b^{\prime}_r+3\sum_{q\in\mathcal{L}\setminus\{r\}} e^{\prime}_{\{r,q\}}$. We infer that
\begin{align*}
&\sum_{p\in\mathcal{L}} (a^{\prime}_p-a_p+b^{\prime}_p-b_p)+c^{\prime}-c+\sum_{z\in\mathcal{K}} (e^{\prime}_z-e_z)=1,\textnormal{ }\sum_{p\in\mathcal{L}} (b^{\prime}_p-b_p)=2(c-c^{\prime})\\
&\textnormal{and }2\sum_{p\in\mathcal{L}} (a^{\prime}_p-a_p)+3\sum_{p\in\mathcal{L}} (b^{\prime}_p-b_p)+3\sum_{p\in\mathcal{L}}\sum_{q\in\mathcal{L}\setminus\{p\}} (e^{\prime}_{\{p,q\}}-e_{\{p,q\}})=0.
\end{align*}

Note that $\sum_{p\in\mathcal{L}}\sum_{q\in\mathcal{L}\setminus\{p\}} (e^{\prime}_{\{p,q\}}-e_{\{p,q\}})=2\sum_{z\in\mathcal{K}} (e^{\prime}_z-e_z)$, and hence $\sum_{p\in\mathcal{L}} (a^{\prime}_p-a_p)=3(c^{\prime}-c)-3\sum_{z\in\mathcal{K}} (e^{\prime}_z-e_z)$. Consequently,
\begin{align*}
1&=\sum_{p\in\mathcal{L}} (a^{\prime}_p-a_p+b^{\prime}_p-b_p)+c^{\prime}-c+\sum_{z\in\mathcal{K}} (e^{\prime}_z-e_z)\\
&=3(c^{\prime}-c)-3\sum_{z\in\mathcal{K}} (e^{\prime}_z-e_z)+2(c-c^{\prime})+c^{\prime}-c+\sum_{z\in\mathcal{K}} (e^{\prime}_z-e_z)\\
&=2(c^{\prime}-c-\sum_{z\in\mathcal{K}} (e^{\prime}_z-e_z)),
\end{align*}
a contradiction.

\smallskip
Now let the equivalent conditions be satisfied. Assume to the contrary that $K$ is an imaginary quadratic number field. Since $\mathcal{O}$ is a non-maximal order with $|{\rm Pic}(\mathcal{O})|=2$, it follows from \cite[page 16]{Kl12a} that $(f,d_K)\in\{(2,-8),(2,-15),(3,-4),(3,-8),(3,-11),(4,-3),(4,-4),(4,-7),(5,-3),(5,-4),(7,-3)\}$.

Since $f$ is squarefree and divisible by a ramified prime, we infer that $f=2$ and $d_K=-8$. Therefore, $\mathcal{O}=\mathbb{Z}+2\sqrt{-2}\mathbb{Z}$. Set $I=8\mathbb{Z}+2\sqrt{-2}\mathbb{Z}$. Observe that $I\in\mathcal{A}(\mathcal{I}^*_2(\mathcal{O}))$ and $\mathcal{N}(I)=8$. Moreover, $I=2\sqrt{-2}\mathcal{O}$ is principal, a contradiction. Consequently, $K$ is a real quadratic number field.

\smallskip
It remains to show that $\min\Delta(\mathcal{O})=2$. There is some ramified prime $p$ which divides $f$ and there is some $J\in\mathcal{A}(\mathcal{I}^*_p(\mathcal{O}))$ with $\mathcal{N}(J)=p^3$. As shown in the proof of the claim, $J^2=p^2L$ for some $L\in\mathcal{A}(\mathcal{I}^*_p(\mathcal{O}))$. By \cite[Corollary 2.11.16]{Ge-HK06a}, there is some invertible prime ideal $P$ of $\mathcal{O}$ that is not principal. Observe that $J$ is not principal. We have $PJ$, $P^2$ and $L$ are principal, and hence there are some $u,v,w\in\mathcal{A}(\mathcal{O})$ such that $PJ=u\mathcal{O}$, $P^2=v\mathcal{O}$, $L=w\mathcal{O}$, and $u^2=p^2vw$. Therefore, $\{2,4\}\subseteq\mathsf{L}(u^2)$, and since $\min\Delta(\mathcal{O})>1$, we infer that $\min\Delta(\mathcal{O})=2$.
\end{proof}

\begin{proposition}\label{proposition 4.15} Let $\mathcal{O}$ be an order in the quadratic number field $K$ with conductor $f\mathcal{O}_K$ for some $f\in\mathbb{N}_{\geq 2}$ such that $\min\Delta(\mathcal{O})>1$, let $g$ be the product of all inert prime divisors of $f$ and let $\mathcal{O}^{\prime}$ be the order in $K$ with conductor $g\mathcal{O}_K$. Then $\mathcal{O}^{\prime}$ is half-factorial and, in particular, $g\in\{1\}\cup\mathbb{P}\cup\{2p\mid p\in\mathbb{P}\setminus\{2\}\}$.
\end{proposition}

\begin{proof} Set $\mathcal{Q}_1=\{P\in\mathfrak{X}(\mathcal{O}^{\prime})\mid P$ is principal$\}$ and $\mathcal{Q}_2=\{P\in\mathfrak{X}(\mathcal{O}^{\prime})\mid P$ is invertible and not principal$\}$. Observe that $\mathcal{N}(I)=|\mathcal{O}/I|=|\mathcal{O}^{\prime}/I\mathcal{O}^{\prime}|=\mathcal{N}(I\mathcal{O}^{\prime})$ for all $I\in\mathcal{I}^*(\mathcal{O})$. Note that for all inert prime divisors $p$ of $f$ and all $I\in\mathcal{A}(\mathcal{I}^*_p(\mathcal{O}))$ and $J\in\mathcal{A}(\mathcal{I}^*_p(\mathcal{O}^{\prime}))$, we have $\mathcal{N}(I)=\mathcal{N}(J)=p^2$. Moreover, for all ramified prime divisors $p$ of $f$, we have $\{\mathcal{N}(I)\mid I\in\mathcal{A}(\mathcal{I}^*_p(\mathcal{O}))\}=\{p^2,p^3\}$. In this proof we will use Theorem~\ref{theorem 4.14} without further citation.

\smallskip
CLAIM 1: For all prime divisors $p$ of $g$ and all $I\in\mathcal{A}(\mathcal{I}^*_p(\mathcal{O}^{\prime}))$, it follows that $I$ is principal. Let $p$ be a prime divisor of $g$ and let $I\in\mathcal{A}(\mathcal{I}^*_p(\mathcal{O}^{\prime}))$. Set $P=P_{f,p}$ and $P^{\prime}=P_{g,p}$. It follows by Proposition~\ref{proposition 3.3} that $\mathcal{O}_P=\mathcal{O}^{\prime}_{P^{\prime}}$ and that $\delta:\mathcal{I}^*_p(\mathcal{O})\rightarrow\mathcal{I}^*_p(\mathcal{O}^{\prime})$ defined by $\delta(J)=J_P\cap\mathcal{O}^{\prime}$ for all $J\in\mathcal{I}^*_p(\mathcal{O})$ is a monoid isomorphism. In particular, we have $\mathcal{A}(\mathcal{I}^*_p(\mathcal{O}^{\prime}))=\{J_P\cap\mathcal{O}^{\prime}\mid J\in\mathcal{A}(\mathcal{I}^*_p(\mathcal{O}))\}$. Therefore, there is some $J\in\mathcal{A}(\mathcal{I}^*_p(\mathcal{O}))$ such that $J_P\cap\mathcal{O}^{\prime}=I$. Note that $\mathcal{N}(I)=p^2=\mathcal{N}(J)=\mathcal{N}(J\mathcal{O}^{\prime})$. Since $J\mathcal{O}^{\prime}\subseteq J\mathcal{O}^{\prime}_{P^{\prime}}\cap\mathcal{O}^{\prime}=J\mathcal{O}_P\cap\mathcal{O}^{\prime}=I$, we infer that $I=J\mathcal{O}^{\prime}$. Since $J$ is a principal ideal of $\mathcal{O}$, it follows that $I$ is principal. This proves Claim 1.

\smallskip
CLAIM 2: If $P\in\mathcal{Q}_2$, $p$ is a ramified prime divisor of $f$ such that $P\cap\mathbb{Z}=p\mathbb{Z}$ and $I\in\mathcal{A}(\mathcal{I}^*_p(\mathcal{O}))$ with $\mathcal{N}(I)=p^3$, then $P^2$ is principal and $I\mathcal{O}^{\prime}=P^3$. Let $P\in\mathcal{Q}_2$, $p$ a ramified prime divisor of $f$ such that $P\cap\mathbb{Z}=p\mathbb{Z}$ and $I\in\mathcal{A}(\mathcal{I}^*_p(\mathcal{O}))$ with $\mathcal{N}(I)=p^3$. Since $p$ is ramified, there is some $A\in\mathfrak{X}(\mathcal{O}_K)$ such that $p\mathcal{O}_K=A^2$. Observe that $\mathcal{N}(A^2)=p^2$, and thus $\mathcal{N}(A)=p$. We have $A\cap\mathcal{O}^{\prime}=P$, $P\mathcal{O}_K=A$ and $\mathcal{N}(P)=\mathcal{N}(A)=p$. Note that since $P$ is invertible, it follows that every $P$-primary ideal of $\mathcal{O}^{\prime}$ is a power of $P$. Therefore, $p\mathcal{O}^{\prime}=P^k$ for some $k\in\mathbb{N}$, and hence $p^k=\mathcal{N}(P^k)=\mathcal{N}(p\mathcal{O}^{\prime})=p^2$. Consequently, $k=2$ and $P^2$ is principal. Clearly, $I\mathcal{O}^{\prime}$ is a $P$-primary ideal of $\mathcal{O}^{\prime}$, and thus $I\mathcal{O}^{\prime}=P^m$ for some $m\in\mathbb{N}$. We infer that $p^m=\mathcal{N}(P^m)=\mathcal{N}(I\mathcal{O}^{\prime})=\mathcal{N}(I)=p^3$, and thus $m=3$ and $I\mathcal{O}^{\prime}=P^3$. This proves Claim 2.

\smallskip
CLAIM 3: $PQ$ is principal for all $P,Q\in\mathcal{Q}_2$. Let $P,Q\in\mathcal{Q}_2$.

\smallskip
CASE 1: $P\cap\mathcal{O}$ and $Q\cap\mathcal{O}$ are invertible. Note that $P=(P\cap\mathcal{O})\mathcal{O}^{\prime}$, $Q=(Q\cap\mathcal{O})\mathcal{O}^{\prime}$ and $P\cap\mathcal{O}$ and $Q\cap\mathcal{O}$ are not principal. Since $|{\rm Pic}(\mathcal{O})|=2$, we have $(P\cap\mathcal{O})(Q\cap\mathcal{O})$ is a principal ideal of $\mathcal{O}$, and thus $PQ=(P\cap\mathcal{O})(Q\cap\mathcal{O})\mathcal{O}^{\prime}$ is principal.

\smallskip
CASE 2: ($P\cap\mathcal{O}$ is invertible and $Q\cap\mathcal{O}$ is not invertible) or ($P\cap\mathcal{O}$ is not invertible and $Q\cap\mathcal{O}$ is invertible). Without restriction let $P\cap\mathcal{O}$ be invertible and let $Q\cap\mathcal{O}$ be not invertible. Observe that $P=(P\cap\mathcal{O})\mathcal{O}^{\prime}$. Moreover, there is some ramified prime $q$ that divides $f$ such that $Q\cap\mathbb{Z}=q\mathbb{Z}$ and there is some $J\in\mathcal{A}(\mathcal{I}^*_q(\mathcal{O}))$ with $\mathcal{N}(J)=q^3$. Observe that $P\cap\mathcal{O}$ and $J$ are not principal. Since $|{\rm Pic}(\mathcal{O})|=2$, it follows that $(P\cap\mathcal{O})J$ is a principal ideal of $\mathcal{O}$. Note that $PQ^3=(P\cap\mathcal{O})J\mathcal{O}^{\prime}$ by Claim 2, and thus $PQ^3$ is principal. Since $Q^2$ is principal by Claim 2, we infer that $PQ$ is principal.

\smallskip
CASE 3: $P\cap\mathcal{O}$ and $Q\cap\mathcal{O}$ are not invertible. There are ramified primes $p$ and $q$ that divide $f$ such that $P\cap\mathbb{Z}=p\mathbb{Z}$ and $Q\cap\mathbb{Z}=q\mathbb{Z}$. There are some $I\in\mathcal{A}(\mathcal{I}^*_p(\mathcal{O}))$ and $J\in\mathcal{A}(\mathcal{I}^*_q(\mathcal{O}))$ with $\mathcal{N}(I)=p^3$ and $\mathcal{N}(J)=q^3$. Since $|{\rm Pic}(\mathcal{O})|=2$ and $I$ and $J$ are not principal, we have $IJ$ is a principal ideal of $\mathcal{O}$. It follows that $P^3Q^3=IJ\mathcal{O}^{\prime}$ by Claim 2, and hence $P^3Q^3$ is principal. Since $P^2$ and $Q^2$ are principal by Claim 2, we have $PQ$ is principal. This proves Claim 3.

\smallskip
Finally, we show that $\mathcal{O}^{\prime}$ is half-factorial. Set $\mathcal{C}=\{PQ\mid P,Q\in\mathcal{Q}_2\}$ and let $\mathcal{H}$ denote the monoid of nonzero principal ideals of $\mathcal{O}^{\prime}$. It is an immediate consequence of Claim 1 and Claim 3 that $\mathcal{A}(\mathcal{H})=\mathcal{Q}_1\cup\mathcal{C}\cup\bigcup_{p\in\mathbb{P},p\mid g}\mathcal{A}(\mathcal{I}^*_p(\mathcal{O}^{\prime}))$.

\smallskip
Let $k,\ell\in\mathbb{N}$ and $I_i,I_j^{\prime}\in\mathcal{A}(\mathcal{H})$ for each $i\in [1,k]$ and $j\in [1,\ell]$ be such that $\prod_{i=1}^k I_i=\prod_{j=1}^{\ell} I_j^{\prime}$. It remains to show that $k=\ell$. Set $b=|\{i\in [1,k]\mid I_i\in\mathcal{Q}_1\}|$, $b^{\prime}=|\{j\in [1,\ell]\mid I_j^{\prime}\in\mathcal{Q}_1\}|$, $c=|\{i\in [1,k]\mid I_i\in\mathcal{C}\}|$, $c^{\prime}=|\{j\in [1,\ell]\mid I_j^{\prime}\in\mathcal{C}\}|$ and for each prime divisor $p$ of $g$ set $a_p=|\{i\in [1,k]\mid I_i\in\mathcal{A}(\mathcal{I}^*_p(\mathcal{O}^{\prime}))\}|$ and $a^{\prime}_p=|\{j\in [1,\ell]\mid I_j^{\prime}\in\mathcal{A}(\mathcal{I}^*_p(\mathcal{O}^{\prime}))\}|$. If $p$ is a prime divisor of $g$, then $\mathcal{I}^*_p(\mathcal{O}^{\prime})$ is half-factorial by Proposition~\ref{proposition 4.6}, and hence $a_p=a^{\prime}_p$ by Claim 1. We have $b=\sum_{i=1}^k\sum_{P\in\mathcal{Q}_1}{\rm v}_P(I_i)=\sum_{P\in\mathcal{Q}_1}{\rm v}_P(\prod_{i=1}^k I_i)=\sum_{P\in\mathcal{Q}_1}{\rm v}_P(\prod_{j=1}^{\ell} I_j^{\prime})=\sum_{j=1}^{\ell}\sum_{P\in\mathcal{Q}_1}{\rm v}_P(I_j^{\prime})=b^{\prime}$.

Moreover, $2c=\sum_{P\in\mathcal{Q}_2}{\rm v}_P(\prod_{i=1}^k I_i)=\sum_{P\in\mathcal{Q}_2}{\rm v}_P(\prod_{j=1}^{\ell} I_j^{\prime})=2c^{\prime}$. Therefore, $k=b+c+\sum_{p\in\mathbb{P},p\mid g} a_p=b^{\prime}+c^{\prime}+\sum_{p\in\mathbb{P},p\mid g} a^{\prime}_p=\ell$.

\smallskip
The remaining assertion follows from \cite[Theorem 3.7.15]{Ge-HK06a}.
\end{proof}

\begin{remark}\label{remark 4.16} Let $\mathcal{O}$ be an order in the quadratic number field $K$ with conductor $f\mathcal{O}_K$ for some $f\in\mathbb{N}$ such that $|{\rm Pic}(\mathcal{O})|=2$ and let $p$ be an odd ramified prime such that ${\rm v}_p(f)=1$ and $I\in\mathcal{A}(\mathcal{I}^*_p(\mathcal{O}))$ such that $\mathcal{N}(I)=p^3$ and $I$ not principal. Then every $J\in\mathcal{A}(\mathcal{I}^*_p(\mathcal{O}))$ with $\mathcal{N}(J)=p^3$ is not principal.
\end{remark}

\begin{proof} Set $\mathcal{L}=\{J\in\mathcal{A}(\mathcal{I}^*_p(\mathcal{O}))\mid\mathcal{N}(J)=p^3\}$ and $\mathcal{K}=\{L\in\mathcal{A}(\mathcal{I}^*_p(\mathcal{O}))\mid\mathcal{N}(L)=p^2\}$. It follows by the claim in the proof of Theorem~\ref{theorem 4.14} that for all $J\in\mathcal{L}$ and $L\in\mathcal{K}$, there is a unique $A\in\mathcal{L}$ such that $AJ=p^2L$. By Theorem~\ref{theorem 3.6} we have $|\mathcal{L}|=|\mathcal{K}|=p$, and hence $|\{(A,J)\in\mathcal{L}^2\mid AJ=p^2L\}|=p$ for all $L\in\mathcal{K}$. Since $p$ is odd, we infer that for each $L\in\mathcal{K}$ there is some $A\in\mathcal{L}$ such that $A^2=p^2L$. Consequently, every $L\in\mathcal{K}$ is principal. Now let $J\in\mathcal{L}$. There is some $B\in\mathcal{K}$ such that $IJ=p^2B$, and thus $IJ$ is principal. Therefore, $J$ is not principal.
\end{proof}

Next we show that the assumption that $p$ is odd in Remark~\ref{remark 4.16} is crucial.

\begin{example}\label{example 4.17} Let $\mathcal{O}=\mathbb{Z}+2\sqrt{-2}\mathbb{Z}$ be the order in the quadratic number field $K=\mathbb{Q}(\sqrt{-2})$ with conductor $2\mathcal{O}_K$. Let $I=8\mathbb{Z}+2\sqrt{-2}\mathbb{Z}$ and $J=8\mathbb{Z}+(4+2\sqrt{-2})\mathbb{Z}$. Then $2$ is ramified, $|{\rm Pic}(\mathcal{O})|=2$, $I,J\in\mathcal{A}(\mathcal{I}^*_2(\mathcal{O}))$, $\mathcal{N}(I)=\mathcal{N}(J)=8$, $I$ is principal and $J$ is not principal.
\end{example}

\begin{proof} It is clear that $J\in\mathcal{A}(\mathcal{I}^*_2(\mathcal{O}))$ and $\mathcal{N}(J)=8$. By the proof of Theorem~\ref{theorem 4.14}, it remains to show that $J$ is not principal. Assume that $J$ is principal. Then there are some $a,b\in\mathbb{Z}$ such that $J=(8a+4b+2\sqrt{-2}b)\mathcal{O}$, and hence $8=\mathcal{N}(J)=|\mathcal{N}_{K/\mathbb{Q}}(8a+4b+2\sqrt{-2}b)|=|(8a+4b)^2+8b^2|$. Therefore, $2(2a+b)^2+b^2=1$. It is clear that $|b|\leq 1$. If $b=0$, then $8a^2=1$, a contradiction. Therefore, $|b|=1$ and $2a+b=0$, a contradiction.
\end{proof}

\begin{lemma}\label{lemma 4.18} Let $d\in\mathbb{N}_{\geq 2}$ be squarefree, let $K=\mathbb{Q}(\sqrt{d})$, let $\mathcal{O}$ be the order in $K$ with conductor $f\mathcal{O}_K$ for some $f\in\mathbb{N}_{\geq 2}$, and let $p$ be a ramified prime with ${\rm v}_p(f)=1$. If $(p\equiv 1\mod 4$ and $(\frac{d/p}{p})=-1)$ or $((\frac{p}{q})=-1$ for some prime $q$ with $q\equiv 1\mod 4$ and $q\mid df)$, then each $I\in\mathcal{A}(\mathcal{I}^*_p(\mathcal{O}))$ with $\mathcal{N}(I)=p^3$ is not principal.
\end{lemma}

\begin{proof} Note that if $p$ is odd, then $\{I\in\mathcal{A}(\mathcal{I}^*_p(\mathcal{O}))\mid\mathcal{N}(I)=p^3\}=\{p^3\mathbb{Z}+(p^2k+\frac{\varepsilon p^2+f\sqrt{d_K}}{2})\mathbb{Z}\mid k\in [0,p-1]\}$. Moreover, if $p=2$ and $d$ is odd, then $\{I\in\mathcal{A}(\mathcal{I}^*_p(\mathcal{O}))\mid\mathcal{N}(I)=p^3\}=\{8\mathbb{Z}+(2k+f\sqrt{d})\mathbb{Z}\mid k\in\{1,3\}\}$. Furthermore, if $p=2$ and $d$ is even, then $\{I\in\mathcal{A}(\mathcal{I}^*_p(\mathcal{O}))\mid\mathcal{N}(I)=p^3\}=\{8\mathbb{Z}+(2k+f\sqrt{d})\mathbb{Z}\mid k\in\{0,2\}\}$.

\smallskip
CASE 1: $p\equiv 1\mod 4$ and $(\frac{d/p}{p})=-1$. Let $I\in\mathcal{A}(\mathcal{I}^*_p(\mathcal{O}))$ be such that $\mathcal{N}(I)=p^3$. Since $p$ is odd, we have $I=p^3\mathbb{Z}+(p^2k+\frac{\varepsilon p^2+f\sqrt{d_K}}{2})\mathbb{Z}$ for some $k\in [0,p-1]$. Assume that $I$ is principal. Then there are some $a,b\in\mathbb{Z}$ such that $I=(p^3a+p^2bk+\frac{\varepsilon p^2+f\sqrt{d_K}}{2}b)\mathcal{O}$. We infer that $p^3=\mathcal{N}(I)=|\mathcal{N}_{K/\mathbb{Q}}(p^3a+p^2bk+\frac{\varepsilon p^2+f\sqrt{d_K}}{2}b)|=\frac{1}{4}|p^4(2pa+2bk+\varepsilon b)^2-f^2b^2d_K|$, and hence $\frac{f^2}{p^2}b^2\frac{d_K}{p}\equiv 4\beta\mod p$ for some $\beta\in\{-1,1\}$. Since $p\equiv 1\mod 4$, we have $(\frac{-1}{p})=1$, and thus $(\frac{d/p}{p})=(\frac{d_K/p}{p})=(\frac{f^2b^2d_K/p^3}{p})=(\frac{4\beta}{p})=1$, a contradiction.

\smallskip
CASE 2: There is some prime $q$ such that $q\equiv 1\mod 4$, $q\mid df$ and $(\frac{p}{q})=-1$. Let $I\in\mathcal{A}(\mathcal{I}^*_p(\mathcal{O}))$ be such that $\mathcal{N}(I)=p^3$. First let $p$ be odd. Then $I=p^3\mathbb{Z}+(p^2k+\frac{\varepsilon p^2+f\sqrt{d_K}}{2})\mathbb{Z}$ for some $k\in [0,p-1]$. Assume that $I$ is principal. Then there are some $a,b\in\mathbb{Z}$ such that $I=(p^3a+p^2bk+\frac{\varepsilon p^2+f\sqrt{d_K}}{2}b)\mathcal{O}$. This implies that $p^3=\mathcal{N}(I)=|\mathcal{N}_{K/\mathbb{Q}}(p^3a+p^2bk+\frac{\varepsilon p^2+f\sqrt{d_K}}{2}b)|=\frac{1}{4}|p^4(2pa+2bk+\varepsilon b)^2-f^2b^2d_K|$, and thus $\ell^2\equiv 4\beta p^3\mod q$ for some $\ell\in\mathbb{Z}$ and $\beta\in\{-1,1\}$. Since $q\equiv 1\mod 4$, we have $(\frac{-1}{q})=1$, and hence $(\frac{p}{q})^3=(\frac{4\beta p^3}{q})=1$. Therefore, $(\frac{p}{q})=1$, a contradiction.

Now let $p=2$. Then $I=8\mathbb{Z}+(2k+f\sqrt{d})\mathbb{Z}$ for some $k\in [0,3]$. Assume that $I$ is principal. Then there are some $a,b\in\mathbb{Z}$ such that $I=(8a+2bk+bf\sqrt{d})\mathcal{O}$. Consequently, $8=\mathcal{N}(I)=|(8a+2bk)^2-b^2f^2d|$, and thus $\ell^2\equiv 8\beta\mod q$ for some $\ell\in\mathbb{Z}$ and $\beta\in\{-1,1\}$. This implies that $(\frac{2}{q})^3=(\frac{8\beta}{q})=1$. Therefore, $(\frac{2}{q})=1$, a contradiction.
\end{proof}

\begin{proposition}\label{proposition 4.19} Let $d\in\mathbb{N}_{\geq 2}$ be squarefree, let $K=\mathbb{Q}(\sqrt{d})$, and let $\mathcal{O}$ be the order in $K$ with conductor $f\mathcal{O}_K$ such that $f$ is a nonempty squarefree product of ramified primes times a squarefree product of inert primes and $|{\rm Pic}(\mathcal{O})|=|{\rm Pic}(\mathcal{O}_K)|=2$. If for every ramified prime divisor $p$ of $f$, we have $(p\equiv 1\mod 4$ and $(\frac{d/p}{p})=-1)$ or $((\frac{p}{q})=-1$ for some prime $q$ with $q\equiv 1\mod 4$ and $q\mid df)$, then $\min\Delta(\mathcal{O})=2$.
\end{proposition}

\begin{proof} It follows by Lemma~\ref{lemma 4.18} that for every ramified prime divisor $p$ of $f$ and every $I\in\mathcal{A}(\mathcal{I}^*_p(\mathcal{O}))$ with $\mathcal{N}(I)=p^3$, we have $I$ is not principal. It follows by the claim in the proof of Theorem~\ref{theorem 4.14} that $I\in\mathcal{A}(\mathcal{I}^*_p(\mathcal{O}))$ is principal if and only if $\mathcal{N}(I)=p^2$. Now let $p$ be an inert prime divisor of $f$ and let $J\in\mathcal{A}(\mathcal{I}^*_p(\mathcal{O}))$. Since $|{\rm Pic}(\mathcal{O})|=|{\rm Pic}(\mathcal{O}_K)|$, it follows that the group epimorphism $\theta:{\rm Pic}(\mathcal{O})\rightarrow {\rm Pic}(\mathcal{O}_K)$ defined by $\theta([L])=[L\mathcal{O}_K]$ for all $L\in\mathcal{I}^*(\mathcal{O})$ is a group isomorphism. Set $P=p\mathcal{O}_K$. Then $J\mathcal{O}_K$ is a $P$-primary ideal of $\mathcal{O}_K$, and hence $J\mathcal{O}_K$ is a principal ideal of $\mathcal{O}_K$. Since $\theta$ is an isomorphism, we infer that $J$ is a principal ideal of $\mathcal{O}$. Now it follows by Theorem~\ref{theorem 4.14} that $\min\Delta(\mathcal{O})=2$.
\end{proof}

Next we provide two counterexamples that show that the additional assumption on the ramified prime divisors of $f$ in Proposition~\ref{proposition 4.19} is important.

\begin{example}\label{example 4.20} There is some real quadratic number field $K$ and some order $\mathcal{O}$ in $K$ with conductor $p\mathcal{O}_K$ for some ramified prime $p$ such that $p\equiv 1\mod 4$, $|{\rm Pic}(\mathcal{O})|=|{\rm Pic}(\mathcal{O}_K)|=2$, and $\min\Delta(\mathcal{O})=1$.
\end{example}

\begin{proof} Let $\mathcal{O}=\mathbb{Z}+5\sqrt{30}\mathbb{Z}$ be the order in the real quadratic number field $K=\mathbb{Q}(\sqrt{30})$ with conductor $5\mathcal{O}_K$. Observe that $5$ is ramified, $5\equiv 1\mod 4$, $|{\rm Pic}(\mathcal{O}_K)|=2$ and $\alpha=11+2\sqrt{30}$ is a fundamental unit of $\mathcal{O}_K$. Since $\alpha\not\in\mathcal{O}$ and $(\mathcal{O}_K^{\times}:\mathcal{O}^{\times})\mid 5$, we infer that $(\mathcal{O}_K^{\times}:\mathcal{O}^{\times})=5$, and hence $|{\rm Pic}(\mathcal{O})|=|{\rm Pic}(\mathcal{O}_K)|\frac{5}{(\mathcal{O}_K^{\times}:\mathcal{O}^{\times})}=2$. Let $I=125\mathbb{Z}+5\sqrt{30}\mathbb{Z}$. Then $I\in\mathcal{A}(\mathcal{I}^*_5(\mathcal{O}))$ with $\mathcal{N}(I)=125$. Since $I=(12625+2305\sqrt{30})\mathcal{O}$ is principal, we infer by Theorem~\ref{theorem 4.14} that $\min\Delta(\mathcal{O})=1$.
\end{proof}

\begin{example}\label{example 4.21} There is some real quadratic number field $K=\mathbb{Q}(\sqrt{d})$ with $d\in\mathbb{N}_{\geq 2}$ squarefree and some order $\mathcal{O}$ in $K$ with conductor $p\mathcal{O}_K$ for some odd ramified prime $p$ such that $(\frac{d/p}{p})=-1$, $|{\rm Pic}(\mathcal{O})|=|{\rm Pic}(\mathcal{O}_K)|=2$, and $\min\Delta(\mathcal{O})=1$.
\end{example}

\begin{proof} Let $\mathcal{O}=\mathbb{Z}+7\sqrt{42}\mathbb{Z}$ be the order in the real quadratic number field $K=\mathbb{Q}(\sqrt{42})$ with conductor $7\mathcal{O}_K$. Note that $7$ is an odd ramified prime, $(\frac{42/7}{7})=-1$, $|{\rm Pic}(\mathcal{O}_K)|=2$ and $\alpha=13+2\sqrt{42}$ is a fundamental unit of $\mathcal{O}_K$. We have $\alpha\not\in\mathcal{O}$ and $(\mathcal{O}_K^{\times}:\mathcal{O}^{\times})\mid 7$. Therefore, $(\mathcal{O}_K^{\times}:\mathcal{O}^{\times})=7$, and thus $|{\rm Pic}(\mathcal{O})|=|{\rm Pic}(\mathcal{O}_K)|\frac{7}{(\mathcal{O}_K^{\times}:\mathcal{O}^{\times})}=2$. Set $I=343\mathbb{Z}+7\sqrt{42}\mathbb{Z}$. Then $I\in\mathcal{A}(\mathcal{I}^*_7(\mathcal{O}))$, $\mathcal{N}(I)=343$, and $I=(825601+127393\sqrt{42})\mathcal{O}$ is principal. Consequently, $\min\Delta(\mathcal{O})=1$ by Theorem~\ref{theorem 4.14}.
\end{proof}

Finally, we provide the examples of orders $\mathcal{O}$ in quadratic number fields with $\min\Delta(\mathcal{O})=2$.

\begin{example}\label{example 4.22} Let $K$ be a quadratic number field and $\mathcal{O}$ the order in $K$ with conductor $f\mathcal{O}_K$ such that $(f,d_K)\in\{(2,60),(3,60),(5,60),(6,60),(10,60),(15,60),(30,60),(10,85),(35,40),(195,65),(30,365)\}$.
\begin{enumerate}
\item[\textnormal{1.}] If $(f,d_K)\in\{(2,60),(3,60),(5,60)\}$, then $f$ is a ramified prime.
\item[\textnormal{2.}] If $(f,d_K)\in\{(6,60),(10,60),(15,60)\}$, then $f$ is the product of two distinct ramified primes.
\item[\textnormal{3.}] If $(f,d_K)=(30,60)$, then $f$ is the product of three distinct ramified primes.
\item[\textnormal{4.}] If $(f,d_K)\in\{(10,85),(35,40)\}$, then $f$ is the product of an inert prime and a ramified prime.
\item[\textnormal{5.}] If $(f,d_K)=(195,65)$, then $f$ is the product of an inert prime and two distinct ramified primes.
\item[\textnormal{6.}] If $(f,d_K)=(30,365)$, then $f$ is the product of two distinct inert primes and a ramified prime.
\item[\textnormal{7.}] $\min\Delta(\mathcal{O})=2$.
\end{enumerate}
\end{example}

\begin{proof} It is straightforward to prove the first six assertions. We prove the last assertion in the case that $d_K=60$ and $f\in\mathbb{N}_{\geq 2}$ is a divisor of $30$. The remaining cases can be proved in analogy by using Proposition~\ref{proposition 4.19}. It is clear that $2$, $3$, and $5$ are ramified primes. Note that $|{\rm Pic}(\mathcal{O}_K)|=2$ (e.g., \cite[page 22]{HK13a}) and $\alpha=4+\sqrt{15}$ is a fundamental unit of $\mathcal{O}_K$.

We have $\alpha^2=31+8\sqrt{15}$, $\alpha^3=244+63\sqrt{15}$, and $\alpha^5=15124+3905\sqrt{15}$. Moreover, $\alpha^6=119071+30744\sqrt{15}$, $\alpha^{10}=457470751+118118440\sqrt{15}$, and $\alpha^{15}=13837575261124+3572846569215\sqrt{15}$. Set $k=(\mathcal{O}_K^{\times}:\mathcal{O}^{\times})$. Then $k$ is a divisor of $f$ by \eqref{equation 6}. Observe that $\alpha\not\in\mathbb{Z}+2\sqrt{15}\mathbb{Z}$, $\alpha\not\in\mathbb{Z}+3\sqrt{15}\mathbb{Z}$, $\alpha\not\in\mathbb{Z}+5\sqrt{15}\mathbb{Z}$, $\alpha^2,\alpha^3\not\in\mathbb{Z}+6\sqrt{15}\mathbb{Z}$, $\alpha^2,\alpha^5\not\in\mathbb{Z}+10\sqrt{15}\mathbb{Z}$, $\alpha^3,\alpha^5\not\in\mathbb{Z}+15\sqrt{15}\mathbb{Z}$, and $\alpha^6,\alpha^{10},\alpha^{15}\not\in\mathbb{Z}+30\sqrt{15}\mathbb{Z}$. This implies that $k=f$, and hence $|{\rm Pic}(\mathcal{O})|=\frac{f}{k}|{\rm Pic}(\mathcal{O}_K)|=|{\rm Pic}(\mathcal{O}_K)|=2$ by \eqref{equation 6}. We have $5\equiv 1\mod 4$ and $(\frac{15/5}{5})=(\frac{3}{5})=(\frac{2}{5})=-1$. We infer by Proposition~\ref{proposition 4.19} that $\min\Delta(\mathcal{O})=2$.
\end{proof}

\smallskip
\section{Unions of sets of lengths}\label{5}
\smallskip

The goal of this section is to show that all unions of sets of lengths of the monoid of (invertible) ideals in orders of quadratic number fields are intervals (Theorem~\ref{theorem 5.2}).
To gather the background on unions of sets of lengths, let $H$ be an atomic monoid with $H\ne H^{\times}$ and $k\in\mathbb{N}_0$. Then
\[
\begin{aligned}
\mathcal{U}_k (H) & =\bigcup_{k\in L\in\mathcal{L}(H)} L \qquad\text{denotes the {\it union of sets of lengths} containing $k$ and } \\
\rho_k (H) & = \sup \mathcal U_k (H) \qquad \text{is the {\it $k$th elasticity} of $H$} \,.
\end{aligned}
\]
Then, for the {\it elasticity} $\rho (H)$ of $H$, we have (\cite[Proposition 2.7]{F-G-K-T17}),
\[
\rho(H)=\sup\{\rho (L) \mid L\in\mathcal{L}(H)\} = \lim_{k\to\infty}\frac{\rho_k (H)}{k}\,.
\]
Clearly, $\mathcal{U}_0(H)=\{0\}$, $\mathcal{U}_1(H)=\{1\}$ and $\mathcal{U}_k(H)$ is the set of all $\ell\in\mathbb{N}_0$ with the following property:
\begin{itemize}
\item[] There are atoms $u_1,\ldots,u_k,v_1,\ldots,v_{\ell}$ in $H$ such that $u_1\cdot\ldots\cdot u_k= v_1\cdot\ldots\cdot v_{\ell}$.
\end{itemize}
Let $d\in\mathbb{N}$ and $M\in\mathbb{N}_0$. A subset $L\subset\mathbb{Z}$ is called an AAP (with difference $d$ and bound $M$) if
\[
L=y+\big( L'\cup L^*\cup L''\big)\subset y+d\mathbb{Z}\,,
\]
where $y\in\mathbb{Z}$, $L^*$ is a non-empty arithmetical progression with difference $d$ and $\min L^*=0$, $L'\subset [-M,-1]$, and $L''\subset\sup L^*+[1,M]$ (with the convention that $L''=\emptyset$ if $L^*$ is infinite).
We say that $H$ satisfies the {\it Structure Theorem for Unions} if there are $d\in\mathbb{N}$ and $M\in\mathbb{N}_0$ such that $\mathcal U_k (H)$ is an AAP with difference $d$ and bound $M$ for all sufficiently large $k\in\mathbb{N}$. If $\Delta (H)$ is finite and the Structure Theorem for Unions holds for some parameter $d\in\mathbb{N}$, then $d = \min \Delta (H)$ (\cite[Lemma 2.12]{F-G-K-T17}).

The Structure Theorem for Unions holds for a wealth of monoids and domains (see \cite{Ba-Sm18,Fa-Tr18a,Tr18a} for recent contributions and see \cite[Theorem 4.2]{F-G-K-T17} for an example where it does not hold). Since it holds for C-monoids (\cite{Ga-Ge09b}), it holds for the monoid of invertible ideals of orders in number fields. In some special cases (including Krull monoids having prime divisors in all classes) all unions of sets of lengths are intervals, in other words the Structure Theorem for Unions holds with $d=1$ and $M=0$ (\cite[Theorem 3.1.3]{Ge09a}, \cite[Theorem 5.8]{Ge-Ka-Re15a}, \cite{Sm13a}). In Theorem~\ref{theorem 5.2} we show that the same is true for the monoids of (invertible) ideals of orders in quadratic number fields.

\medskip
\begin{proposition}\label{proposition 5.1}
Let $p$ be a prime divisor of $f$ and let $N=\sup\{{\rm v}_p(\mathcal{N}(A))\mid A\in\mathcal{A}(\mathcal{I}^*_p(\mathcal{O}_f))\}$.
\begin{enumerate}
\item[\textnormal{1.}] If $p$ splits, then $\mathcal{U}_{\ell}(\mathcal{I}_p(\mathcal{O}_f))=\mathcal{U}_{\ell}(\mathcal{I}^*_p(\mathcal{O}_f))=\mathbb{N}_{\geq 2}$ for all $\ell\in\mathbb{N}_{\geq 2}$.

\item[\textnormal{2.}] If $p$ does not split, then $\mathcal{U}_{\ell}(\mathcal{I}_p(\mathcal{O}_f))\cap\mathbb{N}_{\geq\ell}=\mathcal{U}_{\ell}(\mathcal{I}^*_p(\mathcal{O}_f))\cap\mathbb{N}_{\geq\ell}=[\ell,\lfloor\frac{\ell N}{2}\rfloor]$ for all $\ell\in\mathbb{N}_{\geq 2}$.
\end{enumerate}
\end{proposition}

\begin{proof}
We prove 1. and 2. simultaneously. By Proposition~\ref{proposition 3.3}.3 we can assume without restriction that $f=p^{{\rm v}_p(f)}$. First we show that both assertions are true for $\ell=2$. It follows from Theorem~\ref{theorem 3.6} that $[2, N]=[2,2{\rm v}_p(f)]\cup\{{\rm v}_p(\mathcal{N}(A))\mid A\in\mathcal{A}(\mathcal{I}^*_p(\mathcal{O}_f))\}$. It is obvious that $\mathcal{U}_2(\mathcal{I}^*_p(\mathcal{O}_f))\subset\mathcal{U}_2(\mathcal{I}_p(\mathcal{O}_f))$. It follows from Lemma~\ref{lemma 4.9} that $\mathcal{U}_2(\mathcal{I}_p(\mathcal{O}_f))\subset [2, N]$.

Let $k\in [2, N]$. It remains to show that $k\in\mathcal{U}_2(\mathcal{I}^*_p(\mathcal{O}_f))$. If $k>2{\rm v}_p(f)$, then there is some $I\in\mathcal{A}(\mathcal{I}^*_p(\mathcal{O}_f))$ such that $\mathcal{N}(I)=p^k$. It follows by Proposition~\ref{proposition 3.2}.5 that $I\overline{I}=(p\mathcal{O}_f)^k$, and hence $k\in\mathcal{U}_2(\mathcal{I}^*_p(\mathcal{O}_f))$. Now let $k\leq 2{\rm v}_p(f)$. By Proposition~\ref{proposition 4.8}.1 we can assume without restriction that ${\rm v}_p(f)\geq 2$ and $k\geq 4$.

\smallskip
CASE 1: $d\not\equiv 1\mod 4$ or $(d\equiv 1\mod 4$, $p=2$ and $k\leq 2({\rm v}_2(f)-1))$. We set $a={\rm v}_p(\mathcal{N}_{K/\mathbb{Q}}(p^{k-2}+\tau))$ and $b={\rm v}_p(\mathcal{N}_{K/\mathbb{Q}}(p^{k-2}(p-1)+\tau))$. Observe that if $d\not\equiv 1\mod 4$, then $a,b\geq\min\{2k-4,2{\rm v}_p(f)\}\geq k$. Moreover, if $d\equiv 1\mod 4$, $p=2$ and $k\leq 2({\rm v}_2(f)-1)$, then $a,b\geq\min\{2k-4,2({\rm v}_2(f)-1)\}\geq k$. Set $I=p^a\mathbb{Z}+(p^{k-2}+\tau)\mathbb{Z}$ and $J=p^b\mathbb{Z}+(p^{k-2}(p-1)+\tau)\mathbb{Z}$. Then $I,J\in\mathcal{A}(\mathcal{I}^*_p(\mathcal{O}_f))$, $\min\{a,b,{\rm v}_p(p^{k-2}+p^{k-2}(p-1)+\varepsilon)\}=k-1$, and $a+b-2(k-1)>0$. Therefore, there is some $L\in\mathcal{A}(\mathcal{I}^*_p(\mathcal{O}_f))$ such that $IJ=p^{k-1}L$, and hence $k\in\mathsf{L}(IJ)\subset\mathcal{U}_2(\mathcal{I}^*_p(\mathcal{O}_f))$.

\smallskip
CASE 2: $d\equiv 1\mod 4$ and $p\not=2$. We set $a={\rm v}_p(\mathcal{N}_{K/\mathbb{Q}}(\frac{p^{k-2}-1}{2}+\tau))$ and $b={\rm v}_p(\mathcal{N}_{K/\mathbb{Q}}(\frac{p^{k-2}(p^2+p-1)-1}{2}+\tau))$. Note that $a,b\geq\min\{2k-4,2{\rm v}_p(f)\}\geq k$. Set $I=p^a\mathbb{Z}+(\frac{p^{k-2}-1}{2}+\tau)\mathbb{Z}$ and $J=p^b\mathbb{Z}+(\frac{p^{k-2}(p^2+p-1)-1}{2}+\tau)\mathbb{Z}$. Then $I,J\in\mathcal{A}(\mathcal{I}^*_p(\mathcal{O}_f))$, $\min\{a,b,{\rm v}_p(\frac{p^{k-2}-1}{2}+\frac{p^{k-2}(p^2+p-1)-1}{2}+\varepsilon)\}=k-1$, and $a+b-2(k-1)>0$. Consequently, there is some $L\in\mathcal{A}(\mathcal{I}^*_p(\mathcal{O}_f))$ such that $IJ=p^{k-1}L$, and thus $k\in\mathsf{L}(IJ)\subset\mathcal{U}_2(\mathcal{I}^*_p(\mathcal{O}_f))$.

\smallskip
CASE 3: $d\equiv 1\mod 8$, $p=2$ and $k\in\{2{\rm v}_2(f)-1,2{\rm v}_2(f)\}$. Set $h={\rm v}_2(f)$. If $h=2$, then $k=4$, and hence $k\in\mathcal{U}_2(\mathcal{I}^*_2(\mathcal{O}_f))$ by Proposition~\ref{proposition 4.4}. Now let $h\geq 3$. Note that $2$ splits. By Theorem~\ref{theorem 3.6} there are some $I,J,L\in\mathcal{A}(\mathcal{I}^*_2(\mathcal{O}_f))$ such that $\mathcal{N}(I)=2^{2h+1}$, $\mathcal{N}(J)=2^{2h+2}$ and $\mathcal{N}(L)=16$. By Proposition~\ref{proposition 3.2}.5 we have $L\overline{L}=16\mathcal{O}_f$, $I\overline{I}=2^{2h+1}\mathcal{O}_f=2^{2h-3}L\overline{L}$ and $J\overline{J}=2^{2h+2}\mathcal{O}_f=2^{2h-2}L\overline{L}$. We infer that $k\in\{2h-1,2h\}\subset\mathcal{U}_2(\mathcal{I}^*_2(\mathcal{O}_f))$.

\smallskip
CASE 4: $d\equiv 5\mod 8$, $p=2$ and $k\in\{2{\rm v}_2(f)-1,2{\rm v}_2(f)\}$. Set $h={\rm v}_2(f)$. If $h=2$, then $k=4$, and thus $k\in\mathcal{U}_2(\mathcal{I}^*_2(\mathcal{O}_f))$ by Proposition~\ref{proposition 4.4}. Now let $h\geq 3$. Set $A=2^{2h}\mathbb{Z}+(2^{h-1}+\tau)\mathbb{Z}$, $B=2^{2h}\mathbb{Z}+(2^{2h-2}-2^{h-1}+\tau)\mathbb{Z}$, and $C=2^{2h}\mathbb{Z}+(2^{2h-1}-2^{h-1}+\tau)\mathbb{Z}$. Then $A,B,C\in\mathcal{A}(\mathcal{I}^*_2(\mathcal{O}_f))$, $AB=2^{2h-2}I$ and $AC=2^{2h-1}J$ for some $I,J\in\mathcal{A}(\mathcal{I}^*_2(\mathcal{O}_f))$. Therefore, $k\in\{2h-1,2h\}\subset\mathcal{U}_2(\mathcal{I}^*_2(\mathcal{O}_f))$.

\smallskip
So far we have proved that both assertions are true for $\ell=2$. If $p$ splits, then we have $N=\infty$ by Theorem~\ref{theorem 3.6}, and hence $\mathcal{U}_2(\mathcal{I}_p(\mathcal{O}_f))=\mathcal{U}_2(\mathcal{I}^*_p(\mathcal{O}_f))=\mathbb{N}_{\geq 2}$. The first assertion now follows easily by induction on $\ell$. Now let $p$ not split. Then $N<\infty$. Next we show that 2. is true for $\ell=3$.

Since $[3,N+1]=\{1\}+\mathcal{U}_2(\mathcal{I}^*_p(\mathcal{O}_f))\subset\mathcal{U}_3(\mathcal{I}^*_p(\mathcal{O}_f))\cap\mathbb{N}_{\geq 3}\subset\mathcal{U}_3(\mathcal{I}_p(\mathcal{O}_f))\cap\mathbb{N}_{\geq 3}\subset [3,\lfloor\frac{3N}{2}\rfloor]$ by Lemma~\ref{lemma 4.9} and $N\in\{2{\rm v}_p(f),2{\rm v}_p(f)+1\}$, it remains to show that $N+m\in\mathcal{U}_3(\mathcal{I}^*_p(\mathcal{O}_f))$ for all $m\in [2,{\rm v}_p(f)]$. Let $m\in [2,{\rm v}_p(f)]$. It is sufficient to show that there are some $I,J,L\in\mathcal{A}(\mathcal{I}^*_p(\mathcal{O}_f))$ such that $IJ=p^mL$ and $\mathcal{N}(L)=p^N$, since then $IJ\overline{L}=p^{N+m}\mathcal{O}_f$ by Proposition~\ref{proposition 3.2}.5, and thus $N+m\in\mathcal{U}_3(\mathcal{I}^*_p(\mathcal{O}_f))$.

\smallskip
CASE 1: $p$ is inert. Observe that $N=2{\rm v}_p(f)$ by Theorem~\ref{theorem 3.6}. Let $m\in [2,{\rm v}_p(f)]$. First let $p\not=2$. If $d\not\equiv 1\mod 4$, then set $I=p^{2m}\mathbb{Z}+(p^m+\tau)\mathbb{Z}$ and $J=p^{2{\rm v}_p(f)}\mathbb{Z}+(p^{2{\rm v}_p(f)-m}+\tau)\mathbb{Z}$. If $d\equiv 1\mod 4$, then set $I=p^{2m}\mathbb{Z}+(\frac{p^m-1}{2}+\tau)\mathbb{Z}$ and $J=p^{2{\rm v}_p(f)}\mathbb{Z}+(\frac{p^{2{\rm v}_p(f)-m}-1}{2}+\tau)\mathbb{Z}$. In any case we have $I,J\in\mathcal{A}(\mathcal{I}^*_p(\mathcal{O}_f))$ and $IJ=p^mL$ for some $L\in\mathcal{A}(\mathcal{I}^*_p(\mathcal{O}_f))$ with $\mathcal{N}(L)=p^N$.

Next let $p=2$. Since $2$ is inert, it follows that $d\equiv 5\mod 8$. If $m<{\rm v}_2(f)-1$, then set $I=2^{2m}\mathbb{Z}+(2^m+\tau)\mathbb{Z}$. If $m={\rm v}_2(f)-1$, then set $I=2^{2m}\mathbb{Z}+\tau\mathbb{Z}$. Finally, if $m={\rm v}_2(f)$, then set $I=2^{2m}\mathbb{Z}+(2^{m-1}+\tau)\mathbb{Z}$. Set $J=2^{2{\rm v}_2(f)}\mathbb{Z}+(2^{{\rm v}_2(f)-1}+\tau)\mathbb{Z}$. Observe that $I,J\in\mathcal{A}(\mathcal{I}^*_2(\mathcal{O}_f))$ and $IJ=2^mL$ for some $L\in\mathcal{A}(\mathcal{I}^*_2(\mathcal{O}_f))$ with $\mathcal{N}(L)=2^N$.

\smallskip
CASE 2: $p$ is ramified. It follows that $N=2{\rm v}_p(f)+1$ by Theorem~\ref{theorem 3.6}. Let $m\in [2,{\rm v}_p(f)]$. First let $p\not=2$. Since $p$ is ramified, we have $p\mid d$. If $d\not\equiv 1\mod 4$, then set $I=p^{2m}\mathbb{Z}+(p^m+\tau)\mathbb{Z}$ and $J=p^{2{\rm v}_p(f)+1}\mathbb{Z}+(p^{{\rm v}_p(f)+1}+\tau)\mathbb{Z}$. If $d\equiv 1\mod 4$, then set $I=p^{2m}\mathbb{Z}+(\frac{p^m-1}{2}+\tau)\mathbb{Z}$ and $J=p^{2{\rm v}_p(f)+1}\mathbb{Z}+(\frac{p^{{\rm v}_p(f)+1}-1}{2}+\tau)\mathbb{Z}$. We infer that $I,J\in\mathcal{A}(\mathcal{I}^*_p(\mathcal{O}_f))$ and $IJ=p^mL$ for some $L\in\mathcal{A}(\mathcal{I}^*_p(\mathcal{O}_f))$ with $\mathcal{N}(L)=p^N$ in any case.

Now let $p=2$. Since $2$ is ramified, we have $d\not\equiv 1\mod 4$. If $d$ is even or $m<{\rm v}_2(f)$, then set $I=2^{2m}\mathbb{Z}+(2^m+\tau)\mathbb{Z}$. If $d$ is odd and $m={\rm v}_2(f)$, then set $I=2^{2m}\mathbb{Z}+\tau\mathbb{Z}$. If $d$ is even, then set $J=2^{2{\rm v}_2(f)+1}\mathbb{Z}+\tau\mathbb{Z}$. If $d$ is odd, then set $J=2^{2{\rm v}_2(f)+1}\mathbb{Z}+(2^{{\rm v}_2(f)}+\tau)\mathbb{Z}$. In any case we have $I,J\in\mathcal{A}(\mathcal{I}^*_2(\mathcal{O}_f))$ and $IJ=2^mL$ for some $L\in\mathcal{A}(\mathcal{I}^*_2(\mathcal{O}_f))$ with $\mathcal{N}(L)=2^N$.

\smallskip
Finally, we prove the second assertion by induction on $\ell$. Let $\ell\in\mathbb{N}_{\geq 2}$ and let $H\in\{\mathcal{I}_p(\mathcal{O}_f),\mathcal{I}^*_p(\mathcal{O}_f)\}$. Without restriction we can assume that $\ell\geq 4$. We infer by the induction hypothesis that $(\mathcal{U}_{\ell-2}(H)\cap\mathbb{N}_{\geq\ell-2})+\mathcal{U}_2(H)=[\ell-2,\lfloor\frac{(\ell-2)N}{2}\rfloor]+[2,N]=[\ell,\lfloor\frac{\ell N}{2}\rfloor]$. Observe that $(\mathcal{U}_{\ell-2}(H)\cap\mathbb{N}_{\geq\ell-2})+\mathcal{U}_2(H)\subset\mathcal{U}_{\ell}(H)\cap\mathbb{N}_{\geq\ell}$. It follows by Lemma~\ref{lemma 4.9} that $\mathcal{U}_{\ell}(H)\cap\mathbb{N}_{\geq\ell}\subset [\ell,\lfloor\frac{\ell N}{2}\rfloor]$, and thus $\mathcal{U}_{\ell}(H)\cap\mathbb{N}_{\geq\ell}=[\ell,\lfloor\frac{\ell N}{2}\rfloor]$.
\end{proof}

\smallskip
\begin{theorem}\label{theorem 5.2}
Let $\mathcal{O}$ be an order in a quadratic number field $K$ with conductor $f\mathcal{O}_K$ for some $f\in\mathbb{N}_{\geq 2}$.
\begin{enumerate}
\item[\textnormal{1.}] If $f$ is divisible by a split prime, then $\mathcal{U}_k(\mathcal{I}(\mathcal{O}))=\mathcal{U}_k(\mathcal{I}^*(\mathcal{O}))=\mathbb{N}_{\geq 2}$ for all $k\in\mathbb{N}_{\geq 2}$.

\item[\textnormal{2.}] Suppose that $f$ is not divisible by a split prime and set $M=\max\{{\rm v}_p(f)\mid p\in\mathbb{P}\}$. Then $\mathcal{U}_k(\mathcal{I}(\mathcal{O}))=\mathcal{U}_k(\mathcal{I}^*(\mathcal{O}))$ is a finite interval for all $k\in\mathbb{N}_{\geq 2}$, and for their maxima we have{\rm \,:}
\begin{enumerate}
\item[\textnormal{(a)}] If ${\rm v}_q(f)=M$ for a ramified prime $q$, then $\rho_k(\mathcal{I}(\mathcal{O}))=\rho_k(\mathcal{I}^*(\mathcal{O}))=kM+\lfloor\frac{k}{2}\rfloor$ for all $k\in\mathbb{N}_{\geq 2}$ and $\rho(\mathcal{I}(\mathcal{O}))=\rho(\mathcal{I}^*(\mathcal{O}))=M+\frac{1}{2}$.

\item[\textnormal{(b)}] If ${\rm v}_q(f)<M$ for all ramified primes $q$, then $\rho_k(\mathcal{I}(\mathcal{O}))=\rho_k(\mathcal{I}^*(\mathcal{O}))=kM$ for all $k\in\mathbb{N}_{\geq 2}$ and $\rho(\mathcal{I}(\mathcal{O}))=\rho(\mathcal{I}^*(\mathcal{O}))=M$.
\end{enumerate}
\end{enumerate}
\end{theorem}

\begin{proof}
1. Let $f$ be divisible by a split prime $p$ and let $k\in\mathbb{N}_{\geq 2}$. Since $\mathcal{I}^*_p(\mathcal{O})$ is a divisor-closed submonoid of $\mathcal{I}^*(\mathcal{O})$ and $\mathcal{I}_p(\mathcal{O})$ is a divisor-closed submonoid of $\mathcal{I}(\mathcal{O})$, it follows from Proposition~\ref{proposition 5.1}.1 that $\mathcal{U}_k(\mathcal{I}(\mathcal{O}))=\mathcal{U}_k(\mathcal{I}^*(\mathcal{O}))=\mathbb{N}_{\geq 2}$.

\smallskip
2. Let $k\in\mathbb{N}_{\geq 2}$ and $\ell\in\mathcal{U}_k(\mathcal{I}(\mathcal{O}))$. There are $I_i\in\mathcal{A}(\mathcal{I}(\mathcal{O}))$ for each $i\in [1,k]$ and $J_j\in\mathcal{A}(\mathcal{I}(\mathcal{O}))$ for each $j\in [1,\ell]$ such that $\prod_{i=1}^k I_i=\prod_{j=1}^{\ell} J_j$. Note that $\sqrt{I_i},\sqrt{J_j}\in\mathfrak{X}(\mathcal{O})$ for all $i\in [1,k]$ and $j\in [1,\ell]$. For $P\in\mathfrak{X}(\mathcal{O})$ set $k_P=|\{i\in [1,k]\mid\sqrt{I_i}=P\}|$ and $\ell_P=|\{j\in [1,\ell]\mid\sqrt{J_j}=P\}|$. If $p$ is a prime divisor of $f$, then set $k_p=k_{P_{f,p}}$ and $\ell_p=\ell_{P_{f,p}}$. Observe that $k=\sum_{P\in\mathfrak{X}(\mathcal{O})} k_P$ and $\ell=\sum_{P\in\mathfrak{X}(\mathcal{O})}\ell_P$. Recall that the $P$-primary components of $\prod_{i=1}^k I_i$ are uniquely determined, and thus $\ell_P\in\mathcal{U}_{k_P}(\mathcal{I}_P(\mathcal{O}))$ for all $P\in\mathfrak{X}(\mathcal{O})$. If $P\in\mathfrak{X}(\mathcal{O})$ does not contain the conductor, then $\mathcal{I}_P(\mathcal{O})$ is factorial, and hence $\ell_P=k_P$. Also note that if $P\in\mathfrak{X}(\mathcal{O})$ and $k_P\leq 1$, then $\ell_P=k_P$. If $p$ is an inert prime that divides $f$, then it follows from Proposition~\ref{proposition 5.1}.2 and Theorem~\ref{theorem 3.6} that $\rho_r(\mathcal{I}_p(\mathcal{O}))=\rho_r(\mathcal{I}^*_p(\mathcal{O}))=r{\rm v}_p(f)$ for all $r\in\mathbb{N}_{\geq 2}$. We infer again by Proposition~\ref{proposition 5.1}.2 and Theorem~\ref{theorem 3.6} that $\rho_r(\mathcal{I}_p(\mathcal{O}))=\rho_r(\mathcal{I}^*_p(\mathcal{O}))=r{\rm v}_p(f)+\lfloor\frac{r}{2}\rfloor$ for all ramified primes $p$ that divide $f$ and all $r\in\mathbb{N}_{\geq 2}$.

\smallskip
CASE 1: ${\rm v}_q(f)=M$ for some ramified prime $q$. If $P\in\mathfrak{X}(\mathcal{O})$, then $\ell_P\leq k_PM+\lfloor\frac{k_P}{2}\rfloor$.

Consequently, $\ell=\sum_{P\in\mathfrak{X}(\mathcal{O})}\ell_P\leq (\sum_{P\in\mathfrak{X}(\mathcal{O})} k_P)M+\sum_{P\in\mathfrak{X}(\mathcal{O})}\lfloor\frac{k_P}{2}\rfloor\leq kM+\lfloor\frac{k}{2}\rfloor$. In particular, $\rho_k(\mathcal{I}(\mathcal{O}))\leq kM+\lfloor\frac{k}{2}\rfloor=\max\{\rho_k(\mathcal{I}^*_p(\mathcal{O}))\mid p\in\mathbb{P},p\mid f\}\leq\rho_k(\mathcal{I}^*(\mathcal{O}))\leq\rho_k(\mathcal{I}(\mathcal{O}))$. This implies that $\rho_k(\mathcal{I}(\mathcal{O}))=\rho_k(\mathcal{I}^*(\mathcal{O}))=\max\{\rho_k(\mathcal{I}^*_p(\mathcal{O}))\mid p\in\mathbb{P},p\mid f\}=kM+\lfloor\frac{k}{2}\rfloor$.

\smallskip
CASE 2: ${\rm v}_q(f)<M$ for all ramified primes $q$. Note that $\ell_p\leq k_p {\rm v}_p(f)+\lfloor\frac{k_p}{2}\rfloor\leq k_pM$ for all ramified primes $p$ that divide $f$. Therefore, $\ell_P\leq k_PM$ for all $P\in\mathfrak{X}(\mathcal{O})$. This implies that $\ell=\sum_{P\in\mathfrak{X}(\mathcal{O})}\ell_P\leq (\sum_{P\in\mathfrak{X}(\mathcal{O})} k_P)M=kM$. We infer that $\rho_k(\mathcal{I}(\mathcal{O}))\leq kM=\max\{\rho_k(\mathcal{I}^*_p(\mathcal{O}))\mid p\in\mathbb{P},p\mid f\}\leq\rho_k(\mathcal{I}^*(\mathcal{O}))\leq\rho_k(\mathcal{I}(\mathcal{O}))$, and thus $\rho_k(\mathcal{I}(\mathcal{O}))=\rho_k(\mathcal{I}^*(\mathcal{O}))=\max\{\rho_k(\mathcal{I}^*_p(\mathcal{O}))\mid p\in\mathbb{P},p\mid f\}=kM$.

\smallskip
By Proposition~\ref{proposition 5.1}.2, we obtain that $\mathcal{U}_k(\mathcal{I}(\mathcal{O}))\cap\mathbb{N}_{\geq k}=\mathcal{U}_k(\mathcal{I}^*(\mathcal{O}))\cap\mathbb{N}_{\geq k}$ is a finite interval. Since the last assertion holds for every $k\in\mathbb{N}_{\geq 2}$, we infer that $\mathcal{U}_k(\mathcal{I}(\mathcal{O}))=\mathcal{U}_k(\mathcal{I}^*(\mathcal{O}))$ is a finite interval for all $k\in\mathbb{N}_{\geq 2}$. If ${\rm v}_q(f)=M$ for some ramified prime $q$, then
\[
\rho(\mathcal{I}(\mathcal{O}))=\rho(\mathcal{I}^*(\mathcal{O}))=\lim_{k\rightarrow\infty}\frac{\rho_k(\mathcal{I}(\mathcal{O}))}{k}=\lim_{k\rightarrow\infty} M+\frac{1}{k}\left\lfloor\frac{k}{2}\right\rfloor=M+\frac{1}{2}.
\]
Finally, let ${\rm v}_q(f)<M$ for all ramified primes $q$. Then
\[
\rho(\mathcal{I}(\mathcal{O}))=\rho(\mathcal{I}^*(\mathcal{O}))=\lim_{k\rightarrow\infty}\frac{\rho_k(\mathcal{I}(\mathcal{O}))}{k}=\lim_{k\rightarrow\infty}\frac{kM}{k}=M. \qedhere\]
\end{proof}

In a final remark we gather what is known on further arithmetical invariants of monoids of ideals of orders in quadratic number fields.

\smallskip
\begin{remark}\label{remark 5.3}
Let $\mathcal{O}$ be an order in a quadratic number field $K$ with conductor $f\mathcal{O}_K$ for some $f\in\mathbb{N}_{\geq 2}$.

1. The monotone catenary degree of $\mathcal I^*(\mathcal{O})$ is finite by \cite[Corollary 5.14]{Ge-Re18d}. Precise values for the monotone catenary degree are available so far only in the seminormal case (\cite[Theorem 5.8]{Ge-Ka-Re15a}).

2. The tame degree of $\mathcal I^*(\mathcal{O})$ is finite if and only if the elasticity is finite if and only if $f$ is not divisible by a split prime. This follows from Equations~\ref{equation 3} and~\ref{equation 4}, Theorem~\ref{theorem 5.2}, and from \cite[Theorem 3.1.5]{Ge-HK06a}. Precise values for the tame degree are not known so far.

3. For an atomic monoid $H$, the set $\{\rho(L)\mid L\in\mathcal{L}(H)\}\subset\mathbb{Q}_{\ge 1}$ of all elasticities was first studied by Chapman et al. and then it found further attention by several authors (e.g., \cite{Ba-Ne-Pe17a,Ch-Ho-Mo06}, \cite[Theorem 5.5]{Ge-Sc-Zh17b}, \cite{Ge-Zh19a,Zh19a}). We say that $H$ is {\it fully elastic} if for every rational number $q$ with $1 < q < \rho (H)$ there is an $L\in\mathcal{L}(H)$ with $\rho(L)=q$. Since $\mathcal I^*(\mathcal{O})$ is cancellative and has a prime element, it is fully elastic by \cite[Lemma 2.1]{B-C-C-K-W06}. Since $\mathcal I^*(\mathcal{O})\subset\mathcal I(\mathcal{O})$ is divisor-closed and $\rho(\mathcal{I}(\mathcal{O}))=\rho(\mathcal{I}^*(\mathcal{O}))$ by Theorem \ref{theorem 5.2}, it follows that $\mathcal I(\mathcal{O})$ is fully elastic.

4. For an atomic monoid $H$, let
\[
\daleth^*(H)=\{\min(L\setminus\{2\})\mid 2\in L\in\mathcal{L}(H)\ \text{with}\ |L|>1\}\subset\mathbb{N}_{\ge 3}\,.
\]
By definition, we have $\daleth^*(H)\subset 2+\Delta (H)$ and in \cite{Fa-Ge17a,Ge-Zh19a} the invariant $\daleth^*(H)$ was used as a tool to study $\Delta(H)$. Proposition~\ref{proposition 4.1}.4 shows that, both for $H=\mathcal I(\mathcal{O})$ and for $H=\mathcal I^*(\mathcal{O})$, we have $\max\daleth^*(H)=2+\max\Delta(H)$.
\end{remark}

\medskip
\noindent
{\bf Acknowledgements.} We would like to thank the referee for carefully reading the manuscript and for many suggestions and comments that improved the quality of this paper and simplified the proof of Theorem~\ref{theorem 3.6}.

\providecommand{\bysame}{\leavevmode\hbox to3em{\hrulefill}\thinspace}

\end{document}